\documentclass[12pt,oneside,reqno,english]{amsart}
\usepackage[foot]{amsaddr}
\usepackage{lmodern} 
\usepackage[T1]{fontenc}
\usepackage[latin9]{luainputenc}
\usepackage{geometry}
\usepackage[active]{srcltx}
\usepackage{color}
\usepackage{babel}
\usepackage{amstext}
\usepackage{amsthm}
\usepackage{amssymb} 
\usepackage{xcolor}
\usepackage{stmaryrd}
\usepackage{graphicx}
\usepackage{setspace}
\usepackage{enumitem}
\usepackage[unicode=true,
 bookmarks=false,
 breaklinks=false,pdfborder={0 0 1},backref=section,colorlinks=true]
 {hyperref}
\hypersetup{
 linkcolor=blue,urlcolor=blue,citecolor=blue,anchorcolor=blue}

\makeatletter
\numberwithin{equation}{section}
\numberwithin{figure}{section}
\theoremstyle{plain}
\newtheorem{thm}{\protect\theoremname}[section]
\theoremstyle{plain}
\newtheorem{assumption}[thm]{\protect\assumptionname}
\theoremstyle{remark}
\newtheorem{rem}[thm]{\protect\remarkname}
\theoremstyle{plain}
\newtheorem{defn}[thm]{\protect\definitionname}
\theoremstyle{remark}
\newtheorem{notation}[thm]{\protect\notationname}
\theoremstyle{plain}
\newtheorem{prop}[thm]{\protect\propositionname}
\theoremstyle{plain}
\newtheorem{lem}[thm]{\protect\lemmaname}
\theoremstyle{plain}

\@ifundefined{date}{}{\date{}}
\usepackage{babel}

\makeatother

\newcommand{\R}{\mathbb{R}}

\providecommand{\assumptionname}{Assumption}
\providecommand{\lemmaname}{Lemma}
\providecommand{\corollaryname}{Corollary}
\providecommand{\notationname}{Notation}
\providecommand{\propositionname}{Proposition}
\providecommand{\remarkname}{Remark}
\providecommand{\definitionname}{Definition}
\providecommand{\theoremname}{Theorem}

\begin{document}
\title[Regular points and the boundary-value problem for a hypoelliptic process]{Regular points and the boundary-value problem for a hypoelliptic process in a general smooth domain}
\author[Mouad Ramil]{Mouad Ramil, Mathias Rousset}
\address{Univ Rennes, INRIA, IRMAR [(Institut de Recherche Mathematique de Rennes)] - UMR 6625, F-35000 Rennes, France}
\email{mouad.ramil@inria.fr, mathias.rousset@inria.fr}

\begin{abstract} We fully describe the set of regular and irregular boundary points for a general hypoelliptic stochastic process $(X(t))_{t\geq0}$ absorbed outside a $C^{3,\,1}$ domain $\Omega$. The well-posedness of the associated boundary-value problem on $\Omega$ is also established in the classical sense. Let us emphasize that, throughout this work, no assumption is made regarding the inherent structure of the boundary of $\Omega$. In particular, characteristic and non-characteristic points may coexist on $\partial\Omega$, thus extending previous results obtained in the literature. We show that, unlike previous settings, the solution to this boundary-value problem may exhibit discontinuity at the boundary $\partial\Omega$. In certain cases involving a $C^\infty$ domain $\Omega$ and $C^\infty$ coefficients, the discontinuity set is even of positive Lebesgue surface measure on $\partial\Omega$. Finally, we highlight a cone condition under which the continuity is preserved at the boundary.
\end{abstract}

\maketitle

\section{Introduction} 

Consider the hypoelliptic process $(X(t)=(q(t),\,p(t)))_{t\geq0}$ which is the strong solution to the stochastic differential equation (SDE)
\begin{equation}\label{eq:sde}
\left\{ \begin{aligned} & \mathrm{d}q(t)=G(p(t))\mathrm{d}t\;,\\
 & \mathrm{d}p(t)=F(q(t),\,p(t))\mathrm{d}t+\Sigma(q(t),\,p(t))\mathrm{d}B(t)\;,
\end{aligned}
\right.
\end{equation}
where $G:\R^d\to\R^d$, $F:\mathbb{R}^{2d}\to\mathbb{R}^d$ and $\Sigma:\R^{2d}\to\R^{d\times d}$ are $C^\infty$ functions.
Given a smooth open domain $\Omega\subset\R^{2d}$, let
$$\tau_{\Omega^c}:=\inf\{t>0:\,X(t)\notin \Omega\}\;.$$
In this article, we are interested in three main questions arising in Potential Theory:
\begin{enumerate}
    \item Considering an initial condition $x\in\partial\Omega$, when is the first exit time $\tau_{\Omega^c}$ equal to zero?
    \item What are the subsets of $\partial\Omega$ which are attained by $(X(t))_{t\geq0}$ with probability one?
    \item Which is an appropriate definition of the classical boundary-value problem on $\Omega$ associated with the hypoelliptic infinitesimal generator of $(X(t))_{t\geq0}$? At which points of $\overline{\Omega}$ is the solution continuous?
\end{enumerate}

\subsection{Question $(1)$} The first question has been widely studied in the literature. The standard terminology is the following; we say that
\begin{itemize}
    \item $x$ is \emph{regular} if $\mathbb{P}_x(\tau_{\Omega^c}=0)=1$,
    \item $x$ is \emph{irregular} if $\mathbb{P}_x(\tau_{\Omega^c}=0)=0$.
\end{itemize}
Note that the Blumenthal zero-one law guarantees that the probability above is either equal to $0$ or $1$. Previous results have shown that when $\Omega$ is a smooth domain and the diffusion process under study admits a uniformly elliptic infinitesimal generator, all boundary points are regular. However, given the lack of noise in the position coordinates $q$ for the process~\eqref{eq:sde}, the infinitesimal generator in this case is only hypoelliptic.

The classification of regular and irregular boundary points for hypoelliptic diffusion processes such as~\eqref{eq:sde} has only been established in the literature in specific settings. Consider for instance the cylindrical setting $\Omega=\mathcal{O}\times\R^d$ where $\mathcal{O}$ is a smooth domain. In this case, the boundary is fully characteristic and the classification of regular and irregular boundary points was established in~\cite{LelRamRey}. The exit event in this case is only determined by the position coordinate which significantly simplifies the computations. More recently, the description of regular boundary points was completed in~\cite{avelin2026weak} for product domains $\Omega:=\mathcal{U}\times\mathcal{V}$. The scheme of proof relies on the construction of well-suited barrier functions and heavily relies on the product structure of the domain.

In this work, we achieve a full classification of the regular and irregular boundary points for a general $C^{3,\,1}$ domain $\Omega$. Our result detailed in Theorem~\ref{thm:reg points} highlights the importance of second order terms of the distance function in the description of regular boundary points. We also show that the tangential or grazing boundary points are to be regular only in certain cases which are explicitly identified, providing an important difference with the characteristic setting studied in~\cite{LelRamRey}.

\subsection{Question $(2)$} We say that the set $\mathcal{P}$ is polar for the process $(X(t))_{t\geq0}$ if for all $x\notin\mathcal{P}$,
$$\mathbb{P}_x(\exists t>0,\,X(t)\in \mathcal{P})=0\;.$$
The boundary $\partial\Omega$ is partitioned in Section~\ref{sec:domain def} using relevant quantities constructed from the distance function. We investigate the polarity of each set of this explicit partition. In particular, we extend a result of~\cite{LelRamRey} and assert that the set of grazing boundary points is always polar, see Section~\ref{sec:regular polar pts}.

\subsection{Question $(3)$} We consider the problem of finding classical solutions $u$ to the boundary-value problem 
\begin{equation}\label{eq:bdy value pb}
\left\{
    \begin{aligned}
        \mathcal{L} u(x)+c(x)u(x) &=g(x)\,,&&\quad x\in \Omega,\\
        u(x) &=f(x) \,,&&\quad x\in \Lambda
    \end{aligned}
\right.
\end{equation}
where $\Lambda\subset\partial\Omega$ is a boundary subset to be determined, $\mathcal{L}$ is the infinitesimal generator of~\eqref{eq:sde}, $c,\,g$ are smooth functions, and $f$ is the prescribed boundary value. We propose an appropriate boundary subset $\Lambda$ leading to the existence of a unique solution $u \in C^2(\Omega)\cap C_b(\Omega\cup\Lambda)$ for any $f$ bounded and continuous on $\Lambda$. We also provide a precise study of the continuity of $u$ at boundary points.

Using standard probabilistic arguments, it can be shown that, if a smooth solution $u$ of~\eqref{eq:bdy value pb} exists, it admits the following probabilistic representation
\begin{equation}\label{eq:def fctn h}
h:x\in\Omega\mapsto\mathbb{E}_x\bigg[\mathrm{e}^{\int_0^{\tau_{\Omega^c}}c(X(s))\mathrm{d}s}f(X(\tau_{\Omega^c}))-\int_0^{\tau_{\Omega^c}}\mathrm{e}^{\int_0^sc(X(r))\mathrm{d}r}g(X(s))\mathrm{d}s\bigg]\;.
\end{equation}
Given the term $f(X(\tau_{\Omega^c}))$ in the expectation, the boundary condition applies only to an almost-surely attainable subset of $\partial\Omega$ (a subset hit by~\eqref{eq:sde} when exiting $\Omega$ with probability one). As a result, studying the exit event and excluding the set of polar points obtained in Question $(2)$, we are able to deduce the well-posedness of~\eqref{eq:bdy value pb} for an appropriate and explicit boundary subset $\Lambda$. The fact that~\eqref{eq:def fctn h} is indeed the unique classical solution to~\eqref{eq:bdy value pb} is obtained by checking that it is a weak solution combined with a hypoellipticity argument. In addition, the proof of the continuity of $h$ at $\Lambda$ involves trajectorial arguments.

Previous works have already provided the well-posedness of the boundary-value problem associated to a hypoelliptic operator on a smooth domain admitting a non-characteristic boundary $\partial\Omega$, see~\cite{cattiaux1992stochastic,derridj1971probleme,bony1969principe}. Through different techniques, the authors were also able to show therein the smoothness of the solution $u$ up to the boundary $\partial\Omega$. In the fully characteristic setting, a vast literature was developed as well regarding the well-posedness of weak solutions, for instance in~\cite{zhu2024regularity,garain2022regularity,sanchez2023krein,bernard2024existence,avelin2026weak,Mischler}. The well-posedness of classical solutions was also developed in~\cite{Vel,LelRamRey}. Solutions therein are shown to be continuous in $\overline{\Omega}$ but exhibit a loss of regularity at the grazing set where they are only H\"older continuous. The optimal H\"older regularity at the grazing set was established in~\cite{ramil2022explicit,kim2026sharp}.

 In this article, we also explore the continuity of the solution $u$ to~\eqref{eq:bdy value pb} on a general $C^{3,\,1}$ domain $\Omega$. We show that the solution $u$ does not always belong to $C_b(\overline{\Omega})$. The discontinuity phenomenon happens when characteristic and non-characteristic points are allowed to coexist at the boundary $\partial\Omega$. This constitutes a fundamental difference with the previous cases studied in the literature. In particular, we are also able to provide a solution $u$ of~\eqref{eq:bdy value pb} with $C^\infty$ coefficients on a $C^\infty$ domain $\Omega$  such that the discontinuity set at the boundary has positive Lebesgue surface measure on $\partial\Omega$, highlighting the possible roughness of solutions even in smooth settings. Nonetheless, the boundary continuity is shown to hold at any boundary point provided the convergent sequence remains in a cone supported by the boundary point, thus excluding sequences that graze the boundary. All in all, this work may be a building stone for more refined work on the study of the boundary regularity of solutions to~\eqref{eq:bdy value pb} for hypoelliptic operators on general domains including less smooth domains.

\section{Model and main results}\label{sec:model}
\subsection{Model}
\subsubsection{Hypoelliptic process}
Let $d\geq1$. In this article, we focus on the study of a general hypoelliptic process $(X(t)=(q(t),\,p(t)))_{t\geq0}$
in $\mathbb{R}^{d}\times\mathbb{R}^{d}$ solution to~\eqref{eq:sde} and study the well-posedness of the associated boundary-value problem~\eqref{eq:bdy value pb} whose coefficients are assumed to satisfy the following conditions.

\begin{assumption}[Coefficients]\label{ass:coeff}\
\begin{itemize}
    \item $F\in C^\infty(\mathbb{R}^{2d},\,\mathbb{R}^d)$\,;
    \item $G:\R^d\to\R^d$ is a $C^\infty$ diffeomorphism\,;
    \item $\Sigma\in C^\infty(\R^{2d},\,\R^{d\times d})$ and $\Sigma(x)$ is invertible for all $x\in\R^{2d}$\,;
    \item $c,\,g\in C_b^\infty(\mathbb{R}^{2d},\,\mathbb{R})$.
\end{itemize}
\end{assumption}

For $x\in\R^{2d}$, we shall denote by $\mathbb{P}_x$ the law of the process $(X(t))_{t\geq0}$
under which $X(0)=x$. The expectation under $\mathbb{P}_x$ is denoted by $\mathbb{E}_x$. Let us also make the following assumption.
\begin{assumption}[Non-explosion]\label{ass:non expl}
The explosion time $\tau_\infty$ of~\eqref{eq:sde} is such that for all $x\in\R^{2d}$,
$$\mathbb{P}_x(\tau_\infty=\infty)=1\;.$$
\end{assumption}

The infinitesimal generator $\mathcal{L}$ associated to the stochastic process $(X(t))_{t\geq0}$ writes for $(q,\;p)\in\mathbb{R}^{d}\times\mathbb{R}^{d}$,
\begin{equation}
\mathcal{L}=\langle G(p),\,\nabla_{q}\rangle+\langle F(q,\,p),\,\nabla_{p}\rangle+\frac{1}{2}\mathrm{Tr}\big(\Sigma(q,\,p)\Sigma(q,\,p)^\dagger\nabla^2_{p,\,p}\big)\;.\label{eq:gen}
\end{equation}
Let us now introduce several notations used throughout this work.

\begin{notation}
\label{not:main}The following notations will be frequently used in
this article. 
\begin{itemize}
\item (Vector notation) We denote by 
\begin{equation}
x=(q,\,p)\in\mathbb{R}^{2d}\label{eq:xqp}
\end{equation}
 a $2d$-dimensional vector where $q,\,p\in\mathbb{R}^{d}$.
\item (Coordinate) $v_i$ is the $i$-th coordinate of the vector $v\in\R^n$.
\item (Open ball) $\mathrm{B}(v,\,\rho)$ is the open ball centered on $v\in\R^n$ of radius $\rho>0$\;.
\item (Identity) $I_n$ is the identity function on $\R^n$\;.
\item (Space function) $C^{k,\,\alpha}(\R^n,\,\R^m)$ is the set of $C^k$ functions with values on $\R^m$ whose partial derivatives of order $k$ are $\alpha$-H\"olderian continuous. Additionally, let $C^{k,\,\alpha}(\R^n):=C^{k,\,\alpha}(\R^n,\,\R)$
\item (Differential) For $G\in C^1(\R^n,\,\R^m)$, we denote by
$$DG(x):=(\partial_{x_j} G_i(x))_{1\leq i\leq m,\,1\leq j\leq n}\in\R^{m\times n}$$
the differential matrix of $G$ at $x\in\R^n$.
\item For a symmetric matrix $\mathbb{M}$, we denote by $\lambda_{\min}(\mathbb{M})$ its smallest eigenvalue and by $\mathrm{rank}(\mathbb{M})$ its rank.
\item (Differential tensors) For $f\in C^{3}(\mathbb{R}^{2d},\,\mathbb{R})$ and $x=(q,\,p)\in\mathbb{R}^{2d}$,
we denote by
\begin{align*}
\nabla_pf(x)&=\left(\partial_{p_i}f(x)\right)_{1\leq i\leq d}\in\R^d\,;\\
\nabla^2_{p,\,p}f(x)&=\left(\partial^2_{p_i,\,p_j}f(x)\right)_{1\leq i,\,j\leq d}\in\R^{d\times d}\,;\\
\nabla^3_{p,\,p,\,p}f(x)&=\left(\partial^3_{p_i,\,p_j,\,p_k}f(x)\right)_{1\leq i,\,j,\,k\leq d}\in\R^{d\times d\times d}\;.
\end{align*}
The same definition holds when considering the partial derivatives with respect to $q$ and/or mixed-derivatives in $q,\,p$. In particular, notice that
$$\nabla^2_{q,\,p}f(x)=(\nabla^2_{p,\,q}f(x))^\dagger\;.$$
\end{itemize}
\end{notation}
\subsubsection{Domain definition}\label{sec:domain def}

For any set $A\subset\R^{2d}$, let $\tau_{A}$ be the first hitting time of $A$ for $(X(t))_{t\geq0}$, i.e.
\begin{equation}\label{eq:not hitting time}
\tau_{A}:=\inf\{t>0:X(t)\in A\}\;.
\end{equation}

Let us now consider a set $\Omega\subset\R^{2d}$ and let $\tau_{\Omega^c}$ be the first exit time from $\Omega$ of  $(X(t))_{t\geq0}$. We make the following assumptions on $\Omega$ and $\tau_{\Omega^c}$ depending on the positivity of the constant
\begin{equation}\label{eq:def c_infty}
c_\infty:=\sup_{x\in\Omega}\big(c(x)\lor0\big)\;.
\end{equation}

\begin{assumption}[Exit time]\label{ass:Omega}
The set $\Omega$ is an open connected and $C^{3,\,1}$ subset of $\R^{2d}$. In addition, we assume that
\begin{equation}\label{eq:uniform integ tau}
\begin{cases}
    \underset{T\rightarrow\infty}{\lim}\;\underset{x\in\Omega}{\sup}\;\mathbb{E}_x\big[\mathbf{1}_{\tau_{\Omega^c}>T}(1+||g||_\infty\tau_{\Omega^c})\big]=0\;,& \text{if $c_\infty\leq0$,}\\
    \underset{T\rightarrow\infty}{\lim}\;\underset{x\in\Omega}{\sup}\;\mathbb{E}_x\big[\mathbf{1}_{\tau_{\Omega^c}>T}\,\mathrm{e}^{c_\infty\tau_{\Omega^c}}\big]=0\;,& \text{otherwise.}
  \end{cases}
  \end{equation}
\end{assumption}
\begin{rem} Assume that $c_\infty>0$ and let $\Omega=\mathcal{O}\times\R^d$ be a cylinder. Under suitable assumptions on the variations of $F,\,\Sigma$ and the existence of a Lyapunov function (see Assumptions (A1-glob), (A2-glob), (H) in~\cite{champagnat2024quasi}), it was shown in~\cite[Theorem 2.5]{champagnat2024quasi} that the process $(X(t))_{t\geq0}$ admits a unique quasi-stationary distribution on $\Omega$, which implies uniform exponential moments of $\tau_{\Omega^c}$, thus ensuring~\eqref{eq:uniform integ tau} provided $c_\infty$ is not too large.
\end{rem}
Let us denote by $\mathrm{d}_{\Omega}$ the signed Euclidean distance to the boundary $\partial\Omega$, i.e.
\begin{equation*}
  \mathrm{d}_{\Omega} : x \in \mathbb{R}^{2d} \mapsto \begin{cases}
    \mathrm{dist}(x,\partial\Omega) & \text{if $x \in \Omega$,}\\
    -\mathrm{dist}(x,\partial\Omega) & \text{if $x \notin \Omega$.}
  \end{cases}
\end{equation*}
Given Assumption~\ref{ass:Omega}, the distance function $\mathrm{d}_{\Omega}$ belongs to $C^{3,\,1}$ in a neighborhood of any boundary point in $\partial\Omega$. Denote by $n(x)\in\R^{2d}$ the unitary inward normal vector to $\Omega$ at $x\in\partial\Omega$. Then, for all $x\in\partial\Omega$,
\begin{equation}\label{eq:normal gradient Psi}
n(x)=\nabla\mathrm{d}_{\Omega}(x)\;.
\end{equation}
The normal vector $n(x)$ is also decomposed as follows
\begin{equation}\label{eq:projection n(x)}
n(x)=\begin{pmatrix}n_q(x) \\
n_p(x)
\end{pmatrix}\in\R^{2d}\,,
\end{equation}
where $n_{q}(x),\,n_p(x)\in\R^d$ correspond to the projections of $n(x)$ on the $(q,\,p)$-coordinates.

Let us now define a partition of the boundary $\partial\Omega$ which, as it will turn out, enables the classification of regular points. In particular, this variety of behaviors contrasts with the well-known elliptic setting when the domain is smooth. The partitioning involves among other quantities the following scalar product defined for $x=(q,\,p)\in\R^{2d}$ by
\begin{equation}\label{eq:def xi}
\xi(x):=\langle G(p),\,n_q(q,\,p)\rangle\;.
\end{equation}
The boundary $\partial\Omega$ is then partitioned as follows
\begin{align*}
\mathrm{T}&:=\{x\in\partial\Omega:\,|n_p(x)|>0\}&&\text{(non-characteristic or transversal boundary)}\;,\\
\Upsilon&:=\{x\in\partial\Omega:\,|n_p(x)|=0\,,\,\,\,\lambda_{\min}(\nabla^2_{p,\,p}\mathrm{d}_\Omega(x))<0\}&&\text{(characteristic $p$-nonconvex boundary)}\;,\\
\Gamma&:=\{x\in\partial\Omega:\,|n_p(x)|=0\,,\,\,\,\lambda_{\min}(\nabla^2_{p,\,p}\mathrm{d}_\Omega(x))\geq0\}&&\text{(characteristic $p$-convex boundary)}\;.
\end{align*} 
Furthermore, the subset $\Gamma$ is partitioned as follows
\begin{align*}
\Gamma^+&:=\{x\in\Gamma:\,\xi(x)>0\}&&\text{(entering velocities)}\,,\\
\Gamma^-&:=\{x\in\Gamma:\,\xi(x)<0\}&&\text{(exiting velocities)}\,,\\
\Gamma^0&:=\{x\in\Gamma:\,\xi(x)=0\}&&\text{(grazing velocities)}\,.
\end{align*}
Note that the normal vector $n(x)$ is pointed toward the inside of $\Omega$, hence the names provided above. This convention differs with the choice made in~\cite{LelRamRey} for instance where the normal vector is pointed toward the outside of $\Omega$. We conclude this partition with a decomposition of the grazing boundary $\Gamma^0$
\begin{align*}
\Gamma^{0,\,r}&:=\{x\in\Gamma^0:\,\mathrm{rank}(\nabla^2_{p,\,p}\mathrm{d}_\Omega(x))\in\llbracket0,\,2\rrbracket\}&&\text{(regular grazing boundary)}\,,\\
\Gamma^{0,\,i}&:=\{x\in\Gamma^0:\,\mathrm{rank}(\nabla^2_{p,\,p}\mathrm{d}_\Omega(x))\geq3\}&&\text{(irregular grazing boundary)}\,.
\end{align*}

The partitioning of $\partial\Omega$ is actually independent of the choice of the level set function defining the domain $\Omega$. In fact, we shall see in Proposition~\ref{prop:indep of the def function} that if we replace the distance $\mathrm{d}_\Omega$ in the partitioning by any smooth level set function without critical point, then the corresponding subsets in the partition of $\partial\Omega$ remain invariant.

\subsubsection{Discussion on the partitioning}\label{sec:discussion partitioning}
The cylindrical setting ($\Omega=\mathcal{O}\times\R^d$) studied for instance in~\cite{LelRamRey} is a particular case of the present work. The partitioning of $\partial\Omega$ in this case is given by
$$\partial\Omega=\Gamma^+\cup\Gamma^-\cup\Gamma^0\;.$$
If we denote by $\mathcal{R}$ and $\mathcal{I}$ the set of regular and irregular boundary points on $\partial\Omega$, then
$$\mathcal{R}=\Gamma^-\cup\Gamma^0\;,\qquad\mathcal{I}=\Gamma^+\;,$$
which was shown in~\cite[Proposition 2.8]{LelRamRey} for a simplified version of~\eqref{eq:sde}. In the non-cylindrical setting, in many instances the domain of interest is delimited by level sets of a separable \emph{Hamiltonian function}, see for instance~\cite{EK},
\begin{equation}\label{eq:hamiltonian set}
\Omega:=\{U(q)+|p|^2/2>U_m\}\;.
\end{equation}
In this case, if we assume that $G=I_d$ in~\eqref{eq:sde}, the partitioning of $\partial\Omega$ is given by
$$\partial\Omega=\mathrm{T}\cup\Gamma^0\;.$$
The description of the regular and irregular boundary points in this case as well as in the cylindrical setting will be a particular consequence of Theorem~\ref{thm:reg points}.

\subsubsection{Cylindrical part}   The interior of the boundary complementary of $\mathrm{T}$ satisfies by definition $\partial_q d_{\Omega} \neq 0 $ on each its points. By the implicit function theorem, this implies that it is locally cylindrical, that is of the form $\mathcal{O}\times\R^d$. In particular the momentum Hessian vanishes: $\nabla_{p,p}^2 d_\Omega = 0$, and the classification of regular points simplifies as in general cylindrical case described above.
The classification above is thus meaningful for points on the boundary of the non-characteristic set $\mathrm{T}$, and it could be argued from the decomposition that for the latter, $\Gamma$ are of "cylindrical type", while  $\Upsilon$ are of "transversal type".

\subsection{Main results}

\subsubsection{Regular, irregular points and polar sets}\label{sec:regular polar pts} Our first main result describes the regularity of boundary points at $\partial\Omega$ for the process~\eqref{eq:sde}. Denote by $\mathcal{R}$ and $\mathcal{I}$ the set of regular and irregular points at $\partial\Omega$ for the process $(X(t))_{t\geq0}$. 
\begin{thm}[Description of $\mathcal{R}$ and $\mathcal{I}$]\label{thm:reg points}\
\begin{align*}
\mathcal{R}&=\mathrm{T}\cup\Upsilon\cup\Gamma^-\cup\Gamma^{0,\,r}\;,\\
\mathcal{I}&=\Gamma^+\cup\Gamma^{0,\,i}\;.
\end{align*}
\end{thm}

\noindent Let us provide some comments on this result. Unlike what can be expected, the boundary subset
$$\{x\in\partial\Omega:\,|n_p(x)|=0\,,\,\,\,\xi(x)>0\}$$
is not necessarily irregular. In fact, the existence of a negative eigenvalue for the Hessian matrix $\nabla^2_{p,\,p}\mathrm{d}_\Omega(x)$ takes precedent over the sign of $\xi(x)$ to guarantee the regularity. Another striking element is that the regularity of a point $x\in\Gamma^0$ depends on the rank of the matrix $\nabla^2_{p,\,p}\mathrm{d}_\Omega(x)$. For instance, in the Hamiltonian setting described in~\eqref{eq:hamiltonian set}, the Hessian matrix is the identity matrix and therefore the regularity of $\Gamma^0$ depends on the phase-space dimension as follows
\begin{itemize}
    \item if $d\leq2$,\quad$\Gamma^0$ is regular;
    \item if $d\geq3$,\quad$\Gamma^0$ is irregular.
\end{itemize}

Let us also show the polarity of certain boundary subsets.

\begin{thm}[Polar set]\label{thm:polarity} The set
$$\big\{x\in\partial\Omega:|n_p(x)|=0,\,\xi(x)=0\big\}\cup\big\{x\in\partial\Omega:|n_p(x)|=0,\,\nabla^2_{p,\,p}\mathrm{d}_\Omega(x)\neq\mathbb{O}_{d}\big\}$$
is polar for the process~\eqref{eq:sde}.
\end{thm}

The polar set above can be interpreted as the set of points $x$ at the characteristic boundary ($|n_p(x)| = 0$) defined either by grazing velocities ($\xi(x)=0$) or by not being locally cylindrical ($\nabla^2_{p, \, p} \mathrm{d}_\Omega (x)\neq\mathbb{O}_{d}$). In particular, an immediate consequence of the theorem above is that the sets $\Gamma^0$ and $\Upsilon$ are always non attainable sets for the process~\eqref{eq:sde} starting from almost-every initial condition in $\R^{2d}$. However, the other sets defined in the partitioning of $\partial\Omega$ may be attained. In fact, $\Gamma^+$ and $\Gamma^-$ are attained in the cylindrical setting, see~\cite{lelievre2024estimation}, and $\mathrm{T}$ is attained in the Hamiltonian setting since $\Gamma^0$ is polar.

\begin{rem}[Grazing velocities] Unlike the cylindrical case, the set
$$\{x\in\partial\Omega:\,\xi(x)=0\}$$
may be attained with a positive probability. As an example, let $\mathbb{A}$ be an invertible skew-symmetric matrix and let us define the open set
$$\Omega:=\{(q,\,p)\in\R^{2d}:\,\langle q,\,\mathbb{A}p\rangle>1\}\;.$$
If $G=I_d$, then for all $x=(q,\,p)\in\partial\Omega$,
$$\xi(x)=\langle p,\mathbb{A}p\rangle=0\;.$$
However, under proper recurrence assumptions, the process~\eqref{eq:sde} exits $\Omega$ in finite time and thus reaches $\partial\Omega$ in finite time almost-surely.
\end{rem}
We complete the previous results by describing the attainable boundary points on $\partial\Omega$ for the process~\eqref{eq:sde} starting inside $\Omega$. 
Let us define the boundary subset:
\begin{equation}\label{eqn:boundary set}
\Lambda := \mathrm{T}\cup\big\{x=(q,\,p)\in\partial\Omega:\,|n_p(x)|=0,\,\nabla^2_{p,\,p}\mathrm{d}_\Omega(x)=\mathbb{O}_{d},\,\xi(x)<0\big\} .
\end{equation}
By construction  $\Lambda \subset \mathrm{T} \cup \Gamma^-$;
in conjunction with Theorem~\ref{thm:reg points}, we deduce that the boundary set $\Lambda$ contains regular points only:
$$
\Lambda \subset \mathcal R.
$$
Note also that
\[
\mathcal R \subset \overline{\Lambda}\;.
\]
In fact, points in $\mathcal R \setminus \Lambda$ are by definition the characteristic points with non zero Hessian in velocity $\nabla^2_{p,\,p}\mathrm{d}_\Omega(x) \neq \mathbb{O}_{d}$ which have zero Lebesgue surface measure as is later shown in Proposition~\ref{prop:control neighborhood Gamma_0} and therefore belong to the boundary of the non-characteristic part $\mathrm{T} \subset\Lambda$.
\begin{thm}[Exit event]\label{thm:polarity Gamma+} Recall the notation~\eqref{eq:not hitting time}. For all $x\in\Omega\cup\mathcal{I}$,
\begin{equation}\label{eq:proba hitting ext vs bdy}
\mathbb{P}_x\big( \tau_{\partial \Omega}  = \tau_{\Omega^c} \big)=1
\end{equation}
and
\begin{equation}\label{eq:bdy exit event}
\mathbb{P}_x\big( X(\tau_{\partial \Omega})\in\Lambda \big)=1\;.
\end{equation}
%$$\mathbb{P}_x\big(\{X(\tau_{\Omega^c})\in\mathrm{T}\}\cup\{|n_p(X(\tau_{\Omega^c}))|=0,\,\nabla^2_{p,\,p}\mathrm{d}_\Omega(X(\tau_{\Omega^c}))=\mathbb{O}_{d},\,\xi(X(\tau_{\Omega^c}))<0\}\big)=1\;.$$
\end{thm}

\begin{rem}\label{rem:first exit point} The theorem above states that the process exhibits two distinct types of exit event: either an exit by the non-characteristic (transversal) part of the boundary, or an exit through the 'cylindrical part' of the boundary with an outward-pointing velocity.

\end{rem}

\subsubsection{Boundary-value problem}

In this last section we state the well-posedness of the boundary-value problem on $\Omega$ associated with the operator $\mathcal{L}$ defined in~\eqref{eq:gen}.

 %We define the function $h$ as in~\eqref{eq:def fctn h} where the coefficients therein satisfy Assumption~\ref{ass:coeff}.

\begin{thm}[Boundary-value problem]\label{thm:edp eq pot} Let $f\in C_b(\Lambda)$. There exists a unique classical solution $u$ in $C^2(\Omega)\cap C_b(\Omega\cup\Lambda)$ of
\begin{equation}\label{eq:edp h}
  \left\{\begin{aligned}
    \mathcal{L}\,u(x) +c(x)u(x) & =g(x), &&\qquad x\in\Omega\;,\\
    u(x) &=f(x), && \qquad x\in\Lambda \;.% \mathrm{T}\cup\Gamma^-\;.
  \end{aligned}\right. 
\end{equation}
Besides, for all $x\in\Omega \cup\Lambda$, $u(x) = h(x)$ where $h$ is defined by the probabilistic representation ~\eqref{eq:def fctn h}.
\end{thm}

Let us now analyze the boundary continuity of the solution $h$. For $f\in C_b(\Lambda)$, define
\begin{equation}\label{eq:def R_f}
\mathcal{R}_f:=\big\{x\in\mathcal{R}:\,f\text{ can be continuously extended at }x\big\}\;.
\end{equation}

\begin{thm}[Continuity of $h$]\label{thm:existence sol PDE}
The unique solution $h$ of the boundary-value problem~\eqref{eq:edp h} is in $C^\infty(\Omega)$. Furthermore, for all $x\in\mathcal{R}_f$,
$$\lim_{y\in\Omega\rightarrow x}h(y)=f(x)\;.$$
For all $x\in\mathrm{int}_{\partial\Omega}\mathcal{I}$,
$$\lim_{y\in\Omega\rightarrow x}h(y)=h(x)\;,$$
where $\mathrm{int}_{\partial\Omega}\mathcal{I}$ is the interior of $\mathcal{I}$ relative to the topology of $\partial\Omega$.  Additionally, if $(x_n)_{n\geq0}$ is a sequence in $\Omega$ converging to $x\in\mathcal{I}\setminus\mathrm{int}_{\partial\Omega}\mathcal{I}$ such that
\begin{equation}\label{eq:cone condition}
\liminf_{n\rightarrow\infty}\frac{\langle x_n-x,\,n(x)\rangle}{|x_n-x|}>0\;;
\end{equation}
then
$$\lim_{n\rightarrow\infty}h(x_n)=h(x)\;.$$
\end{thm}

The theorem above states that $h$ is continuous at any point in $\mathcal{R}_f\cup\mathrm{int}_{\partial\Omega}\mathcal{I}$. Regarding points in $\mathcal{I}\setminus\mathrm{int}_{\partial\Omega}(\mathcal{I})$, the continuity is preserved provided the converging sequence remains in a cone supported by the boundary point, see~\eqref{eq:cone condition}. This excludes convergent sequences grazing the boundary. 

\begin{rem} A consequence of Theorem~\ref{thm:existence sol PDE} is that if $f\in C_b(\mathcal{R})$ and $\overline{\mathcal{R}}\cap\mathcal{I}=\emptyset$ then $h\in C_b(\overline{\Omega})$. The last condition happens for instance if $\Omega$ is a cylinder given that $\overline{\mathcal{R}}=\mathcal{R}$ in this case (see Section~\ref{sec:discussion partitioning}) or in the Hamiltonian setting if $d\in\llbracket1,\,2\rrbracket$ given that $\partial\Omega=\mathcal{R}$ in this case.
\end{rem}
\begin{rem}
Theorem~\ref{thm:edp eq pot} and Theorem~\ref{thm:existence sol PDE} still hold true by the same proofs if the boundary set $\Lambda$ is replaced by any exit subset $\widetilde{\Lambda}\subset\mathcal{R}$ satisfying $\mathbb P_x ( X(\tau_{\partial \Omega}) \in\widetilde{\Lambda}) = 1$ for all $x \in \Omega$.
The definition of the boundary set $\Lambda$ above can also be seen as a refinement of the Fichera boundary set obtained in~\cite{fichera1959unified}.
\end{rem}
We illustrate in the following proposition that in the general case $h\notin C_b(\overline{\Omega})$ as very singular behaviors may appear at the intersection $\overline{\mathcal{R}}\cap\mathcal{I}$ even though the coefficients of~\eqref{eq:sde} and the domain $\Omega$ are $C^\infty$.

\begin{prop}[Discontinuity set]\label{prop:discont set} There exist explicit functions $G,\,F,\,\Sigma,\,c,\,g$ satisfying Assumption~\ref{ass:coeff}, a $C^\infty$ open set $\Omega$ satisfying Assumption~\ref{ass:Omega} and $f\in C^\infty(\mathcal{R})$ such that the function $h$ defined in~\eqref{eq:def fctn h} does not belong to $C_b(\overline{\Omega})$ and its discontinuity set has a positive Lebesgue surface measure on $\partial\Omega$.
\end{prop}

\textbf{Outline of the article.} The proof of Theorem~\ref{thm:reg points} classifying regular and irregular boundary points is provided in Section~\ref{sec:reg points}. Section~\ref{sec: polarity} focuses on the proofs of Theorems~\ref{thm:polarity} and~\ref{thm:polarity Gamma+} describing polar sets at the boundary and the exit subset from $\Omega$. Section~\ref{sec:prop function h} delves into the proof of the well-posedness of the boundary-value problem given in Theorem~\ref{thm:edp eq pot} and proves the continuity of the solution expected in Theorem~\ref{thm:existence sol PDE}. Finally, Section~\ref{sec:counter example} provides an example where the solution of the boundary-value problem in~\eqref{eq:edp h} admits a non-negligible discontinuity set at the boundary $\partial\Omega$ which concludes the proof of Proposition~\ref{prop:discont set}.

\section{Classification of regular points: proof of Theorem~\ref{thm:reg points}}\label{sec:reg points}
The proof of Theorem~\ref{thm:reg points} is divided into three subsections. Section~\ref{sec:prelim results} provides some preliminary results, Section~\ref{sec:small time expansion dist} proves a small-time expansion of the distance $\mathrm{d}_\Omega(X(t))$ which is used in Section~\ref{sec:proof of regular points} to conclude the proof of Theorem~\ref{thm:reg points}.
\subsection{Preliminary results}\label{sec:prelim results}
We first prove here that the partitioning of $\partial\Omega$ is independent of the choice of the defining function of $\Omega$.

\begin{prop}[Defining function]\label{prop:indep of the def function}
Let $\Psi:\R^{2d}\to\R$ be a defining function of $\Omega$, i.e. such that
\begin{align*}
\Omega&=\{x\in\R^{2d}\;:\,\,\Psi(x)>0\}\;,\\
\partial\Omega&=\{x\in\R^{2d}\;:\,\,\Psi(x)=0\}\;,
\end{align*}
and such that for all $x\in\partial\Omega$, $|\nabla\Psi(x)|>0$. Then, for all $x=(q,\,p)\in\partial\Omega$,
$$|n_p(x)|>0\quad\Longleftrightarrow\quad|\nabla_p\Psi(x)|>0\;.$$
Additionally, if $|n_p(x)|=0$, then
\begin{itemize}
    \item $\xi(x)$ and $\langle G(p),\,\nabla_q\Psi(x)\rangle$ share the same sign\,,
    \item $\lambda_{\min}(\nabla_{p,\,p}^2\mathrm{d}_\Omega(x))$ and $\lambda_{\min}(\nabla_{p,\,p}^2\Psi(x))$ share the same sign\,,
    \item $\mathrm{rank}(\nabla_{p,\,p}^2\mathrm{d}_\Omega(x))=\mathrm{rank}(\nabla_{p,\,p}^2\Psi(x))$.
\end{itemize}
\end{prop}
\begin{proof}
Since $\Omega$ is $C^{3,\,1}$ we deduce that $\Psi$ belongs to $C^{3,\,1}$ on a neighborhood of any point in $\partial\Omega$. Also, for all $x\in\partial\Omega$,
\begin{equation}\label{eq:normal defining fctn}
\nabla\mathrm{d}_\Omega(x)=\frac{\nabla\Psi(x)}{|\nabla\Psi(x)|}\;.
\end{equation}
In particular, we easily deduce from~\eqref{eq:normal gradient Psi} and~\eqref{eq:projection n(x)} that $\xi(x)$ and $\langle G(p),\,\nabla_q\Psi(x)\rangle$ share the same sign and that
$$|n_p(x)|>0\qquad\Longleftrightarrow\qquad|\nabla_p\Psi(x)|>0\;.$$
Now let us fix $x\in\partial\Omega$ such that $|n_p(x)|=|\nabla_p\Psi(x)|=0$ and let $e^j\in\R^d$ be the vector defined as follows
$$\forall i\in\llbracket1,\,d\rrbracket\,,\qquad e^j_i=\delta_{i,\,j}\;.$$
Given that the vector $(0,\,e^j)\in\R^{2d}$ is orthogonal to $n(x)$, we deduce from~\eqref{eq:normal defining fctn} that
$$\lim_{t\rightarrow0}\frac{1}{t}\nabla_p\mathrm{d}_\Omega(x+t(0,\,e^j))=\lim_{t\rightarrow0}\frac{1}{t}\frac{\nabla_p\Psi(x+t(0,\,e^j))}{|\nabla\Psi(x+t(0,\,e^j))|}\;.$$
Since this is valid for any $j\in\llbracket1,\,d\rrbracket$ we obtain that
$$\nabla^2_{p,\,p}\mathrm{d}_\Omega(x)=\frac{\nabla^2_{p,\,p}\Psi(x)}{|\nabla\Psi(x)|}\;,$$
which concludes the proof.
\end{proof}

The next lemma introduces small-time comparaisons of time-integrals given a small-time asymptotic of the integrand. We first clarify the meaning of small-time comparaisons through the following definition which is used in the rest of this work.
\begin{defn}[Small-time comparaison]\label{def:small time compar}
Let $\psi$ be a measurable function and let $(Y_1(t))_{t\geq0}$, $(Y_2(t))_{t\geq0}$ be stochastic processes. We say that, almost-surely,
$$Y_1(t)=Y_2(t)+\underset{t\rightarrow0}{O}(\psi(t))$$
if, almost-surely,
$$\limsup_{t\rightarrow0}\frac{|Y_1(t)-Y_2(t)|}{\psi(t)}<\infty\;.$$
\end{defn}

\begin{lem}[Small-time comparaison]\label{lem:holderian martingale}
Let $n\geq1$. Let $(A(t))_{t\geq0}$ be a stochastic process with values in $\R^n$, $(W(t))_{t\geq0}$ be a Brownian motion in $\R^n$. Assume that there exist constants $\rho,\,\theta\geq0$ such that, almost-surely,
\begin{equation}\label{eq:ineq A small time}
A(t)=\underset{t\rightarrow0}{O}(t^\rho\log\log(1/t)^\theta)\;.
\end{equation}
Then, almost-surely,
\begin{align}
\int_0^t|A(s)|\mathrm{d}s&=\underset{t\rightarrow0}{O}(t^{\rho+1}\log\log(1/t)^\theta)\label{eq:cv a.s. quad var small time}\\
\int_0^t\langle A(s),\,\mathrm{d}W(s)\rangle&=\underset{t\rightarrow0}{O}(t^{\rho+1/2}\log\log(1/t)^{\theta+1/2})\;.\label{eq:cv a.s. martingale small time}
\end{align}
\end{lem}
\begin{proof}[Proof of Lemma~\ref{lem:holderian martingale}]
\textbf{Step 1}: Let us first prove~\eqref{eq:cv a.s. quad var small time}. From~\eqref{eq:ineq A small time} we deduce the existence of a finite random variable $C>0$ such that for $t>0$ small enough, almost-surely,
\begin{align*}
\int_0^t|A(s)|\mathrm{d}s&\leq C\int_0^ts^\rho\log\log(1/s)^\theta\mathrm{d}s\\
&=C\frac{t^{\rho+1}}{\rho+1}\log\log(1/t)^\theta+C\frac{\theta}{\rho+1}\int_0^t\frac{s^\rho\log\log(1/s)^\theta}{\log(1/s)\log\log(1/s)}\mathrm{d}s\;,
\end{align*}
by integration by parts. Therefore, we easily deduce~\eqref{eq:cv a.s. quad var small time} since $\log(1/s)\log\log(1/s)\underset{s\rightarrow0}{\longrightarrow}\infty$.

\textbf{Step 2}: Let us now prove~\eqref{eq:cv a.s. martingale small time}. Define the martingale 
$$Y:t\geq0\mapsto\int_0^t\langle A(s),\,\mathrm{d}W(s)\rangle\;.$$
It is a locally square-integrable martingale with quadratic variation $\langle Y\rangle_t=\int_0^t|A(s)|^2\mathrm{d}s$. Assume that there exists $t_0>0$ such that $\langle Y\rangle_{t_0}=0$ then for all $t\in(0,\,t_0)$, $|A(t)|=0$ and~\eqref{eq:cv a.s. martingale small time} immediately ensues.

Assume now that $\langle Y\rangle_{t}>0$ for all $t>0$. By the Dambis-Dubins-Schwartz theorem, there exists a Brownian motion $(\widetilde{W}(t))_{t\geq0}$ such that $Y(t)=\widetilde{W}(\langle Y\rangle_t)$ for all $t\geq0$. Using the Law of the Iterated Logarithm, we deduce that, almost-surely,
$$\limsup_{t\rightarrow0}\frac{|Y(t)|}{\sqrt{2\langle Y\rangle_t\log\log(1/\langle Y\rangle_t)}}=1\;.$$
Using~\eqref{eq:ineq A small time} and Step 1, we also deduce that
$$\langle Y\rangle_t=\underset{t\rightarrow0}{O}(t^{1+2\rho}\log\log(1/t)^{2\theta})\;.$$ 
As a result,
$$\frac{|Y(t)|}{t^{\rho+1/2}\log\log(1/t)^{\theta+1/2}}=\frac{|Y(t)|}{\sqrt{2\langle Y\rangle_t\log\log(1/\langle Y\rangle_t)}}\sqrt{\frac{2\langle Y\rangle_t\log\log(1/\langle Y\rangle_t)}{t^{1+2\rho}\log\log(1/t)^{1+2\theta}}}\;.$$
Given that the function $x>0\mapsto x\log\log(1/x)$ is increasing in a neighborhood of $x=0$ the second term in the right-hand side of the equality above is bounded when $t\rightarrow0$, hence the proof of~\eqref{eq:cv a.s. martingale small time}.
\end{proof}
\subsection{Small-time expansion}\label{sec:small time expansion dist}
In this section, we establish a small-time expansion of $f(X(t))$  for $f\in C^{3,\,1}(\R^{2d})$. In particular, this result provides the small-time behavior of the distance to $\Omega$ of the process~\eqref{eq:sde}. To simplify the statements, we only consider the case when $G=I_d$ in~\eqref{eq:sde}. In Section~\ref{sec:proof of regular points} we shall see how this expansion in the case $G=I_d$ is sufficient to obtain the proof of Theorem~\ref{thm:reg points} for any coefficient $G$ satisfying Assumption~\ref{ass:coeff}.

Let us now remind the reader of the notation~\eqref{eq:xqp} which is used throughout this work and let us state the main result of this subsection. Let
\begin{equation}\label{eq:def M_t}
M:t\geq0\mapsto\int_0^t\Sigma(X(s))\,\mathrm{d}B(s)\;.
\end{equation}
\begin{thm}[Small-time expansion]\label{thm:small time asymp distance}
Assume that $G=I_d$ in~\eqref{eq:sde} and let $f\in C^{3,\,1}(\R^{2d})$. For all $x\in\R^{2d}$, $\mathbb{P}_x$ almost-surely,
\begin{align*}
f(X(t))&=f(x)+\langle\nabla_pf(x),\,M(t)\rangle+\frac{1}{2}\langle\nabla^2_{p,\,p}f(x)M(t),\,M(t)\rangle+t\langle p,\,\nabla_{q}f(x)\rangle\\
&+t\langle F(x),\,\nabla_{p}f(x)\rangle+\frac{1}{6}\langle\nabla^3_{p,\,p,\,p}f(x)\,M(t),\,M(t),\,M(t)\rangle\\
&+\int_0^t\langle M(s),\,\nabla_qf(x)+D_pF(x)^\dagger\nabla_pf(x)\rangle\mathrm{d}s+t\langle M(t),\,\nabla^2_{p,\,q}f(x)\,p+\nabla^2_{p,\,p}f(x)F(x)\rangle\\
&+\underset{t\rightarrow0}{O}(t^2\log\log(1/t)^2)\;,
\end{align*}
where $D_pF(x)=(\partial_{p_j}F_i(x))_{1\leq i,\,j\leq d}\in\R^{d\times d}$.
\end{thm}

The proof of Theorem~\ref{thm:small time asymp distance} relies on several lemmas detailed in this section. Before delving into the proofs, note that the continuity of the sample paths of $(X(t))_{t\geq0}$ ensures that, for all $x\in\R^{2d}$, $\mathbb{P}_x$ almost-surely, 
\begin{equation}\label{eq:asymptote X sde}
X(t)=x+\begin{pmatrix}0 \\
M(t)
\end{pmatrix}+\underset{t\rightarrow0}{O}(t)\;.
\end{equation}
\begin{lem}[Small-time expansion I]\label{lem:small time 1}
Under the assumptions of Theorem~\ref{thm:small time asymp distance}, for all $x\in\R^{2d}$, $\mathbb{P}_x$ almost-surely,
\begin{align*}
&\int_0^t\mathcal{L}\,f(X(s))\mathrm{d}s-t\big(\langle p,\,\nabla_{q}f(x)\rangle+\langle F(x),\,\nabla_{p}f(x)\rangle\big)\\
&=\int_0^t\langle M(s),\,\nabla_{q}f(x)+D_pF(x)^\dagger\nabla_{p}f(x)+\nabla^2_{p,\,q}f(x)p+\nabla^2_{p,\,p}f(x)F(x)\rangle\\
&+\frac{1}{2}\int_0^t\mathrm{Tr}\big(\Sigma(X(s))\Sigma(X(s))^\dagger\nabla^2_{p,\,p}f(X(s))\big)\mathrm{d}s+\underset{t\rightarrow0}{O}(t^2\log\log(1/t))\;.
\end{align*}
\end{lem}
\begin{proof}[Proof of Lemma~\ref{lem:small time 1}] The proof is an immediate consequence of the following small-time expansions proven below
\begin{align}
&\int_0^t\big[\langle p(s),\,\nabla_{q}f(X(s))\rangle-\langle p,\,\nabla_{q}f(x)\rangle\big]\mathrm{d}s\nonumber\\
&=\int_0^t\langle M(s),\,\nabla_{q}f(x)+\nabla^2_{p,\,q}f(x)p\rangle\mathrm{d}s+\underset{t\rightarrow0}{O}(t^2\log\log(1/t))\label{eq:first expansion}\;,
\end{align}
\begin{align}
&\int_0^t\big[\langle F(X(s)),\,\nabla_{p}f(X(s))\rangle-\langle F(x),\,\nabla_{p}f(x)\rangle\big]\mathrm{d}s\nonumber\\
&=\int_0^t\langle M(s),\,D_pF(x)^\dagger\nabla_{p}f(x)+\nabla^2_{p,\,p}f(x)F(x)\rangle\mathrm{d}s+\underset{t\rightarrow0}{O}(t^2\log\log(1/t))\label{eq:second expansion}\;.
\end{align}

Let us start with the proof of~\eqref{eq:first expansion}. We deduce from~\eqref{eq:asymptote X sde} that, $\mathbb{P}_x$ almost-surely,
\begin{align*}
&\langle p(s),\,\nabla_{q}f(X(s))\rangle-\langle p,\,\nabla_{q}f(x)\rangle\\
&=\langle p(s)-p,\,\nabla_{q}f(X(s))\rangle+\langle p,\,\nabla_{q}f(X(s))-\nabla_{q}f(x)\rangle\\
&=\langle M(s),\,\nabla_{q}f(X(s))\rangle+\underset{s\rightarrow0}{O}(s)+\langle p,\,\nabla^2_{q,\,p}f(x)M(s)\rangle+\underset{s\rightarrow0}{O}(|M(s)|^2)\\
&=\langle M(s),\,\nabla_{q}f(x)+\nabla^2_{p,\,q}f(x)p\rangle+\underset{s\rightarrow0}{O}(|M(s)|^2)\;.
\end{align*}
Since $|M(s)|^2=\underset{s\rightarrow0}{O}(s\log\log(1/s))$ by Lemma~\ref{lem:holderian martingale}, we immediately deduce the proof of~\eqref{eq:first expansion}. Similarly, the proof of~\eqref{eq:second expansion} follows from the equalities
\begin{align*}
&\langle F(X(s)),\,\nabla_{p}f(X(s))\rangle-\langle F(x),\,\nabla_{p}f(x)\rangle\\
&=\langle F(X(s))-F(x),\,\nabla_{p}f(X(s))\rangle+\langle F(x),\,\nabla_{p}f(X(s))-\nabla_{p}f(x)\rangle\\
&=\langle D_pF(x)M(s),\,\nabla_{p}f(x)\rangle+\langle F(x),\,\nabla^2_{p,\,p}f(x)M(s)\rangle+\underset{s\rightarrow0}{O}(s\log\log(1/s))\;.
\end{align*}
\end{proof}

The next lemma provides additional small-time expansions appearing in the proof of Theorem~\ref{thm:small time asymp distance}. They involve the following stochastic process in $\R^{2d}$
$$Z:t\geq0\mapsto X(0)+\begin{pmatrix}0 \\
M(t)
\end{pmatrix}\;.$$
\begin{lem}[Small-time expansion II]\label{lem:small time II}
Under the assumptions of Theorem~\ref{thm:small time asymp distance}, for all $x\in\R^{2d}$, $\mathbb{P}_x$ almost-surely,
\begin{align*}
&(1)\quad\int_0^t\langle\nabla_pf(Z(s)),\,\mathrm{d}M(s)\rangle\\
&=\langle\nabla_pf(x),\,M(t)\rangle+\frac{1}{2}\langle\nabla^2_{p,\,p}f(x)M(t),\,M(t)\rangle+\frac{1}{6}\langle\nabla^3_{p,\,p,\,p}f(x)\,M(t),\,M(t),\,M(t)\rangle\\
&-\frac{1}{2}\int_0^t\mathrm{Tr}\big(\Sigma(X(s))\Sigma(X(s))^\dagger\nabla^2_{p,\,p}f(X(s))\big)\mathrm{d}s+\underset{t\rightarrow0}{O}(t^2\log\log(1/t)^{2})\;.\\
&(2)\quad\int_0^t\langle\nabla^2_{p,\,q}f(Z(s))(q(s)-q),\,\mathrm{d}M(s)\rangle\\
&=t\langle M(t),\,\nabla^2_{p,\,q}f(x)\,p\rangle-\int_0^t\langle M(s),\,\nabla^2_{p,\,q}f(x)\,p\rangle\mathrm{d}s+\underset{t\rightarrow0}{O}(t^2\log\log(1/t))\;.\\
&(3)\quad\int_0^t\langle\nabla^2_{p,\,p}f(Z(s))(p(s)-p-M(s)),\,\mathrm{d}M(s)\rangle\\
&=t\langle\nabla^2_{p,\,p}f(x)F(x),\,M(t)\rangle-\int_0^t\langle\nabla^2_{p,\,p}f(x)F(x),\,M(s)\rangle\mathrm{d}s+\underset{t\rightarrow0}{O}(t^{2}\log\log(1/t))\;.
\end{align*}
\end{lem}
\begin{proof}[Proof of Lemma~\ref{lem:small time II}]
Let us fix $x\in\R^{2d}$. We start with the proof of $(1)$. By Lemma~\ref{lem:holderian martingale}, $\mathbb{P}_x$ almost-surely,
\begin{equation}\label{eq:asymptote Z(t)}
|Z(t)-x|=|M(t)|=\underset{t\rightarrow0}{O}(t^{1/2}\log\log(1/t)^{1/2})\;.
\end{equation}
As a result, since $f\in C^{3,\,1}(\R^{2d})$, $\mathbb{P}_x$ almost-surely,
$$\nabla_pf(Z(t))=\nabla_pf(x)+\nabla^2_{p,\,p}f(x)M(t)+\frac{1}{2}\langle\nabla^3_{p,\,p,\,p}f(x)M(t),\,M(t)\rangle+\underset{t\rightarrow0}{O}(t^{3/2}\log\log(1/t)^{3/2})\;.$$
Hence, by Lemma~\ref{lem:holderian martingale},
\begin{align*}
&\int_0^t\langle\nabla_pf(Z(s)),\,\mathrm{d}M(s)\rangle\\
&=\langle\nabla_pf(x),\,M(t)\rangle+\int_0^t\langle\nabla^2_{p,\,p}f(x)M(s),\,\mathrm{d}M(s)\rangle+\frac{1}{2}\int_0^t\langle\nabla^3_{p,\,p,\,p}f(x)M(s),\,M(s),\,\mathrm{d}M(s)\rangle\\
&+\underset{t\rightarrow0}{O}(t^2\log\log(1/t)^2)\;.
\end{align*}
Furthermore, by It\^o's formula,
$$\int_0^t\langle\nabla^2_{p,\,p}f(x)M(s),\,\mathrm{d}M(s)\rangle=\frac{1}{2}\langle\nabla^2_{p,\,p}f(x)M(t),\,M(t)\rangle-\frac{1}{2}\int_0^t\mathrm{Tr}\big(\Sigma(X(s))\Sigma(X(s))^\dagger\nabla^2_{p,\,p}f(x)\big)\mathrm{d}s\;.$$
Additionally,
\begin{align*}
&\int_0^t\langle\nabla^3_{p,\,p,\,p}f(x)\,M(s),\,M(s),\,\mathrm{d}M(s)\rangle\\
&=\frac{1}{3}\langle\nabla^3_{p,\,p,\,p}f(x)\,M(t),\,M(t),\,M(t)\rangle-\int_0^t\mathrm{Tr}\big(\Sigma(X(s))\Sigma(X(s))^\dagger\nabla^3_{p,\,p,\,p}f(x)M(s)\big)\mathrm{d}s\,,
\end{align*}
where $\nabla^3_{p,\,p,\,p}f(x)M(s)\in\R^{d\times d}$ is a matrix defined for $1\leq i,\,j\leq d$ as follows
$$\big(\nabla^3_{p,\,p,\,p}f(x)M(s)\big)_{i,\,j}=\sum_{k=1}^d\partial^3_{p_i,\,p_j,\,p_k}f(x)M_k(s)\;.$$
In order to conclude the proof of $(1)$, it remains to prove that
\begin{align*}
&\int_0^t\mathrm{Tr}\big(\Sigma(X(s))\Sigma(X(s))^\dagger\nabla^2_{p,\,p}f(X(s))\big)\mathrm{d}s\\
&=\int_0^t\mathrm{Tr}\big(\Sigma(X(s))\Sigma(X(s))^\dagger\nabla^2_{p,\,p}f(x)\big)\mathrm{d}s+\int_0^t\mathrm{Tr}\big(\Sigma(X(s))\Sigma(X(s))^\dagger\nabla^3_{p,\,p,\,p}f(x)M(s)\big)\mathrm{d}s\\
&+\underset{t\rightarrow0}{O}(t^2\log\log(1/t)^2)\,,
\end{align*}
which is a consequence of~\eqref{eq:asymptote X sde} and the fact that
$$\nabla^2_{p,\,p}f(X(s))=\nabla^2_{p,\,p}f(x)+\nabla^3_{p,\,p,\,p}f(x)M(s)+\underset{s\rightarrow0}{O}(s\log\log(1/s))\,,$$
hence the proof of $(1)$. Let us now prove the expansion in $(2)$.

By~\eqref{eq:asymptote X sde},~\eqref{eq:asymptote Z(t)} and Lemma~\ref{lem:holderian martingale},
$$q(t)-q=\int_0^tp(s)\mathrm{d}s=tp+\underset{t\rightarrow0}{O}(t^{3/2}\log\log(1/t)^{1/2})\;.$$
Therefore, by Lemma~\ref{lem:holderian martingale},
\begin{align*}
\int_0^t\langle\nabla^2_{p,\,q}f(Z(s))(q(s)-q),\,\mathrm{d}M(s)\rangle&=\int_0^ts\langle\nabla^2_{p,\,q}f(Z(s))\,p,\,\mathrm{d}M(s)\rangle+\underset{t\rightarrow0}{O}(t^2\log\log(1/t))\\
&=\int_0^ts\langle\nabla^2_{p,\,q}f(x)\,p,\,\mathrm{d}M(s)\rangle+\underset{t\rightarrow0}{O}(t^2\log\log(1/t))\;.
\end{align*}
Furthermore, by It\^o's formula,
$$\int_0^ts\langle\nabla^2_{p,\,q}f(x)\,p,\,\mathrm{d}M(s)\rangle=t\langle\nabla^2_{p,\,q}f(x)\,p,\,M(t)\rangle-\int_0^t\langle\nabla^2_{p,\,q}f(x)\,p,\,M(s)\rangle\mathrm{d}s\,,$$
hence $(2)$. It remains now to prove $(3)$. By~\eqref{eq:sde},~\eqref{eq:asymptote X sde} and Lemma~\ref{lem:holderian martingale}
\begin{align*}
p(t)-p-M(t)&=\int_0^tF(X(s))\mathrm{d}s\\
&=tF(x)+\underset{t\rightarrow0}{O}(t^{3/2}\log\log(1/t)^{1/2})\;.
\end{align*}
Therefore, using Lemma~\ref{lem:holderian martingale} we deduce that
$$\int_0^t\langle\nabla^2_{p,\,p}f(Z(s))(p(s)-p-M(s)),\,\mathrm{d}M(s)\rangle=\int_0^ts\langle\nabla^2_{p,\,p}f(x)F(x),\,\mathrm{d}M(s)\rangle+\underset{t\rightarrow0}{O}(t^2\log\log(1/t))\;.$$
Furthermore, by It\^o's formula,
\begin{align*}
&\int_0^ts\langle\nabla^2_{p,\,p}f(x)F(x),\,\mathrm{d}M(s)\rangle\\
&=t\langle\nabla^2_{p,\,p}f(x)F(x),\,M(t)\rangle-\int_0^t\langle\nabla^2_{p,\,p}f(x)F(x),\,M(s)\rangle\mathrm{d}s\;.
\end{align*}
which concludes the proof of $(3)$.
\end{proof}
Let us now prove Theorem~\ref{thm:small time asymp distance}.
\begin{proof}[Proof of Theorem~\ref{thm:small time asymp distance}]
By It\^o's formula, $\mathbb{P}_x$ almost-surely,
\begin{equation}\label{eq:ito formula distance}
f(X(t))=f(x)+\int_0^t\mathcal{L}\,f(X(s))\mathrm{d}s+\int_0^t\langle\nabla_pf(X(s)),\,\mathrm{d}M(s)\rangle\;.
\end{equation}
The asymptotic expansion of the first integral when $t\rightarrow0$ was obtained in Lemma~\ref{lem:small time 1}. It remains to study the second stochastic integral. Define the following stochastic process
$$V:t\geq0\mapsto\nabla_pf(X(t))-\nabla_pf(Z(t))-\nabla(\nabla_{p}f)(Z(t))(X(t)-Z(t))\;.$$
By Taylor's inequality and~\eqref{eq:sde},
$$|V(t)|=\underset{t\rightarrow0}{O}(|X(t)-Z(t)|^2)=\underset{t\rightarrow0}{O}(t^2)\;.$$
Therefore, by Lemma~\ref{lem:holderian martingale},
\begin{align*}
\int_0^t\langle\nabla_pf(X(s)),\,\mathrm{d}M(s)\rangle&=\int_0^t\langle\nabla_pf(Z(s)),\,\mathrm{d}M(s)\rangle+\int_0^t\langle\nabla^2_{p,\,q}f(Z(s))(q(s)-q),\,\mathrm{d}M(s)\rangle\\
&+\int_0^t\langle\nabla^2_{p,\,p}f(Z(s))(p(s)-p-M(s)),\,\mathrm{d}M(s)\rangle+\underset{t\rightarrow0}{O}(t^2)\;.
\end{align*}
As a result, reinjecting the asymptotic expansions obtained in Lemma~\ref{lem:small time II} one immediately concludes the proof.
\end{proof}

\subsection{Proof of Theorem~\ref{thm:reg points}}\label{sec:proof of regular points} This section is devoted to the proof of Theorem~\ref{thm:reg points}. First, we introduce intermediary results which provide small-time asymptotics of the Brownian and time-integrated Brownian motions. Second, we prove Proposition~\ref{prop:liminf bessel int BM} which is a highly non-trivial result aimed at proving the regularity of points in $\Gamma^{0,\,r}$. All these elements are then combined to conclude the proof of Theorem~\ref{thm:reg points}.

Let us first recall some well-known properties satisfied by any given Brownian motion $(W(t))_{t\geq0}$
\begin{itemize}
    \item (Symmetry) $(-W(t))_{t\geq0}$ shares the same law as $(W(t))_{t\geq0}$,
    \item (Time-inversion) $(tW(1/t))_{t\geq0}$ shares the same law as $(W(t))_{t\geq0}$,
    \item (Scaling) $(W(at)/\sqrt{a})_{t\geq0}$ shares the same law as $(W(t))_{t\geq0}$ for any $a>0$.
\end{itemize}

The following lemma is a consequence of well-known results by Erd\"os and Dvoretzky in~\cite{dvoretzky1951some} and Spitzer in~\cite{spitzer1958some} which characterize the lower-class envelopes of the Brownian motion in any dimension.

\begin{lem}[Small-time asymptotic]\label{lem:liminf BM recurrent}
Let $(W(t))_{t\geq0}$ be a Brownian motion in $\R^d$.
\begin{itemize}
    \item If $d=2$,
    \begin{equation}\label{eq:liminf Brownian 2d}
\forall\rho\geq0,\qquad\liminf_{t\rightarrow0}\,\frac{|W(t)|}{t^\rho}=0\;.
\end{equation}
    \item If $d\geq3$,
\begin{equation}\label{eq:liminf Brownian 3d}
\liminf_{t\rightarrow0}\,(\log1/t)^{(1+\epsilon)/(d-2)}\frac{|W(t)|}{t^{1/2}}=\left\{ \begin{aligned} & 0\;,\qquad&&\text{if }\epsilon=0\,,\\
 &+\infty\;,\qquad&&\text{if }\epsilon>0\;.
\end{aligned}
\right.
\end{equation}
\end{itemize}
\end{lem}
\begin{proof}
The first result is a consequence of~\cite[Theorem 1]{spitzer1958some} which states that if $g$ is a positive non-increasing function then
$$\mathbb{P}\left(|W(t)|<\sqrt{t}g(t)\,\text{ for arbitrary large }t>0\right)=0\text{ or }1$$
depending on whether
$$\int^\infty\frac{1}{t|\log g(t)|}\mathrm{d}t<\infty\text{ or }=+\infty\;.$$
Define
$$g:t>1\mapsto\mathrm{exp}(-\log(t)\log\log(t))\;.$$
Then,
$$\int^\infty\frac{1}{t|\log g(t)|}\mathrm{d}t=+\infty\;.$$
As a result,
$$\liminf_{t\rightarrow\infty}\,\frac{|W(t)|}{t^{1/2}\mathrm{exp}(-\log(t)\log\log(t))}\leq1\;.$$
The time-inversion property applied to $(W(t))_{t\geq0}$ ensures that
$$\liminf_{t\rightarrow0}\,\frac{|W(t)|}{t^{1/2+\log\log(1/t)}}\leq1$$
which concludes the proof of~\eqref{eq:liminf Brownian 2d}. 

The proof of~\eqref{eq:liminf Brownian 3d} is a consequence of the result~\cite[Theorem 6]{dvoretzky1951some}. Namely, if $g$ is a positive monotonic function on $(1,\,+\infty)$, then
$$\mathbb{P}\left(|W(t)|<\sqrt{t}g(t)\,\text{ for arbitrary large }t>0\right)=0\text{ or }1$$
depending on whether
$$\int^\infty\frac{g(t)^{d-2}}{t}\mathrm{d}t<\infty\text{ or }=+\infty\;.$$
For $\epsilon>0$, let
$$g:t>1\mapsto\frac{1}{(\log t\log\log t)^{(1+\epsilon)/(d-2)}}\;.$$
Then,
$$\int^\infty\frac{g(t)^{d-2}}{t}\mathrm{d}t=\left\{ \begin{aligned} & +\infty\;,\qquad&&\text{if }\epsilon=0\,,\\
 &<\infty\;,\qquad&&\text{if }\epsilon>0\;.
\end{aligned}
\right.$$
Therefore, the proof follows from~\cite[Theorem 6]{dvoretzky1951some} and the time-inversion property.
\end{proof}
In more recent years, Lachal has also provided in~\cite[Theorem 14]{lachal1997local} the lower-class envelope of the norm of the process $(W_1(t),\,\int_0^tW_1(s)\mathrm{d}s)_{t\geq0}$ where $(W_1(t))_{t\geq0}$ is a one-dimensional Brownian motion. In particular, he has shown that for all $\epsilon>0$, almost-surely,
\begin{equation}\label{eq:cv lachal int BM}
\liminf_{t\rightarrow0}\frac{\log(1/t)^{3/2+\epsilon}}{t^{3/2}}\sqrt{|W_1(t)|^2+\left|\int_0^tW_1(s)\mathrm{d}s\right|^2}=+\infty\;.
\end{equation}

The small-time asymptotics of the Brownian and time-integrated Brownian motions recalled in this section are key ingredients of the proof of Theorem~\ref{thm:reg points}. The most difficult part lies in determining the regularity of the boundary points in $\Gamma^{0,\,r}$ and relies on the following proposition whose proof is fairly technical.
\begin{prop}[Small-time asymptotic]\label{prop:liminf bessel int BM}Let
\begin{itemize}
    \item $\mathbb{S}\in\R^{d\times d}$ be a symmetric non-negative matrix satisfying $\mathrm{rank}(\mathbb{S})\in\llbracket0,\,2\rrbracket$\,;
    \item $\mathbb{T}\in\R^{d\times d\times d}$ be a symmetric tensor\,;
    \item $\alpha,\,\theta\in\R^d$ such that $\theta\neq0$\;.
\end{itemize}
Then, for any $\epsilon>0$, for any $x\in\R^{2d}$, $\mathbb{P}_x$ almost-surely,
\begin{equation}\label{eq:liminf negative}
\liminf_{t\rightarrow0}\frac{\log(1/t)^{3/2+\epsilon}}{t^{3/2}}\left[\langle\mathbb{S}\,M(t),\,M(t)\rangle+\langle\mathbb{T}\,M(t),\,M(t),\,M(t)\rangle+t\langle\alpha,\,M(t)\rangle+\int_0^t\langle\theta,\,M(s)\rangle\mathrm{d}s\right]=-\infty\;,
\end{equation}
where $(M(t))_{t\geq0}$ is defined in~\eqref{eq:def M_t}.
\end{prop}
\begin{proof} 
\textbf{Step 1}: Let us fix $x\in\R^{2d}$ and let us first show the existence of
\begin{itemize}
    \item $\lambda_1,\,\lambda_2>0$\,;
    \item $\widetilde{\mathbb{T}}\in\R^{d\times d\times d}$ a symmetric tensor\,;
    \item $\widetilde{\alpha},\,\widetilde{\theta}\in\R^d$ such that $\widetilde{\theta}\neq0$\,;
    \item $(W(t)=(W_1(t),\ldots,\,W_d(t)))_{t\geq0}$ a Brownian motion in $\R^d$\,;
\end{itemize} 
such that $\mathbb{P}_x$ almost-surely,
\begin{align}
&\langle\mathbb{S}\,M(t),\,M(t)\rangle+\langle\mathbb{T}\,M(t),\,M(t),\,M(t)\rangle+t\langle\alpha,\,M(t)\rangle+\int_0^t\langle\theta,\,M(s)\rangle\mathrm{d}s\nonumber\\
&\leq\lambda_1W_1(t)^2+\lambda_2W_2(t)^2+\langle\widetilde{\mathbb{T}}\,W(t),\,W(t),\,W(t)\rangle+t\langle\widetilde{\alpha},\,W(t)\rangle+\int_0^t\langle\widetilde{\theta},\,W(s)\rangle\mathrm{d}s+\underset{t\rightarrow0}{O}(t^2\log\log(1/t)^2)\label{eq:upper-bound brownian terms}\;.
\end{align}

Define $E:t\geq0\mapsto M(t)-\Sigma(x)B(t)$. By Lemma~\ref{lem:holderian martingale} and~\eqref{eq:asymptote X sde}, $\mathbb{P}_x$ almost-surely,
\begin{equation}\label{eq:difference M(t) Sigma B(t)}
E(t)=\int_0^t(\Sigma(X(s))-\Sigma(x))\mathrm{d}B(s)=\underset{t\rightarrow0}{O}(t\log\log(1/t))\;.
\end{equation}
Thus,
\begin{align}
&\langle\mathbb{S}\,M(t),\,M(t)\rangle\\
&=\langle\mathbb{S}\,\Sigma(x)B(t),\,\Sigma(x)B(t)\rangle+2\langle\mathbb{S}\,\Sigma(x)B(t),\,E(t)\rangle+\underset{t\rightarrow0}{O}(t^2\log\log(1/t)^2)\nonumber\\
&=\langle\Sigma(x)^\dagger\mathbb{S}\,\Sigma(x)B(t),\,B(t)\rangle+2\langle\Sigma(x)^\dagger\mathbb{S}\,\Sigma(x)B(t),\,\Sigma(x)^{-1}E(t)\rangle+\underset{t\rightarrow0}{O}(t^2\log\log(1/t)^2)\;.\label{eq:decomposition SM(t)M(t)}
\end{align}
 Since $\Sigma(x)\in\R^{d\times d}$ is invertible and $\mathbb{S}\in\R^{d\times d}$ is a non-negative symmetric matrix, then $\Sigma(x)^\dagger\mathbb{S}\Sigma(x)$ is a symmetric non-negative matrix satisfying 
\begin{equation}\label{eq:rank product matrix}
\mathrm{rank}(\Sigma(x)^\dagger\mathbb{S}\Sigma(x))=\mathrm{rank}(\mathbb{S})\in\llbracket0,\,2\rrbracket\;.
\end{equation}
Therefore, let $(e_1,\,\ldots,\,e_d)$ be an orthonormal diagonalisation basis of $\Sigma(x)^\dagger\mathbb{S}\Sigma(x)$ associated with the eigenvalues $\lambda_1,\,\lambda_2\geq0$. Up to taking positive upper-bounds of $\lambda_1,\,\lambda_2$, we can assume that $\lambda_1,\,\lambda_2>0$. Consider now the process
$$W:t\geq0\mapsto(\langle B(t),\,e_1\rangle,\,\ldots,\,\langle B(t),\,e_d\rangle)\;.$$
It is easy to see that $(W(t))_{t\geq0}$ is a Brownian motion in $\R^d$. Therefore, by~\eqref{eq:decomposition SM(t)M(t)},
\begin{align*}
&\langle\mathbb{S}\,M(t),\,M(t)\rangle\\
&\leq\lambda_1W_1(t)^2+\lambda_2W_2(t)^2+2\lambda_1W_1(t)\langle e_1,\,\Sigma(x)^{-1}E(t)\rangle+2\lambda_2W_2(t)\langle e_2,\,\Sigma(x)^{-1}E(t)\rangle+\underset{t\rightarrow0}{O}(t^2\log\log(1/t)^2)\\
&\leq2\lambda_1W_1(t)^2+2\lambda_2W_2(t)^2+\lambda_1\langle e_1,\,\Sigma(x)^{-1}E(t)\rangle^2+\lambda_2\langle e_2,\,\Sigma(x)^{-1}E(t)\rangle^2+\underset{t\rightarrow0}{O}(t^2\log\log(1/t)^2)\;,
\end{align*}
using the inequality $2ab\leq a^2+b^2$ for $a,\,b\geq0$. Therefore, by~\eqref{eq:difference M(t) Sigma B(t)},
$$\langle\mathbb{S}\,M(t),\,M(t)\rangle\leq2\lambda_1W_1(t)^2+2\lambda_2W_2(t)^2+\underset{t\rightarrow0}{O}(t^2\log\log(1/t)^2)\;.$$
Additionally, by~\eqref{eq:difference M(t) Sigma B(t)}, we also have that
\begin{align*}
&\langle\mathbb{T}\,M(t),\,M(t),\,M(t)\rangle+t\langle\alpha,\,M(t)\rangle+\int_0^t\langle\theta,\,M(s)\rangle\mathrm{d}s\\
&=\langle\mathbb{T}\,\Sigma(x)B(t),\,\Sigma(x)B(t),\,\Sigma(x)B(t)\rangle+t\langle\alpha,\,\Sigma(x)B(t)\rangle+\int_0^t\langle\theta,\,\Sigma(x)B(s)\rangle\mathrm{d}s+\underset{t\rightarrow0}{O}(t^2\log\log(1/t)^2)\\
&=\langle\widetilde{\mathbb{T}}\,W(t),\,W(t),\,W(t)\rangle+t\langle\widetilde{\alpha},\,W(t)\rangle+\int_0^t\langle\widetilde{\theta},\,W(s)\rangle\mathrm{d}s+\underset{t\rightarrow0}{O}(t^2\log\log(1/t)^2)\;,
\end{align*}
where, for $1\leq i,\,j,\,k\leq d$,
\begin{align*}
\widetilde{\mathbb{T}}_{i,\,j,\,k}&=\langle\mathbb{T}\,\Sigma(x)e_i,\,\Sigma(x)e_j,\,\Sigma(x)e_k\rangle,\\
\widetilde{\alpha}_i&=\langle\Sigma(x)^\dagger\alpha,\,e_i\rangle,\\
\widetilde{\theta}_i&=\langle\Sigma(x)^\dagger\theta,\,e_i\rangle\;.
\end{align*}
Notice that since $\mathbb{T}$ is symmetric, $\widetilde{\mathbb{T}}$ is symmetric as well. Also, since $\theta\neq0$ then $\widetilde{\theta}\neq0$ as well. Hence the proof of~\eqref{eq:upper-bound brownian terms}. 

\textbf{Step 2}: By~\eqref{eq:liminf Brownian 2d} in Lemma~\ref{lem:liminf BM recurrent}, almost-surely, there exists a sequence $(\tau_n)_{n\geq0}$ of positive times decreasing to zero such that for any $\rho \geq 0$
\begin{equation}\label{eq:decreasing W 2d}
|W_1(\tau_n)|^2+|W_2(\tau_n)|^2=\underset{n\rightarrow\infty}{o}(\tau_n^{\rho})\;.
\end{equation}
Let $\widehat{T}\in\R^{(d-2)\times(d-2)\times(d-2)}$, $\widehat{\alpha}\in\R^{d-2}$ and $(\widehat{W}(t))_{t\geq0}$ in $\R^{d-2}$ be the symmetric tensor, vector and Brownian motion respectively, obtained from $\widetilde{\mathbb{T}}$, $\widetilde{\alpha}$ and $(W(t))_{t\geq0}$ by removing the coordinates associated with the indices in $\{1,\,2\}$. Note that if $d\leq2$, these objects are assumed to vanish. 

Let us now assume by contradiction that the conclusion of the present proposition is false: there exists $\epsilon>0$ such that with a $\mathbb{P}_x$-positive probability~\eqref{eq:liminf negative} does not hold. Blumenthal's zero-one law for the Brownian motion (recall that $t\geq0 \mapsto X(t)$ is a strong solution of the equation~\eqref{eq:sde}) ensures that the probability of the event~\eqref{eq:liminf negative} is in fact $0$. We now prove as a consequence that there exists $\rho>0$ such that, almost-surely,
\begin{equation}\label{eq:limit to zero remaining Brownian terms}
\frac{\log(1/\tau_n)^{3/2+\rho}}{\tau_n^{3/2}}\left[\langle\widehat{\mathbb{T}}\widehat{W}(\tau_n),\,\widehat{W}(\tau_n),\,\widehat{W}(\tau_n)\rangle+\tau_n\langle\widehat{\alpha},\,\widehat{W}(\tau_n)\rangle+\int_0^{\tau_n}\langle\widetilde{\theta},\,W(s)\rangle\mathrm{d}s\right]=\underset{n\rightarrow\infty}{o}(1)\;.
\end{equation}

Indeed, we can deduce from the fact that~\eqref{eq:liminf negative} has zero probability and from~\eqref{eq:upper-bound brownian terms} the existence of $\lambda>0$ such that, almost-surely,
$$\liminf_{t\rightarrow0}\frac{\log(1/t)^{3/2+\epsilon}}{t^{3/2}}\left[\lambda (W_1(t)^2+W_2(t)^2)+\langle\widetilde{\mathbb{T}}\,W(t),\,W(t),\,W(t)\rangle+t\langle\widetilde{\alpha},\,W(t)\rangle+\int_0^t\langle\widetilde{\theta},\,W(s)\rangle\mathrm{d}s\right]>-\infty$$
and by the symmetry property of the Brownian motion $(W(t))_{t\geq0}$,
$$\liminf_{t\rightarrow0}\frac{\log(1/t)^{3/2+\epsilon}}{t^{3/2}}\left[\lambda (W_1(t)^2+W_2(t)^2)-\langle\widetilde{\mathbb{T}}\,W(t),\,W(t),\,W(t)\rangle-t\langle\widetilde{\alpha},\,W(t)\rangle-\int_0^t\langle\widetilde{\theta},\,W(s)\rangle\mathrm{d}s\right]>-\infty\;.$$
As a result, for any $\rho\in(0,\,\epsilon)$, almost-surely,
\begin{equation}\label{eq:liminf pm}
\liminf_{t\rightarrow0}\frac{\log(1/t)^{3/2+\rho}}{t^{3/2}}\left[\lambda (W_1(t)^2+W_2(t)^2)\pm\langle\widetilde{\mathbb{T}}\,W(t),\,W(t),\,W(t)\rangle\pm t\langle\widetilde{\alpha},\,W(t)\rangle\pm\int_0^t\langle\widetilde{\theta},\,W(s)\rangle\mathrm{d}s\right]\geq0\;.
\end{equation}
Given that for any $k\in\llbracket3,\,d\rrbracket$, the Brownian terms $W_k(\tau_n)$ are of order at most $\tau_n^{1/2}\log\log(1/\tau_n)^{1/2}$ by the law of the iterated logarithm, one deduces from~\eqref{eq:decreasing W 2d} that
\begin{align*}
&\langle\widetilde{\mathbb{T}}\,W(\tau_n),\,W(\tau_n),\,W(\tau_n)\rangle+\tau_n\langle\widetilde{\alpha},\,W(\tau_n)\rangle+\int_0^{\tau_n}\langle\widetilde{\theta},\,W(s)\rangle\mathrm{d}s\\
&=\langle\widehat{\mathbb{T}}\,\widehat{W}(\tau_n),\,\widehat{W}(\tau_n),\,\widehat{W}(\tau_n)\rangle+\tau_n\langle\widehat{\alpha},\,\widehat{W}(\tau_n)\rangle+\int_0^{\tau_n}\langle\widetilde{\theta},\,W(s)\rangle\mathrm{d}s+\underset{n\rightarrow\infty}{O}(\tau_n^2\log\log(1/\tau_n))\;.
\end{align*}
Consequently, considering the sequence $(\tau_n)_{n\geq0}$ in~\eqref{eq:liminf pm} we deduce that
$$\liminf_{n\rightarrow\infty}\pm\frac{\log(1/\tau_n)^{3/2+\rho}}{\tau_n^{3/2}}\left[\langle\widehat{\mathbb{T}}\,\widehat{W}(\tau_n),\,\widehat{W}(\tau_n),\,\widehat{W}(\tau_n)\rangle+\tau_n\langle\widehat{\alpha},\,\widehat{W}(\tau_n)\rangle+\int_0^{\tau_n}\langle\widetilde{\theta},\,W(s)\rangle\mathrm{d}s\right]\geq0$$
hence~\eqref{eq:limit to zero remaining Brownian terms}. The rest of the proof is devoted to proving that the limit in~\eqref{eq:limit to zero remaining Brownian terms} can be false with a positive probability, which will conclude the proof by a contradiction argument.

\textbf{Step 3}: In this step, we assume ~\eqref{eq:limit to zero remaining Brownian terms} almost surely holds and isolate terms involving a one-dimensional Brownian motion in order to simplify the expression. We then proceed to obtain consequences that will turn out to be untrue.

Following the definition in~\eqref{eq:decreasing W 2d}, the sequence $(\tau_n)_{n\geq0}$ can be constructed independently of the sign of the coordinates of the Brownian motion $(W(t))_{t\geq0}$ (conditional on their absolute values). In particular, the fact that the limit~\eqref{eq:limit to zero remaining Brownian terms} is assumed to be almost surely true can be applied to any $\R^d$-Brownian motion $(W^{r}(t))_{t\geq0}$ deduced from $(W(t))_{t\geq0}$ as follows
$$\forall k\in\llbracket1,\,d\rrbracket\,,\qquad W_k^{r}:t\geq0\mapsto\left\{\begin{aligned}
   &-W_k(t)\,, &&\qquad k\neq r\;,\\
    &W_{r}(t)\,,&&\qquad k=r\;.
  \end{aligned}\right.
$$

Also, since $\widetilde{\theta}\neq0$, there exists $j\in\llbracket1,\,d\rrbracket$ such that $$\widetilde{\theta}_{j}\neq0.$$ Assume first that $j\in\{1,\,2\}$. Since the first two coordinates of $(W(t))_{t\geq0}$ do not appear in the process $(\widehat{W}(t))_{t\geq0}$, we obtain by summing the quantity~\eqref{eq:limit to zero remaining Brownian terms} for both Brownian motions $(W(t))_{t\geq0}$ and $(W^{j}(t))_{t\geq0}$ that
\begin{equation}\label{eq:lim integrale W_j}
\frac{\log(1/\tau_n)^{3/2+\rho}}{\tau_n^{3/2}}\int_0^{\tau_n}W_{j}(s)\mathrm{d}s=\underset{n\rightarrow\infty}{o}(1)\;.
\end{equation}
However, by Lachal's result recalled in~\eqref{eq:cv lachal int BM}, almost-surely,
$$\lim_{n\rightarrow\infty}\frac{\log(1/\tau_n)^{3/2+\rho}}{\tau_n^{3/2}}\sqrt{|W_j(\tau_n)|^2+\left|\int_0^{\tau_n}W_j(s)\mathrm{d}s\right|^2}=+\infty\;.$$
Therefore, it follows from~\eqref{eq:decreasing W 2d} that, almost-surely,
\begin{equation}\label{eq:lim integrale Brownian subsequence}
\lim_{n\rightarrow\infty}\frac{\log(1/\tau_n)^{3/2+\rho}}{\tau_n^{3/2}}\left|\int_0^{\tau_n}W_j(s)\mathrm{d}s\right|=+\infty
\end{equation}
which contradicts~\eqref{eq:lim integrale W_j}.

We now assume that $j\in\llbracket3,\,d\rrbracket$. Define now the process
\begin{equation}\label{eq:expr process V}
V:t\geq0\mapsto 3\sum_{i,\,k\in\llbracket3,\,d\rrbracket\setminus\{j\}}\widehat{\mathbb{T}}_{j,\,i,\,k}W_i(t)W_k(t)+\widehat{\alpha}_{j}t\;.
\end{equation}
The summation of~\eqref{eq:limit to zero remaining Brownian terms} for both Brownian motions $(W(t))_{t\geq0}$ and $(W^{j}(t))_{t\geq0}$ in the case $j\in\llbracket3,\,d\rrbracket$ yields that, almost-surely,
\begin{equation}\label{eq:limit to zero remaining terms}
\frac{\log(1/\tau_n)^{3/2+\rho}}{\tau_n^{3/2}}\left(W_{j}(\tau_n)\left[V(\tau_n)+\widehat{\mathbb{T}}_{j,\,j,\,j}|W_{j}(\tau_n)|^2\right]+\widetilde{\theta}_{j}\int_0^{\tau_n}W_{j}(s)\mathrm{d}s\right)=\underset{n\rightarrow\infty}{o}(1)\;,
\end{equation}
since $\widehat{\mathbb{T}}$ is a symmetric tensor. We perform the following co-variant time change for each $i \geq 3$:
\[
\widetilde{ W^n_i}(l )  := \sqrt{1/\tau_n} W_{i}( l \tau_n)
\]
whose joint marginal distributions is again a triplet of unit Brownian motions for any $n$ since the sequence $(\tau_n)_{n \geq 1}$ and the Brownian motions $W^n_i$ for $ i \geq 3$ are independent by construction. Denoting $\widetilde V^n$ the analogous of the expression~\eqref{eq:expr process V} obtained with these modified $\widetilde W _i $, ~\eqref{eq:limit to zero remaining terms} can be re-written after a routine time re-scaling:
\begin{equation*}
\log(1/\tau_n)^{3/2+\rho} \left(\widetilde W^n_{j}(1)\left[\widetilde V^n(1)+\widehat{\mathbb{T}}_{j,\,j,\,j}|\widetilde W^n_{j}(1)|^2\right]+\widetilde{\theta}_{j}\int_0^{1} \widetilde W^n_{j}(s)\mathrm{d}s\right)=\underset{n\rightarrow\infty}{o}(1)\;,
\end{equation*}
which implies that the quantity $\widetilde W^n_{j}(1)\left[\widetilde V^n(1)+\widehat{\mathbb{T}}_{j,\,j,\,j}|\widetilde W^n_{j}(1)|^2\right]+\widetilde{\theta}_{j}\int_0^{1} \widetilde W^n_{j}(s)\mathrm{d}s $, whose distribution is independent of $n$ tends to $0$ and thus is equal to $0$ for all $n$. Now, since $\widetilde V^n_j$ is independent of $\widetilde W^n_j$, and $\widetilde W^n_{j}(s)\mathrm{d}s$ has non-zero variance conditional on $\widetilde W^n_j(1) $, we obtain a contradiction with the fact that $\tilde \theta_j \neq 0$.
\end{proof}

Let us now conclude this section with the proof of Theorem~\ref{thm:reg points}. The proof relies mainly on the small-time asymptotics obtained in Theorem~\ref{thm:small time asymp distance} and in Proposition~\ref{prop:liminf bessel int BM}. However, Theorem~\ref{thm:small time asymp distance} only applies to the case $G=I_d$. For this purpose, we define
\begin{equation}\label{eq:def Phi}
\Phi:(q,\,p)\in\R^{2d}\mapsto (q,\,G(p))\;.
\end{equation}
We ought to place ourselves in the setting $G=I_d$ by considering the stochastic process $(\Phi(X(t)))_{t\geq0}$. Using~\eqref{eq:sde} and It\^o's formula, we can show that the process $(\Phi(X(t)))_{t\geq0}$ satisfies the same SDE as~\eqref{eq:sde} with coefficients $\widetilde{G},\,\widetilde{F},\,\widetilde{\Sigma}$ defined for $(q,\,p)\in\R^{2d}$ by
\begin{equation}\label{eq:expr coeff X tilde}
\left\{ \begin{aligned} \widetilde{G}(p)&=p\;,\\
 \widetilde{F}(q,\,p)&=DG(G^{-1}(p))F(q,\,G^{-1}(p))+\frac{1}{2}\mathrm{Tr}\big[D^2G(G^{-1}(p))\Sigma(q,\,G^{-1}(p))\Sigma^\dagger(q,\,G^{-1}(p))\big]\;,\\
 \widetilde{\Sigma}(q,\,p)&=DG(G^{-1}(p))\Sigma(q,\,G^{-1}(p))\;.
\end{aligned}
\right.
\end{equation}
In particular, since $G$ is a $C^\infty$ diffeomorphism in $\R^d$, these coefficients verify Assumption~\ref{ass:coeff}. Let us now define
\begin{equation}\label{eq:distance X tilde}
\widetilde{\mathrm{d}_\Omega}:(q,\,p)\in\R^{2d}\mapsto\mathrm{d}_\Omega(q,\,G^{-1}(p))
\end{equation}
and
$$\widetilde{M}:t\geq0\mapsto\int_0^t\widetilde{\Sigma}(\Phi(X(s)))\mathrm{d}B(s)\;.$$

\begin{proof}[Proof of Theorem~\ref{thm:reg points}] Let us fix $x=(q,\,p)\in\partial\Omega$. By definition, $\mathbb{P}_x$ almost-surely, for all $t\geq0$,
\begin{equation}\label{eq:distance processes X X_tilde}
\mathrm{d}_\Omega(X(t))=\widetilde{\mathrm{d}_\Omega}(\Phi(X(t)))\;.
\end{equation} 
The point $x$ is regular if there exists a positive function $\Psi$ on $\R_+^*$ such that, $\mathbb{P}_x$ almost-surely,
\begin{equation}\label{eq:liminf modified dist}
\liminf_{t\rightarrow0}\frac{\mathrm{d}_\Omega(X(t))}{\Psi(t)}=\liminf_{t\rightarrow0}\frac{\widetilde{\mathrm{d}_\Omega}(\Phi(X(t)))}{\Psi(t)}<0\;.
\end{equation}
If the liminf above is positive instead, then $x$ is irregular. However, no conclusion can be drawn when the liminf above vanishes. Therefore, the goal is to construct a function $\Psi$ such that the liminf above is non-zero. For this purpose we heavily rely on the small-time expansion obtained in Theorem~\ref{thm:small time asymp distance}.

Given that the liminf in~\eqref{eq:liminf modified dist} only depends on the values of $\widetilde{\mathrm{d}_\Omega}$ in any small neighborhood of $\Phi(x)$, up to extending $\widetilde{\mathrm{d}_\Omega}$ outside of this neighborhood, we can assume that $\widetilde{\mathrm{d}_\Omega}\in C^{3,\,1}(\R^{2d})$ given the smoothness of $\Omega$ in Assumption~\ref{ass:Omega} and~\cite[Lemma 14.16]{GT}. This allows us to apply Theorem~\ref{thm:small time asymp distance} to the function $\widetilde{\mathrm{d}_\Omega}$ and the process $(\Phi(X(t)))_{t\geq0}$. Finally, notice that by Lemma~\ref{lem:holderian martingale}, $\mathbb{P}_x$ almost-surely,
\begin{equation}\label{eq:DL M(t)}
\big|\widetilde{M}(t)-\widetilde{\Sigma}(\Phi(x))B(t)\big|=\big|\widetilde{M}(t)-DG(p)\Sigma(x)B(t)\big|=\underset{t\rightarrow0}{O}(t\log\log(1/t))\;.
\end{equation}

\underline{\textbf{Case} $x\in\mathrm{T}$}. By~\eqref{eq:distance processes X X_tilde},
$$\widetilde{\mathrm{d}_\Omega}(\Phi(x))=\mathrm{d}_{\Omega}(x)=0\;.$$
Additionally, given that $x\in\mathrm{T}$, we deduce from~\eqref{eq:normal gradient Psi} that
\begin{equation}\label{eq:norm grad p distance case 1}
|\nabla_p\mathrm{d}_\Omega(x)|>0  
\end{equation}
Furthermore, the definition~\eqref{eq:distance X tilde} ensures that
\begin{equation}\label{eq:gradient p distance tilde}
\nabla_p\widetilde{\mathrm{d}_\Omega}(\Phi(x))=DG^{-1}(G(p))^{\dagger}\nabla_p\mathrm{d}_\Omega(x)=DG(p)^{-\dagger}\nabla_p\mathrm{d}_\Omega(x)
\end{equation}
Therefore, applying Theorem~\ref{thm:small time asymp distance} to the function $\widetilde{\mathrm{d}_\Omega}$ and the process $(\Phi(X(t)))_{t\geq0}$, we deduce from~\eqref{eq:distance processes X X_tilde} and~\eqref{eq:DL M(t)} that, $\mathbb{P}_x$ almost-surely,
\begin{align*}
\mathrm{d}_\Omega(X(t))&=\langle\nabla_p\widetilde{\mathrm{d}_\Omega}(\Phi(x)),\,\widetilde{M}(t)\rangle+\underset{t\rightarrow0}{O}(t\log\log(1/t))\\
&=\langle DG(p)^{-\dagger}\nabla_p\mathrm{d}_\Omega(x),\,DG(p)\Sigma(x)B(t)\rangle+\underset{t\rightarrow0}{O}(t\log\log(1/t))
\end{align*}
By the Law of the Iterated Logarithm, $\mathbb{P}_x$ almost-surely,
\begin{align*}
\liminf_{t\rightarrow0}\frac{\mathrm{d}_\Omega(X(t))}{\sqrt{2t\log\log(1/t)}}&=\liminf_{t\rightarrow0}\frac{\langle \Sigma(x)^\dagger \nabla_p\mathrm{d}_\Omega(x),\,B(t)\rangle}{\sqrt{2t\log\log(1/t)}}\nonumber\\
&=-|\Sigma(x)^\dagger \nabla_p\mathrm{d}_\Omega(x)|\;.
\end{align*}

Therefore, we deduce from~\eqref{eq:norm grad p distance case 1} that the liminf above is negative, hence $x$ is regular.

\underline{\textbf{Case} $x\in\Upsilon$}. In this case, given that $\nabla_p\mathrm{d}_\Omega(x)=~0$, we deduce from~\eqref{eq:gradient p distance tilde} that
$$\nabla_p\widetilde{\mathrm{d}_\Omega}(\Phi(x))=0\;.$$ Consequently, differentiating~\eqref{eq:distance X tilde} two times with respect to $p$, we obtain that
$$\nabla^2_{p,\,p}\widetilde{\mathrm{d}_\Omega}(\Phi(x))=DG(p)^{-\dagger}\nabla^2_{p,\,p}\mathrm{d}_\Omega(x)DG(p)^{-1}\;.$$
As a result, Theorem~\ref{thm:small time asymp distance} along with~\eqref{eq:DL M(t)} ensure that, $\mathbb{P}_x$ almost-surely,
\begin{align*}
\mathrm{d}_\Omega(X(t))&=\frac{1}{2}\langle\nabla^2_{p,\,p}\widetilde{\mathrm{d}_\Omega}(\Phi(x))\widetilde{M}(t),\,\widetilde{M}(t)\rangle+\underset{t\rightarrow0}{O}(t)\\
&=\frac{1}{2}\langle\Sigma(x)^\dagger DG(p)^\dagger\nabla^2_{p,\,p}\widetilde{\mathrm{d}_\Omega}(\Phi(x))DG(p)\Sigma(x) B(t),\,B(t)\rangle+\underset{t\rightarrow0}{O}(t)\\
&=\frac{1}{2}\langle\Sigma(x)^\dagger \nabla^2_{p,\,p}\mathrm{d}_{\Omega}(x)\Sigma(x) B(t),\,B(t)\rangle+\underset{t\rightarrow0}{O}(t)\;.
\end{align*}
Since $\Sigma(x)^\dagger \nabla^2_{p,\,p}\mathrm{d}_{\Omega}(x)\Sigma(x)$ and $\nabla^2_{p,\,p}\mathrm{d}_\Omega(x)$ are symmetric congruent matrices, they have the same rank and the same number of negative eigenvalues. Therefore, since $x\in\Upsilon$, both matrices admit at least one negative eigenvalue. Denote now by $z$ a unitary eigenvector associated with the smallest negative eigenvalue $\lambda_{\min}(\Sigma(x)^\dagger \nabla^2_{p,\,p}\mathrm{d}_{\Omega}(x)\Sigma(x))$. By~\cite[Theorem 1]{gantert1993inversion}, $z$ belongs almost-surely to the set of limit points of $B(t)/\sqrt{2t\log\log(1/t)}$ when $t$ goes to zero. As a result,
\begin{align*}
&\liminf_{t\rightarrow0}\frac{\mathrm{d}_\Omega(X(t))}{2t\log\log(1/t)}\\
&=\frac{1}{2}\liminf_{t\rightarrow0}\left\langle\Sigma(x)^\dagger \nabla^2_{p,\,p}\mathrm{d}_{\Omega}(x)\Sigma(x)\frac{B(t)}{\sqrt{2t\log\log(1/t)}},\,\frac{B(t)}{\sqrt{2t\log\log(1/t)}}\right\rangle\\
&=\frac{1}{2}\lambda_{\min}(\Sigma(x)^\dagger \nabla^2_{p,\,p}\mathrm{d}_{\Omega}(x)\Sigma(x))<0\,,
\end{align*}
hence $x\in\mathcal{R}$.

\underline{\textbf{Case} $x\in\Gamma^-\cup\Gamma^+$}. Differentiating~\eqref{eq:distance X tilde} with respect to the $q$-coordinate ensures that
\begin{equation}\label{eq:grad q distance tilde case 3}
\nabla_q\widetilde{\mathrm{d}_\Omega}(\Phi(x))=\nabla_q\mathrm{d}_\Omega(x)\;.
\end{equation}
Therefore, Theorem~\ref{thm:small time asymp distance} and the same arguments deployed in the case $x\in\Upsilon$ yield that, $\mathbb{P}_x$ almost-surely,
\begin{equation}\label{eq:developt asymptot Gamma}
\mathrm{d}_\Omega(X(t))=\frac{1}{2}\langle\Sigma(x)^\dagger \nabla^2_{p,\,p}\mathrm{d}_{\Omega}(x)\Sigma(x) B(t),\,B(t)\rangle+t\xi(x)+\underset{t\rightarrow0}{O}(t^{3/2}\log\log(1/t)^{3/2})
\end{equation}
where $\xi(x)$ is defined in~\eqref{eq:def xi}. Also, since $x\in\Gamma^-\cup\Gamma^+$, the matrix $\nabla^2_{p,\,p}\mathrm{d}_\Omega(x)$ admits only non-negative eigenvalues and so does its congruent matrix $\Sigma(x)^\dagger \nabla^2_{p,\,p}\mathrm{d}_{\Omega}(x)\Sigma(x)$. Consequently,
\begin{align*}
\liminf_{t\rightarrow0}\frac{\mathrm{d}_\Omega(X(t))}{t}&=\liminf_{t\rightarrow0}\frac{\langle\Sigma(x)^\dagger \nabla^2_{p,\,p}\mathrm{d}_{\Omega}(x)\Sigma(x) B(t),\,B(t)\rangle}{2t}+\xi(x)\\
&=\xi(x)\,,
\end{align*}
since by Lemma~\ref{lem:liminf BM recurrent},
$$0\leq\liminf_{t\rightarrow0}\frac{\langle \Sigma(x)^\dagger \nabla^2_{p,\,p}\mathrm{d}_{\Omega}(x)\Sigma(x)B(t),\,B(t)\rangle}{2t}\leq\Vert\Sigma(x)^\dagger \nabla^2_{p,\,p}\mathrm{d}_{\Omega}(x)\Sigma(x)\Vert\,\liminf_{t\rightarrow0}\frac{|B(t)|^2}{2t}=0\;.$$ Therefore, $x$ is regular if $\xi(x)<0$ and irregular if $\xi(x)>0$.

\underline{\textbf{Case} $x\in\Gamma^{0,\,i}$}. It holds that
$$\mathrm{rank}\big(\Sigma(x)^\dagger \nabla^2_{p,\,p}\mathrm{d}_{\Omega}(x)\Sigma(x)\big)=\mathrm{rank}\big(\nabla^2_{p,\,p}\mathrm{d}_{\Omega}(x)\big)\geq3\;.$$
Therefore, there exists an orthonormal family of vectors $(e_1,\,e_2,\,e_3)$ in $\R^d$ associated with eigenvalues $\lambda_1\geq\lambda_2\geq\lambda_3>0$ of $\Sigma(x)^\dagger \nabla^2_{p,\,p}\mathrm{d}_{\Omega}(x)\Sigma(x)$. As a result, Theorem~\ref{thm:small time asymp distance} ensures that, $\mathbb{P}_x$ almost-surely,
\begin{align*}
\mathrm{d}_\Omega(X(t))&=\frac{1}{2}\langle\Sigma(x)^\dagger \nabla^2_{p,\,p}\mathrm{d}_{\Omega}(x)\Sigma(x)B(t),\,B(t)\rangle+\underset{t\rightarrow0}{O}(t^{3/2}\log\log(1/t)^{3/2})\;.\\
&\geq\frac{\lambda_3}{2}\sum_{i=1}^3\langle B(t),\,e_i\rangle^2+\underset{t\rightarrow0}{O}(t^{3/2}\log\log(1/t)^{3/2})\;.
\end{align*}
Since the processes $(\langle B(t),\,e_i\rangle)_{t\geq0}$ are independent one-dimensional Brownian motions for $i\in\llbracket1,\,3\rrbracket$, we deduce by Lemma~\ref{lem:liminf BM recurrent} that for any $\epsilon>0$,
\begin{align*}
&\liminf_{t\rightarrow0}(\log1/t)^{2+\epsilon}\frac{\mathrm{d}_\Omega(X(t))}{t}\\
&\geq\frac{\lambda_3}{2}\,\liminf_{t\rightarrow0}(\log1/t)^{2+\epsilon}\frac{\sum_{i=1}^3\langle B(t),\,e_i\rangle^2}{t}=+\infty,
\end{align*}
hence $x$ is irregular.

\underline{\textbf{Case} $x\in\Gamma^{0,\,r}$}. In this case,
$$\mathrm{rank}\big(\Sigma(x)^\dagger \nabla^2_{p,\,p}\mathrm{d}_{\Omega}(x)\Sigma(x)\big)=\mathrm{rank}\big(\nabla^2_{p,\,p}\mathrm{d}_{\Omega}(x)\big)\in\llbracket0,\,2\rrbracket\;.$$
Also, by~\eqref{eq:gradient p distance tilde} and~\eqref{eq:grad q distance tilde case 3},
$$|\nabla_q\widetilde{\mathrm{d}_\Omega}(\Phi(x))+D_p\widetilde{F}(\Phi(x))^\dagger\nabla_p\widetilde{\mathrm{d}_\Omega}(\Phi(x))|=|\nabla_q\mathrm{d}_\Omega(x)|=1$$
since the Eikonal equation ensures that $|\nabla_q\mathrm{d}_\Omega(x)|^2+|\nabla_p\mathrm{d}_\Omega(x)|^2=1$. Therefore, an immediate application of Theorem~\ref{thm:small time asymp distance} and Proposition~\ref{prop:liminf bessel int BM} yields that $x$ is regular.
\end{proof}

\section{Polarity and exit event: proof of Theorems~\ref{thm:polarity} and~\ref{thm:polarity Gamma+}}\label{sec: polarity} In Section~\ref{sec:prelim results polarity}, we provide a precise upper-bound on a surface measure which appears later in Section~\ref{sec:thm polarity Gamma0} for the proof of Theorem~\ref{thm:polarity}. In Section~\ref{sec:proof thm polarity Gamma +} we prove Theorem~\ref{thm:polarity Gamma+}.
\subsection{Preliminary results}\label{sec:prelim results polarity}
This first lemma provides a boundary identity which is a consequence of the Eikonal equation.
\begin{lem}[Boundary identity]\label{lem:signed distance}
For all $x\in\partial\Omega$,
$$\nabla^2_{p,\,q}\mathrm{d}_{\Omega}(x)\,n_q(x)+\nabla^2_{p,\,p}\mathrm{d}_{\Omega}(x)\,n_p(x)=0\;.$$
\end{lem}
\begin{proof} The Eikonal equation ensures that for any point $x$ sufficiently close to $\partial\Omega$,
$$ \vert\nabla_q\mathrm{d}_{\Omega}(x)\vert^2+\vert\nabla_p\mathrm{d}_{\Omega}(x)\vert^2 = 1\;.$$
Differentiating the equality above with respect to the $p$-coordinates at $x\in\partial\Omega$ immediately concludes the proof.
\end{proof}
The following proposition provides a control on the surface measure of a subset of $\partial\Omega$ which appears later in the proof of Theorem~\ref{thm:polarity}.

\begin{prop}[Surface measure]\label{prop:control neighborhood Gamma_0}
Let $\rho\in(0,\,1/3)$ and let $\theta,\,M>0$. There exists a constant $C>0$ such that for all $\epsilon>0$ small enough,
\begin{equation}\label{eq:upper bound neighborhood tangential}
\int_{\partial\Omega\cap\mathrm{B}(0,\,M)}\mathbf{1}_{\Vert\nabla^2_{p,\,p}\mathrm{d}_{\Omega}(x)\Vert\geq\epsilon^\rho}\mathbf{1}_{|n_p(x)|\leq\theta\epsilon}\,\sigma_{2d-1}(\mathrm{d}x)\leq C\epsilon^{1-3\rho}
\end{equation}
where $\sigma_{2d-1}$ is the $(2d-1)$-dimensional Lebesgue surface measure. Furthermore,
\begin{equation}\label{eq:surface meas Q}
\sigma_{2d-1}\big(\big\{x\in\partial\Omega:|n_p(x)|=0,\,\nabla^2_{p,\,p}\mathrm{d}_\Omega(x)\neq\mathbb{O}_{d}\big\}\big)=0\;.
\end{equation}
\end{prop}
\begin{proof} Let $M>0$ and let
$$\partial\Omega_M:=\partial\Omega\cap\mathrm{B}(0,\,M)\;.$$
Since the signed distance function $\mathrm{d}_{\Omega}$ is $C^{3,\,1}$ on $\partial\Omega_M$, there exists a constant $C_{\mathrm{Lip}}>0$ such that for all $x,\,y\in\partial\Omega_M$,
\begin{equation}\label{eq:Lipschitz hess distance}
\left\Vert\nabla^2_{p,\,p}\mathrm{d}_{\Omega}(x)-\nabla^2_{p,\,p}\mathrm{d}_{\Omega}(y)\right\Vert\leq C_{\mathrm{Lip}}|x-y|\;.
\end{equation}
Given that $\partial\Omega_M$ is bounded, it can be covered by a finite number $N_\epsilon$ of balls of radius $\epsilon^\rho/4C_{\mathrm{Lip}}$. Besides, since $\partial\Omega_M$ is Lipschitz continuous there exists a constant $\beta>0$ independent of $\epsilon>0$ when $\epsilon$ is small enough such that
\begin{equation}\label{eq:nbr of balls}
N_\epsilon\leq\frac{\beta}{\epsilon^{\rho(2d-1)}}\;.
\end{equation}
Denote by $z_j$ for $j\in\llbracket1,\,N_\epsilon\rrbracket$ the center of the corresponding balls and let
$$V_j:=\mathrm{B}(z_j,\,\epsilon^\rho/4C_{\mathrm{Lip}})\;.$$
Let us also define a $C^\infty$ non-negative function $\psi$ with values in $[0,\,1]$ and satisfying
\begin{equation}\label{eq:def K}
\psi:\lambda\in\R_+\mapsto\left\{\begin{aligned}
    1, & \qquad \lambda\leq1\\
    0, &\qquad \lambda\geq4.
  \end{aligned}\right. 
\end{equation}
and let
\begin{equation}\label{eq:def expr Psi_j}
\Psi_j:x\in\R^{2d}\mapsto\psi\bigg(\frac{|x-z_j|^2}{\epsilon^{2\rho}/(4C_{\mathrm{Lip}})^2}\bigg)\;.
\end{equation}
The elements above ensure that the following inequalities hold
\begin{align}
\int_{\partial\Omega_M}\mathbf{1}_{\Vert\nabla^2_{p,\,p}\mathrm{d}_\Omega(x)\Vert\geq\epsilon^\rho}\mathbf{1}_{|n_p(x)|\leq\theta\epsilon}\,\sigma_{2d-1}(\mathrm{d}x)&\leq\sum_{j=1}^{N_\epsilon}\int_{\partial\Omega_M\cap V_j}\mathbf{1}_{\Vert\nabla^2_{p,\,p}\mathrm{d}_\Omega(x)\Vert\geq\epsilon^\rho}\mathbf{1}_{|n_p(x)|\leq\theta\epsilon}\,\sigma_{2d-1}(\mathrm{d}x)\nonumber\\
&\leq\sum_{j=1}^{N_\epsilon}\int_{\partial\Omega_M\cap V_j}\Psi_j(x)\mathbf{1}_{\Vert\nabla^2_{p,\,p}\mathrm{d}_\Omega(x)\Vert\geq\epsilon^\rho}\mathbf{1}_{|n_p(x)|\leq\theta\epsilon}\,\sigma_{2d-1}(\mathrm{d}x)\label{eq:boundary integral partitioning}
\end{align}
since $\Psi_j$ is non-negative and $\Psi_j=1$ on $V_j$.

\textbf{Step 1}: Let us prove that for any $j\in\llbracket1,\,N_\epsilon\rrbracket$, if $\Vert\nabla^2_{p,\,p}\mathrm{d}_\Omega(\widehat{x})\Vert\geq\epsilon^\rho$ for some $\widehat{x}\in\partial\Omega_M\cap V_j$, then there exists a unitary vector $v_j\in\R^d$ such that for all $x\in\partial\Omega_M\cap V_j$,
\begin{equation}\label{eq:lower bound hessian velocity v_j}
|\nabla^2_{p,\,p}\mathrm{d}_\Omega(x)v_j|\geq\epsilon^\rho/2\;.
\end{equation}

Suppose the existence of $j\in\llbracket1,\,N_\epsilon\rrbracket$ and $\widehat{x}\in\partial\Omega_M\cap V_j$ such that $\Vert\nabla^2_{p,\,p}\mathrm{d}_\Omega(\widehat{x})\Vert\geq\epsilon^\rho$. Then, there exists a unitary vector $v_j(\widehat{x})\in\R^d$ such that $|\nabla^2_{p,\,p}\mathrm{d}_\Omega(\widehat{x})v_j(\widehat{x})|\geq\epsilon^\rho$. Additionally, it follows from~\eqref{eq:Lipschitz hess distance} that for any $x\in\partial\Omega_M\cap V_j$,
\begin{align*}
|\nabla^2_{p,\,p}\mathrm{d}_\Omega(x)v_j(\widehat{x})|&\geq|\nabla^2_{p,\,p}\mathrm{d}_\Omega(\widehat{x})v_j(\widehat{x})|-\Vert\nabla^2_{p,\,p}\mathrm{d}_\Omega(\widehat{x})-\nabla^2_{p,\,p}\mathrm{d}_\Omega(x)\Vert\,|v_j(\widehat{x})|\\
&\geq\frac{\epsilon^\rho}{2}\;,
\end{align*}
since $V_j$ is a ball of diameter $\epsilon^\rho/2C_{\mathrm{Lip}}$, hence the proof of Step 1.

\textbf{Step 2}: Let $j\in\llbracket1,\,N_\epsilon\rrbracket$ satisfying the assumption of Step 1. Let us define the function 
\begin{equation}\label{eq:expr g_j eps}
g^{(\epsilon)}_j:x\in\partial\Omega\mapsto\frac{\left\langle n_p(x),\,v_j\right\rangle}{\epsilon^\rho}\;.
\end{equation}
Our objective in this step is to apply the coarea formula to the integrals appearing in the summation~\eqref{eq:boundary integral partitioning} using the function $g^{(\epsilon)}_j$ as a level set function. For this purpose, let us first provide a lower-bound to the norm of the tangential gradient of $g^{(\epsilon)}_j$ on $\partial\Omega_M\cap V_j$. Note that its gradient is given by
\begin{equation}\label{eq:expr gradient g}
\nabla g^{(\epsilon)}_j(x)=\frac{1}{\epsilon^\rho}
\begin{pmatrix}
\nabla^2_{q,\,p}\mathrm{d}_{\Omega}(x) v_j  \\
\nabla^2_{p,\,p}\mathrm{d}_{\Omega}(x)v_j
\end{pmatrix}\in\R^{2d}\;.
\end{equation}
In particular, by Lemma~\ref{lem:signed distance},
\begin{align*}
\langle \nabla g^{(\epsilon)}_j(x),\,n(x)\rangle&=\frac{1}{\epsilon^\rho}\left\langle\nabla^2_{p,\,q}\mathrm{d}_{\Omega}(x)n_q(x)+\nabla^2_{p,\,p}\mathrm{d}_{\Omega}(x)n_p(x),\,v_j\right\rangle\\
&=0\;.
\end{align*}
Therefore, its tangential gradient on $\partial\Omega$ is given by
\begin{align}
\nabla_{\partial\Omega} g^{(\epsilon)}_j(x)&=\nabla g^{(\epsilon)}_j(x)-\langle \nabla g^{(\epsilon)}_j(x),\,n(x)\rangle n(x)\nonumber\\
&=\nabla g^{(\epsilon)}_j(x)\label{eq:expr gradient g tangential}\;.
\end{align}
As a result, by~\eqref{eq:lower bound hessian velocity v_j}, for all $x\in\partial\Omega_M\cap V_j$,
$$\left|\nabla_{\partial\Omega} g^{(\epsilon)}_j(x)\right|\geq\frac{1}{\epsilon^\rho}|\nabla^2_{p,\,p}\mathrm{d}_{\Omega}(x)v_j|\geq\frac{1}{2}\;.$$
Also, notice that if $|n_p(x)|\leq\theta\epsilon$, then $|g_j^{(\epsilon)}(x)|\leq \theta\epsilon^{1-\rho}$. Consequently, the following upper-bound holds for any $j\in\llbracket1,\,N_\epsilon\rrbracket$,
\begin{equation}\label{eq:summation V_j tangential grad}
\int_{\partial\Omega_M\cap V_j}\Psi_j(x)\mathbf{1}_{\Vert\nabla^2_{p,\,p}\mathrm{d}_\Omega(x)\Vert\geq\epsilon^\rho}\mathbf{1}_{|n_p(x)|\leq\theta\epsilon}\,\sigma_{2d-1}(\mathrm{d}x)\leq\int_{\partial\Omega}\Psi_j(x)\frac{|\nabla_{\partial\Omega} g^{(\epsilon)}_j(x)|^2}{1/4}\mathbf{1}_{|g_j^{(\epsilon)}(x)|\leq\theta\epsilon^{1-\rho}}\,\sigma_{2d-1}(\mathrm{d}x)\;.
\end{equation}
Furthermore, by the coarea formula~\cite[Theorem 18.8]{maggi2012sets},
\begin{align}
&\int_{\partial\Omega}\Psi_j(x)|\nabla_{\partial\Omega} g^{(\epsilon)}_j(x)|^2\mathbf{1}_{|g_j^{(\epsilon)}(x)|\leq\theta\epsilon^{1-\rho}}\,\sigma_{2d-1}(\mathrm{d}x)\nonumber\\
&=\int_{-\theta\epsilon^{1-\rho}}^{\theta\epsilon^{1-\rho}}\int_{\partial\Omega\cap\{g_j^{(\epsilon)}=t\}}\Psi_j(x)|\nabla_{\partial\Omega} g^{(\epsilon)}_j(x)|\sigma_{2d-2}(\mathrm{d}x)\mathrm{d}t\label{eq:coarea formula V_j}\,.
\end{align}
In addition, by the divergence theorem, for all $t\in[-\theta\epsilon^{1-\rho},\,\theta\epsilon^{1-\rho}]$,
\begin{equation}\label{eq:div thm haussdorf meas}
\int_{\partial\Omega\cap\{g_j^{(\epsilon)}=t\}}\Psi_j(x)|\nabla_{\partial\Omega} g^{(\epsilon)}_j(x)|\sigma_{2d-2}(\mathrm{d}x)=\int_{\partial\Omega\cap\{g_j^{(\epsilon)}<t\}}\mathrm{div}_{\partial\Omega}\big(\Psi_j\nabla_{\partial\Omega} g^{(\epsilon)}_j\big)(x)\sigma_{2d-1}(\mathrm{d}x)\;.
\end{equation}
Moreover, for $x\in\partial\Omega$,
\begin{equation}\label{eq:decomp div partial Omega}
\mathrm{div}_{\partial\Omega}\big(\Psi_j\nabla_{\partial\Omega} g^{(\epsilon)}_j\big)(x)=\Psi_j(x)\Delta_{\partial\Omega}g^{(\epsilon)}_j(x)+\langle\nabla_{\partial\Omega}\Psi_j(x),\,\nabla_{\partial\Omega}g^{(\epsilon)}_j(x)\rangle\;.    
\end{equation}
From the definitions provided in~\eqref{eq:def expr Psi_j} and~\eqref{eq:expr g_j eps}, we deduce the existence of a constant $C_1>0$ independent of $\epsilon>0$ such that for all $x\in\partial\Omega$,
$$|\Delta_{\partial\Omega}g^{(\epsilon)}_j(x)|+|\nabla_{\partial\Omega}g^{(\epsilon)}_j(x)|+|\nabla_{\partial\Omega}\Psi_j(x)|\leq\frac{C_1}{\epsilon^\rho}\;.$$
Therefore, we deduce from~\eqref{eq:decomp div partial Omega} that for all $x\in\partial\Omega$,
$$\big|\mathrm{div}_{\partial\Omega}\big(\Psi_j\nabla_{\partial\Omega} g^{(\epsilon)}_j\big)(x)\big|\leq\bigg(\frac{C_1}{\epsilon^\rho}+\frac{C_1^2}{\epsilon^{2\rho}}\bigg)\mathbf{1}_{|x-z_j|\leq\epsilon^\rho/2C_{\mathrm{Lip}}}$$
since $\Psi_j$ is compactly supported on the ball $\mathrm{B}(z_j,\,\epsilon^\rho/2C_{\mathrm{Lip}})$. Reinjecting into~\eqref{eq:div thm haussdorf meas} we obtain that
\begin{align*}
\int_{\partial\Omega\cap\{g_j^{(\epsilon)}=t\}}\Psi_j(x)|\nabla_{\partial\Omega} g^{(\epsilon)}_j(x)|\sigma_{2d-2}(\mathrm{d}x)&\leq\frac{C_1+C_1^2}{\epsilon^{2\rho}}\int_{\partial\Omega}\mathbf{1}_{|x-z_j|\leq\epsilon^\rho/2C_{\mathrm{Lip}}}\sigma_{2d-1}(\mathrm{d}x)\\
&\leq\frac{C_2}{\epsilon^{2\rho}}\;\epsilon^{\rho(2d-1)} 
\end{align*}
for some constant $C_2>0$ independent of $\epsilon>0$. All in all, we deduce from~\eqref{eq:boundary integral partitioning},~\eqref{eq:summation V_j tangential grad} and~\eqref{eq:coarea formula V_j} that
\begin{align*}
\int_{\partial\Omega_M}\mathbf{1}_{\Vert\nabla^2_{p,\,p}\mathrm{d}_\Omega(x)\Vert\geq\epsilon^\rho}\mathbf{1}_{|n_p(x)|\leq\theta\epsilon}\,\sigma_{2d-1}(\mathrm{d}x)&\leq8\theta\epsilon^{1-\rho}N_\epsilon\frac{C_2}{\epsilon^{2\rho}}\;\epsilon^{\rho(2d-1)}\\
&=8\theta\beta\epsilon^{1-3\rho}C_2
\end{align*}
using~\eqref{eq:nbr of balls}, hence the proof of~\eqref{eq:upper bound neighborhood tangential}.

\textbf{Step 3}: Let us conclude the proof of this proposition by showing~\eqref{eq:surface meas Q}. Taking the limit $\epsilon\rightarrow0$ in~\eqref{eq:upper bound neighborhood tangential} ensures that
$$\sigma_{2d-1}\big(\partial\Omega\cap\big\{x\in\partial\Omega:|n_p(x)|=0,\,\nabla^2_{p,\,p}\mathrm{d}_\Omega(x)\neq\mathbb{O}_{d}\big\}\big)=0\;.$$
Given that $(\partial\Omega_M)_{M\geq1}$ is an increasing sequence, the identity~\eqref{eq:surface meas Q} immediately follows by taking the limit $M\rightarrow\infty$ in the equality above.
\end{proof}
\subsection{Proof of Theorem~\ref{thm:polarity}}\label{sec:thm polarity Gamma0}
Before proving Theorem~\ref{thm:polarity}, let us show the following lemma which ensures a control on the transition density of the process~\eqref{eq:sde} remaining in a given ball.

\begin{lem}[Density upper-bound]\label{lem:upper bound density}
Let $M,\,T,\,\delta>0$. There exists a constant $C>0$ such that for all measurable subsets $A\subset\mathrm{B}(0,\,M)$, for all $t\in(0,\,T)$, for all $x\in\mathrm{B}(0,\,M)$ satisfying $\mathrm{dist}(x,\,A)\geq\delta$, 
\begin{equation}\label{eq:control density law in ball B_M}
\mathbb{P}_x\big(X(t)\in A,\,\sup_{s\in[0,\,t]}|X(s)|\leq M\big)\leq C|A|\,,
\end{equation}
where $|A|$ is the Euclidean volume of $A$.
\end{lem}
\begin{proof} Recall the function $\Phi$ defined in~\eqref{eq:def Phi} and let us fix $M>0$. For any $t>0$, for any $x\in\mathrm{B}(0,\,M)$,
$$\mathbb{P}_{x}\big(X(t)\in A,\,\sup_{s\in[0,\,t]}|X(s)|\leq M\big)\leq\mathbb{P}_{x}\big(\Phi(X(t))\in\Phi(A),\,\sup_{s\in[0,\,t]}|X(s)|\leq M\big)\;.$$

The process $(\Phi(X(t)))_{t\geq0}$ satisfies the SDE~\eqref{eq:sde} with coefficients $\widetilde{G},\,\widetilde{F},\,\widetilde{\Sigma}$ detailed in~\eqref{eq:expr coeff X tilde}. The goal now is to apply~\cite[Theorem 2.1]{Menozzi} to obtain an upper-bound on the transition density of the process $(\Phi(X(t)))_{t\geq0}$. However, since the coefficients $\widetilde{G},\,\widetilde{F},\,\widetilde{\Sigma}$ may be bounded and Lipschitz-continuous only locally, the following regularization procedure is applied.

It is easy to see that the function $p\in\R^{2d}\mapsto DG(0)^{-1}G(p)$ is a $C^\infty$ diffeomorphism which additionally preserves the orientation. By~\cite{palais1960}, this function can be extended outside of the ball $\mathrm{B}(0,\,M)$ by a $C^\infty$ diffeomorphism equal to the identity outside of some compact set. Therefore, we can extend $G$ by a function $G_M:\R^d\to\R^d$ equal to $p\mapsto DG(0)p$ outside of some compact set. In particular, all the derivatives of the extension $G_M$ are bounded on $\R^{2d}$. Also, let us take $C^\infty$ functions $F_M,\,\Sigma_M$ which coincide with $F,\,\Sigma$ on $\mathrm{B}(0,\,M)$ and are bounded, globally Lipschitz continuous on $\R^{2d}$ such that $\Sigma_M$ remains invertible on $\R^{2d}$.

Let us define $(X^M(t))_{t\geq0}$ as the solution to the SDE~\eqref{eq:sde} with coefficients $G_M,\,F_M,\,\Sigma_M$. By Friedman's uniqueness result~\cite[Theorem 5.2.1]{F}, for any $t>0$, the trajectories of $(X(s))_{s\in[0,\,t]}$ and $(X^M(s))_{s\in[0,\,t]}$ satisfying $X(0)=X^M(0)$, coincide on the event $\{\sup_{0\leq s\leq t}\vert X(s)\vert\leq~M\}$. Define now
$$\Phi_M:(q,\,p)\in\R^{2d}\mapsto (q,\,G_M(p))\;.$$
For all $x\in\mathrm{B}(0,\,M)$,
\begin{align}
&\mathbb{P}_{x}\big(\Phi(X(t))\in\Phi(A),\,\sup_{s\in[0,\,t]}|X(s)|\leq M\big)\nonumber\\
&=\mathbb{P}_{x}\big(\Phi_M(X^M(t))\in\Phi_M(A),\,\sup_{s\in[0,\,t]}|X(s)|\leq M\big)\nonumber\\
&\leq\mathbb{P}_{x}\big(\Phi_M(X^M(t))\in\Phi_M(A)\big)\label{eq:ineq law Phi X_t by Phi_M}\;.
\end{align}
The process $(\Phi_M(X^M(t)))_{t\geq0}$ satisfies the SDE~\eqref{eq:sde} with coefficients obtained from~\eqref{eq:expr coeff X tilde} by replacing $G,\,F,\,\Sigma$ with $G_M,\,F_M,\,\Sigma_M$. In particular, the drift and diffusion coefficients of $(\Phi_M(X^M(t)))_{t\geq0}$ are thus bounded and globally Lipschitz continuous. Therefore, by~\cite[Theorem 2.1]{Menozzi}, for any $T>0$, there exist constants $c,\,C>0$ such that the transition density of $\big(\Phi_M(X^M(t))\big)_{t\geq0}$ is bounded by the following Gaussian density for $t\in(0,\,T)$, $x=(q,\,p),\,x'=(q',\,p')\in\R^{2d}$,
\begin{equation}\label{eq:gaussian upper bound}
\mathrm{p}(t;x,\,x')=\frac{C}{t^{2d}}\mathrm{exp}\left(-c\left[\frac{|p-p'|^2}{4t}+\frac{3|q'-q-tp|^2}{t^3}\right]\right)\;.
\end{equation}
\textbf{Step 1}: Let us prove that for any $\delta>0$,
\begin{equation}\label{eq:sup density Phi X_t}
\sup_{t\in(0,\,T)}\sup_{x\in\mathrm{B}(0,\,M)}\sup_{|x-x'|\geq\delta}\mathrm{p}(t;\,\Phi_M(x),\,\Phi_M(x'))<\infty\;.
\end{equation}
Assume that the property above does not hold then there exist sequences $(t_n)_{n\geq1}$ in $(0,\,T)$, $(x_n=(q_n,\,p_n))_{n\geq1}$ in $\mathrm{B}(0,\,M)$, $(x'_n=(q'_n,\,p'_n))_{n\geq1}$ in $\R^{2d}$ satisfying $|x_n-x'_n|\geq\delta$ for all $n\geq1$ and such that
\begin{equation}\label{eq:limit density equal infty}
\lim_{n\rightarrow\infty}\mathrm{p}(t_n;\,\Phi_M(x_n),\,\Phi_M(x'_n))=\infty\;.
\end{equation}
Given that the sequence $(x_n)_{n\geq1}$ is bounded, up to extracting a subsequence, we can assume that it admits a limit when $n\rightarrow\infty$. Additionally, since 
$$\mathrm{p}(t_n;\,\Phi_M(x_n),\,\Phi_M(x'_n))\leq \frac{C}{t_n^{2d}}\,,$$
we deduce from~\eqref{eq:limit density equal infty} that $t_n\rightarrow0$ when $n\rightarrow\infty$. Furthermore, the inequality
$$\mathrm{p}(t_n;\,\Phi_M(x_n),\,\Phi_M(x'_n))\leq\frac{C}{t_n^{2d}}\left[\mathrm{exp}\left(-c\,\frac{\big|G_M(p_n)-G_M(p'_n)\big|^2}{4t_n}\right)\land\mathrm{exp}\left(-3c\,\frac{\big|q'_n-q_n-t_nG_M(p_n)\big|^2}{t_n^3}\right)\right]\;,$$
ensures that
$$\big|G_M(p_n)-G_M(p'_n)\big|\underset{n\rightarrow\infty}{\longrightarrow0}\,,\qquad\text{and}\qquad\big|q'_n-q_n-t_nG_M(p_n)\big|\underset{n\rightarrow\infty}{\longrightarrow}0\;.$$
Therefore, since the sequence $(p_n)_{n\geq1}$ is convergent and $G_M$ is a diffeomorphism, we deduce that 
$$|p_n-p'_n|\underset{n\rightarrow\infty}{\longrightarrow}0\,.$$
Moreover, given that $|p_n|\leq M$, we deduce that $t_n|p_n|\rightarrow0$ and therefore that $|q_n-q'_n|\rightarrow0$ when $n\rightarrow\infty$. Consequently, $|x_n-x'_n|$ converges to zero which is in contradiction with the fact that $|x_n-x'_n|\geq\delta$ for all $n\geq1$, hence~\eqref{eq:sup density Phi X_t}.

\textbf{Step 2}: Let us now conclude the proof of the inequality~\eqref{eq:control density law in ball B_M} using~\eqref{eq:ineq law Phi X_t by Phi_M}. For all $x\in\mathrm{B}(0,\,M)$,
\begin{align*}
\mathbb{P}_{x}\big(\Phi_M(X^M(t))\in\Phi_M(A)\big)&\leq\int_{\Phi_M(A)}\mathrm{p}(t;\,\Phi_M(x),\,y)\mathrm{d}y\\
&=\int_{A}\mathrm{p}(t;\,\Phi_M(x),\,\Phi_M(x'))\,\mathrm{det}|D\Phi_M(x')|\mathrm{d}x'\;.
\end{align*}
Since the derivatives of $\Phi_M$ are bounded on $\R^{2d}$, the proof immediately follows from~\eqref{eq:sup density Phi X_t}.
\end{proof}
\begin{proof}[Proof of Theorem~\ref{thm:polarity}]
For $\rho\geq0$, let us define the sets
\begin{align*}
\mathcal{G}_\rho&:=\big\{x\in\partial\Omega:|n_p(x)|=0,\,||\nabla^2_{p,\,p}\mathrm{d}_\Omega(x)||>\rho\big\}\\
\mathcal{T}&:=\big\{x\in\partial\Omega:|n_p(x)|=0,\,\xi(x)=0\big\}\;.
\end{align*}
Let also
\begin{equation}\label{eq:tau G_0 T}
\tau_{\mathcal{G}_0\cup\mathcal{T}}:=\inf\{t>0:\,X(t)\in\mathcal{G}_0\cup\mathcal{T}\}\;.
\end{equation}
The proof of Theorem~\ref{thm:polarity} immediately follows if for all $x\notin\mathcal{G}_0\cup\mathcal{T}$, for all $T\geq0$,
\begin{equation}\label{eq: proba tau_G_0 leq T}
\mathbb{P}_x(\tau_{\mathcal{G}_0\cup\mathcal{T}}\leq T)=0,
\end{equation}
by taking the limit $T\rightarrow\infty$. We first provide the proof of~\eqref{eq: proba tau_G_0 leq T} for any $x\notin\mathcal{G}_0\cup\mathcal{T}$ satisfying $\mathrm{dist}(x,\,\mathcal{G}_0\cup\mathcal{T})>~0$ and extend the result later to any $x\notin\mathcal{G}_0\cup\mathcal{T}$ in the last step of the proof.

Let us first fix $x\notin\mathcal{G}_0\cup\mathcal{T}$ satisfying $\mathrm{dist}(x,\,\mathcal{G}_0\cup\mathcal{T})>0$ and let $T\geq0$. Before delving into the proof of~\eqref{eq: proba tau_G_0 leq T} let us define constants $\alpha\in(0,\,1/2)$ and $\beta\in(0,\,\alpha/3)$ which satisfy
\begin{equation}\label{eq:conditions beta alpha}
2\alpha+\beta>1\,,\qquad3\alpha-3\beta>1\;.
\end{equation}
In practice, taking $\alpha\in(0,\,1/2)$ close enough to $1/2$ and $\beta>0$ small enough provide solutions to~\eqref{eq:conditions beta alpha}. For instance, $\alpha=17/36$ and $\beta=1/12$ satisfy these conditions. These constants are used below in the proof of~\eqref{eq: proba tau_G_0 leq T}.

\textbf{Step 1}: Let us fix an arbitrary small $\epsilon>0$. Consider also $M>0$ large enough such that
\begin{equation}\label{eq:x in ball M}
x\in\mathrm{B}(0,\,M)\;.
\end{equation}
Define the event
$$\mathcal{A}:=\left\{\sup_{0\leq t\leq T}\vert X(t)\vert\leq M ,\,\sup_{0\leq s<t\leq T}\frac{\vert p(t)-p(s)\vert}{\vert t-s\vert^{\alpha}}\leq M \right\}\;.$$
Given that the process $(X(t))_{t\geq0}$ does not explode in finite time almost-surely according to Assumption~\ref{ass:non expl} and since the trajectories of the process $(p(t))_{t\geq0}$ in~\eqref{eq:sde} are $\alpha$-H\"olderian continuous, there exists by $\sigma$-additivity $M = M _\epsilon$ large enough so that $\mathbb{P}_x(\mathcal{A}^c)\leq\epsilon$. As a result,
$$\mathbb{P}_x(\tau_{\mathcal{G}_0\cup\mathcal{T}}\leq T)\leq\mathbb{P}_x(\tau_{\mathcal{G}_0\cup\mathcal{T}}\leq T,\,\mathcal{A})+\epsilon\;.$$
Moreover, by the definition~\eqref{eq:tau G_0 T},
$$\mathbb{P}_x(\tau_{\mathcal{G}_0\cup\mathcal{T}}\leq T,\,\mathcal{A})=\lim_{n\rightarrow\infty}\mathbb{P}_x\big(\exists t\in(T/n,\,T],\,X(t)\in\mathcal{G}_{1/n^\beta}\cup\mathcal{T},\,\mathcal{A}\big)\;.$$
For $n\geq1$, consider the sequence of times $(t_k:=kT/n)_{1\leq k\leq n}$. It follows that
\begin{align}
&\mathbb{P}_x\big(\exists t\in(T/n,\,T],\,X(t)\in\mathcal{G}_{1/n^\beta}\cup\mathcal{T},\,\mathcal{A}\big)\nonumber\\
&\leq\sum_{k=1}^{n-1}\mathbb{P}_x\big(\exists t\in(t_k,\,t_{k+1}],\,X(t)\in\mathcal{G}_{1/n^\beta}\cup\mathcal{T},\,\mathcal{A}\big)\nonumber\\
&\leq\sum_{k=1}^{n-1}\mathbb{P}_x\big(\exists t\in(t_k,\,t_{k+1}],\,X(t)\in\mathcal{G}_{1/n^\beta},\,\mathcal{A}\big)+\sum_{k=1}^{n-1}\mathbb{P}_x\big(\exists t\in(t_k,\,t_{k+1}],\,X(t)\in\mathcal{T}\setminus\mathcal{G}_{1/n^\beta},\,\mathcal{A}\big)\label{eq:sums t_k}\;.
\end{align}
Therefore, the rest of the proof consists in showing that the sums above vanish when $n\rightarrow\infty$ by providing sharp upper-bounds on the probabilities therein.

\textbf{Step 2}: Let us provide in this step some preliminary estimates. Under the event $\mathcal{A}$, let us assume the existence of $\tau\in(t_k,\,t_{k+1}]$ for some $k\in\llbracket1,\,n-1\rrbracket$ such that $X(\tau)\in\mathcal{G}_{1/n^\beta}\cup\mathcal{T}$. First, let us deduce the existence of a deterministic constant $C_1>0$ independent of $k\in\llbracket1,\,n-1\rrbracket$ and $n\geq1$ such that for all $t\in[t_k,\,t_{k+1}]$,
\begin{equation}\label{eq:oscill X_t_k}
|q(t)-q(\tau)|\leq\frac{C_1}{n},\,\qquad|p(t)-p(\tau)|\leq\frac{C_1}{n^\alpha}\;.
\end{equation}
From~\eqref{eq:sde} and the continuity of $G$ on $\R^d$, we obtain that on the event $\mathcal{A}$,
$$|q(t)-q(\tau)|=\bigg|\int_{t}^\tau G(p(s))\mathrm{d}s\bigg|\leq\frac{T}{n}\sup_{z\in\mathrm{B}(0,\,M)}|G(z)|\;.$$
Additionally, under the event $\mathcal{A}$,
$$|p(t)- p(\tau)\vert\leq\frac{MT^\alpha}{n^\alpha},$$
hence~\eqref{eq:oscill X_t_k}. In particular, the inequality~\eqref{eq:oscill X_t_k} ensures that
\begin{equation}\label{eq:distance X_t_k G n beta}
|X(t_k)-X(\tau)|\leq\frac{2C_1}{n^\alpha}\;.
\end{equation}

\textbf{Step 3}: The goal of this step is to examine the limit of the first sum in~\eqref{eq:sums t_k} when $n\rightarrow\infty$. For this purpose, under the event $\mathcal{A}$, let us now assume the existence of $\tau\in(t_k,\,t_{k+1}]$ for some $k\in\llbracket1,\,n-1\rrbracket$ such that $X(\tau)\in\mathcal{G}_{1/n^\beta}$.

For $n\geq1$ and $C>0$, let us define the set
$$\mathcal{B}_n(C):=\bigg\{x\in\R^{2d}:\;|\nabla_p\mathrm{d}_{\Omega}(x)|\leq\frac{C}{n^{\alpha}}\,,\qquad|\mathrm{d}_{\Omega}(x)|\leq\frac{C}{n^{2\alpha}}\,,\qquad||\nabla^2_{p,\,p}\mathrm{d}_{\Omega}(x)||\geq\frac{1}{2n^{\beta}}\bigg\}\;.$$

Let us first prove the existence of a constant $C_2>0$ independent of $k\in\llbracket1,\,n-1\rrbracket$ and $n\geq1$ such that $X(t_{k})\in\mathcal{B}_n(C_2)$. Recall that the distance function $\mathrm{d}_\Omega$ is $C^{3,\,1}$ on a tubular neighborhood, denoted by $\mathcal{C}$, of $\partial\Omega\cap\mathrm{B}(0,\,M)$. We deduce from~\eqref{eq:distance X_t_k G n beta} that for $n\geq1$ large enough,   $X(t_k)\in\mathcal{C}$. Therefore, since $|\nabla_p\mathrm{d}_{\Omega}(X(\tau))|=0$, there exists a constant $K>0$ such that
\begin{align*}
|\nabla_p\mathrm{d}_{\Omega}(X(t_k))|&=|\nabla_p\mathrm{d}_{\Omega}(X(t_k))-\nabla_p\mathrm{d}_{\Omega}(X(\tau))|\\
&\leq K|X(t_k)-X(\tau)|\leq\frac{2KC_1}{n^\alpha}
\end{align*}
by~\eqref{eq:distance X_t_k G n beta}. The second inequality in the definition of $\mathcal{B}_n(C)$ is a consequence of the Taylor-Lagrange inequality along with~\eqref{eq:distance X_t_k G n beta},
\begin{align*}
\big|\mathrm{d}_{\Omega}(X(t_k))-\langle\nabla\mathrm{d}_{\Omega}(X(\tau)),\,X(t_k)-X(\tau)\rangle\big|&\leq\frac{1}{2}\sup_{z\in\mathcal{C}}||\nabla^2\mathrm{d}_\Omega(z)||\,|X(t_k)-X(\tau)|^2\\
&\leq\sup_{z\in\mathcal{C}}||\nabla^2\mathrm{d}_\Omega(z)||\frac{2C_1^2}{n^{2\alpha}}\;.
\end{align*}
In addition, by~\eqref{eq:oscill X_t_k},
\begin{align*}
|\langle\nabla\mathrm{d}_{\Omega}(X(\tau)),\,X(t_k)-X(\tau)\rangle|&=|\langle n_q(X(\tau)),\,q(t_k)-q(\tau)\rangle|\leq\frac{C_1}{n},
\end{align*}
hence the proof of the second inequality. Regarding the third inequality, the triangle inequality and the $C^{3,\,1}$ regularity of $\mathrm{d}_\Omega$ on the tube $\mathcal{C}$ ensure the existence of a constant $K'>0$ such that
\begin{align*}
||\nabla^2_{p,\,p}\mathrm{d}_{\Omega}(X(t_{k}))||&\geq||\nabla^2_{p,\,p}\mathrm{d}_{\Omega}(X(\tau))||-||\nabla^2_{p,\,p}\mathrm{d}_{\Omega}(X(t_{k}))-\nabla^2_{p,\,p}\mathrm{d}_{\Omega}(X(\tau))||\\
&\geq\frac{1}{n^\beta}-\frac{2K'C_1}{n^\alpha}\\
&\geq\frac{1}{2n^\beta}
\end{align*}
for $n\geq1$ large enough since $\alpha>\beta$, hence $X(t_{k})\in\mathcal{B}_n(C_2)$ for some constant $C_2>0$ independent of $n\geq1$.

Consider now the first probability in~\eqref{eq:sums t_k}. One deduces from~\eqref{eq:distance X_t_k G n beta} that for $n$ large enough,
\begin{align*}
&\mathbb{P}_x\big(\exists t\in(t_k,\,t_{k+1}],\,X(t)\in\mathcal{G}_{1/n^\beta},\,\mathcal{A}\big)\\
&\leq\mathbb{P}_x\big(X(t_k)\in\mathcal{B}_n(C_2),\,\mathrm{dist}(X(t_k),\,\mathcal{G}_{1/n^\beta})\leq2C_1/n^\alpha,\,\mathcal{A}\big)\;, \\
&\leq\mathbb{P}_x\big(X(t_k)\in\mathcal{B}_n(C_2),\,|X(t_k)-x|\geq \delta ,\,\mathcal{A}\big)\;,
\end{align*}
for some $\delta>0$ since $\mathrm{dist}(x,\,\mathcal{G}_{1/n^\beta})\geq\mathrm{dist}(x,\,\mathcal{G}_0\cup\mathcal{T})> 0$ given that $\mathcal{G}_{1/n^\beta}\subset\mathcal{G}_{0}$. Therefore, we deduce from~\eqref{eq:x in ball M} and Lemma~\ref{lem:upper bound density} the existence of a constant $C_3>0$ independent of $k\in\llbracket1,\,n-1\rrbracket$ and $n\geq1$ large enough such that
\begin{align*}
&\mathbb{P}_x\big(\exists t\in(t_k,\,t_{k+1}],\,X(t)\in\mathcal{G}_{1/n^\beta},\,\mathcal{A}\big)\\
&\leq C_3\int_{\R^{2d}}\mathbf{1}_{|x|\leq M}\mathbf{1}_{|\mathrm{d}_{\Omega}(x)|\leq C_2/n^{2\alpha}}\mathbf{1}_{|\nabla_p\mathrm{d}_{\Omega}(x)|\leq C_2/n^{\alpha}}\mathbf{1}_{\Vert\nabla^2_{p,\,p}\mathrm{d}_{\Omega}(x)\Vert\geq1/2n^\beta}\mathrm{d}x\;.
\end{align*}
As the support of the integral above lies in a bounded and close neighborhood of the $C^{3,\,1}$ boundary $\partial\Omega$, one can apply the change of variable defined by the tubular parametrization $(x,\,\lambda) \in \partial \Omega \times \mathbb R \mapsto x':=x+\lambda n(x)$ and detailed for instance in~\cite[Lemma 14.16]{GT} to obtain the following integral
\begin{equation}\label{eq:integral case X_tau G}
C_3\int_{-C_2/n^{2\alpha}}^{C_2/n^{2\alpha}}\int_{\partial\Omega}\mathbf{1}_{|x+\lambda n(x)|\leq M}\mathbf{1}_{|\nabla_p\mathrm{d}_{\Omega}(x+\lambda n(x))|\leq C_2/n^{\alpha}}\mathbf{1}_{||\nabla^2_{p,\,p}\mathrm{d}_{\Omega}(x+\lambda n(x))||\geq1/2n^\beta} \, J(x,\lambda) \sigma_{2d-1}(\mathrm{d}x)\mathrm{d}\lambda\;,
\end{equation}
where in the above $J(x,\lambda)$ is the Jacobian determinant which is bounded from above on $B(0,M)$ by a constant $C_J>0$. Notice that for $n\geq1$ large enough, for all $x\in\partial\Omega$, for all $\lambda\in[-C_2/n^{2\alpha},\,C_2/n^{2\alpha}]$,
$$\mathbf{1}_{|x+\lambda n(x)|\leq M}\leq\mathbf{1}_{|x|\leq2M}\;.$$
Similarly, the $C^{3,\,1}$ regularity of the distance function $\mathrm{d}_\Omega$ on a bounded tube around $\partial\Omega$ ensures that for $n\geq1$ large enough, the integral in~\eqref{eq:integral case X_tau G} admits the following upper-bound
$$\frac{2C_2C_3C_J>0}{n^{2\alpha}}\int_{\partial\Omega\cap\mathrm{B}(0,\,2M)}\mathbf{1}_{|\nabla_p\mathrm{d}_{\Omega}(x)|\leq2C_2/n^{\alpha}}\mathbf{1}_{||\nabla^2_{p,\,p}\mathrm{d}_{\Omega}(x)||\geq1/4n^\beta}\sigma_{2d-1}(\mathrm{d}x)\;.$$
Applying Proposition~\ref{prop:control neighborhood Gamma_0} for $\epsilon=1/n^\alpha$ and $\rho=\beta/\alpha$ thus yields the existence of a constant $C_4>0$ independent of $n\geq1$ large enough such that the above quantity is bounded by $C_4/n^{3\alpha-3\beta}$.
All in all, we obtain that
\begin{align*}
\sum_{k=1}^{n-1}\mathbb{P}_x\big(\exists t\in(t_k,\,t_{k+1}],\,X(t)\in\mathcal{G}_{1/n^\beta},\,\mathcal{A}\big)&\leq\sum_{k=1}^{n-1}\frac{C_4}{n^{3\alpha-3\beta}}=\frac{C_4}{n^{3\alpha-3\beta-1}}
\end{align*}
which vanishes when $n\rightarrow\infty$ since $3\alpha-3\beta>1$ by~\eqref{eq:conditions beta alpha}.

\textbf{Step 4}: The goal of this step is to show that the second sum in~\eqref{eq:sums t_k} vanishes when $n\rightarrow\infty$. Under the event $\mathcal{A}$, let us assume this time the existence of $k\in\llbracket1,\,n-1\rrbracket$ and $\tau\in(t_k,\,t_{k+1}]$ such that $X(\tau)\in\mathcal{T}\setminus\mathcal{G}_{1/n^\beta}$. Let us first prove that there exists a constant $C_5>0$ independent of $k\in\llbracket1,\,n-1\rrbracket$ such that for $n\geq1$ large enough,
\begin{equation}\label{eq:upper bound distance case T setminus G}
|\mathrm{d}_{\Omega}(X(t_{k}))|\leq\frac{C_5}{n^{2\alpha+\beta}}\;.
\end{equation}

Applying the Taylor-Lagrange inequality up to the second order, we have that
\begin{align*}
&\big\vert\mathrm{d}_{\Omega}(X(t_{k}))-\langle\nabla\mathrm{d}_{\Omega}(X(\tau)),\,X(t_k)-X(\tau)\rangle-\frac{1}{2}\langle X(t_k)-X(\tau),\,\nabla^2\mathrm{d}_{\Omega}(X(\tau))( X(t_k)-X(\tau))\rangle\big\vert\\
&\leq\frac{1}{6}\sup_{z\in\mathcal{C}}|||\nabla^3\mathrm{d}_\Omega(z)|||\,|X(t_k)-X(\tau)|^3\\
&\leq\frac{4}{3}\sup_{z\in\mathcal{C}}|||\nabla^3\mathrm{d}_\Omega(z)|||\frac{C_1^3}{n^{3\alpha}}
\end{align*}
using~\eqref{eq:distance X_t_k G n beta}. Moreover, given that $G$ is Lipschitz continuous on $\mathrm{B}(0,\,M)$, we deduce the existence of a constant $C_6>0$ such that
\begin{align*}
\big|\langle\nabla\mathrm{d}_{\Omega}(X(\tau)),\,X(t_k)-X(\tau)\rangle\big|&=\big|\langle n_q(X(\tau)),\,q(t_k)-q(\tau)\rangle\big|\\
&=\bigg|\int_\tau^{t_k}\langle n_q(X(\tau)),\,G(p(s))-G(p(\tau))\rangle\mathrm{d}s\bigg|\\
&\leq\frac{C_6}{n^{1+\alpha}}
\end{align*}
using~\eqref{eq:oscill X_t_k} and the fact that $X(\tau)\in\mathcal{T}$. The inequality~\eqref{eq:oscill X_t_k} also guarantees that
\begin{align*}
&\big|\langle X(t_k)-X(\tau),\,\nabla^2\mathrm{d}_{\Omega}(X(\tau)) (X(t_k)-X(\tau))\rangle\big|\\
&\leq\big|\langle q(t_k)-q(\tau),\,\nabla^2_{q,\,q}\mathrm{d}_{\Omega}(X(\tau)) (q(t_k)-q(\tau))\rangle+2\langle q(t_k)-q(\tau),\,\nabla^2_{q,\,p}\mathrm{d}_{\Omega}(X(\tau)) (p(t_k)-p(\tau))\rangle\big|\\
&+\big|\langle p(t_k)-p(\tau),\,\nabla^2_{p,\,p}\mathrm{d}_{\Omega}(X(\tau)) (p(t_k)-p(\tau))\rangle\big|\\
&\leq||\nabla^2_{p,\,p}\mathrm{d}_{\Omega}(X(\tau))||\,\frac{C_1^2}{n^{2\alpha}}+||\nabla^2\mathrm{d}_\Omega(X(\tau))||\bigg(\frac{C_1^2}{n^2}+2\frac{C_1^2}{n^{1+\alpha}}\bigg)\\
&\leq\frac{C_1^2}{n^{2\alpha+\beta}}+3\sup_{z\in\mathcal{C}}||\nabla^2\mathrm{d}_\Omega(z)||\frac{C_1^2}{n^{1+\alpha}}
\end{align*}
where in the last line we have used that $X(\tau) \notin \mathcal G_{1/n^\beta}$. Hence~\eqref{eq:upper bound distance case T setminus G} since $2\alpha+\beta<1+\alpha$. Therefore, given~\eqref{eq:distance X_t_k G n beta} and~\eqref{eq:upper bound distance case T setminus G}, we get for all $n$ large enough,
\begin{align*}
&\mathbb{P}_x\big(\exists t\in(t_k,\,t_{k+1}],\,X(t)\in\mathcal{T}\setminus\mathcal{G}_{1/n^\beta},\,\mathcal{A}\big)\\
&\leq\mathbb{P}_x\big(|\mathrm{d}_{\Omega}(X(t_{k}))|\leq C_5/n^{2\alpha+\beta},\,\mathrm{dist}(X(t_k),\,\mathcal{T})\leq2C_1/n^\alpha,\,\mathcal{A}\big)\;  \\
& \leq \mathbb{P}_x\big(|\mathrm{d}_{\Omega}(X(t_{k}))|\leq C_5/n^{2\alpha+\beta},\,|X(t_k)-x| \geq \delta,\,\mathcal{A}\big)\;,
\end{align*}
for some $\delta>0$ since $\mathrm{dist}(x,\,\mathcal{T})>0$. We deduce from~\eqref{eq:x in ball M} and Lemma~\ref{lem:upper bound density}, the existence of a constant $C_7>0$ such that for $n\geq1$ large enough,
\begin{align*}
\mathbb{P}_x\big(\exists t\in(t_k,\,t_{k+1}],\,X(t)\in\mathcal{T}\setminus\mathcal{G}_{1/n^\beta},\,\mathcal{A}\big)&\leq C_7\int_{\R^{2d}}\mathbf{1}_{|x|\leq M}\mathbf{1}_{|\mathrm{d}_{\Omega}(x)|\leq C_5/n^{2\alpha+\beta}}\mathrm{d}x\\
&\leq\frac{C_8}{n^{2\alpha+\beta}}
\end{align*}
for some constant $C_8>0$ independent of $n\geq1$ using the Weyl tube formula~\cite{WeylTube}. Finally, we obtain that
$$\sum_{k=1}^{n-1}\mathbb{P}_x\big(\exists t\in(t_k,\,t_{k+1}],\,X(t)\in\mathcal{T}\setminus\mathcal{G}_{1/n^\beta},\,\mathcal{A}\big)\leq\frac{C_8}{n^{2\alpha+\beta-1}}$$
which vanishes when $n\rightarrow\infty$ by~\eqref{eq:conditions beta alpha}.

\textbf{Step 5}: Let us fix an arbitrary point $x\notin\mathcal{G}_0\cup\mathcal{T}$ and let us show that~\eqref{eq: proba tau_G_0 leq T} also holds for $x$. The definitions of the sets $\mathcal{G}_0,\,\mathcal{T}$ guarantee that either $\mathrm{dist}(x,\,\mathcal{G}_0\cup\mathcal{T})>0$ or that
\begin{equation}\label{eq:remaining case x}
x\in\partial\Omega,\qquad |n_p(x)|=0,\qquad\xi(x)\neq0,\qquad||\nabla^2_{p,\,p}\mathrm{d}_\Omega(x)||=0\;.
\end{equation}
The case $\mathrm{dist}(x,\,\mathcal{G}_0\cup\mathcal{T})>0$ was already successfully treated in the previous steps and it remains to consider the case~\eqref{eq:remaining case x}. Following the proof of Theorem~\ref{thm:reg points} and in particular the small-time expansion in~\eqref{eq:developt asymptot Gamma}, one has that, $\mathbb{P}_x$ almost-surely,
$$\lim_{t\rightarrow0}\frac{\mathrm{d}_\Omega(X(t))}{t}=\xi(x)\;.$$
Therefore, $\mathbb{P}_x(\tau_{\partial\Omega}>0)=1$ and in particular there exists $\eta\in(0,\,T)$ small enough such that
$$\mathbb{P}_x(\tau_{\partial\Omega}\leq\eta)\leq\epsilon\;.$$
Consider now the probability in~\eqref{eq: proba tau_G_0 leq T}. By the Markov property,
\begin{align*}
\mathbb{P}_x(\tau_{\mathcal{G}_0\cup\mathcal{T}}\leq T)&\leq\mathbb{P}_x(\tau_{\mathcal{G}_0\cup\mathcal{T}}\leq T,\,\tau_{\partial\Omega}>\eta)+\epsilon\nonumber\\
&=\mathbb{E}_x\big[\mathbf{1}_{\tau_{\partial\Omega}>\eta}\mathbb{P}_{X(\eta)}(\tau_{\mathcal{G}_0\cup\mathcal{T}}\leq T-\eta)\big]+\epsilon
\end{align*}
Besides, any point $y\in\R^{2d}\setminus\partial\Omega$ is at a positive distance from $\mathcal{G}_0\cup\mathcal{T}$ and therefore the probability in the expectation above vanishes which concludes the proof by taking $\epsilon\rightarrow0$.
\end{proof}
\subsection{Proof of Theorem~\ref{thm:polarity Gamma+}}\label{sec:proof thm polarity Gamma +}
We conclude the current section with the proof of Theorem~\ref{thm:polarity Gamma+}.
\begin{proof}[Proof of Theorem~\ref{thm:polarity Gamma+}] 

Let us first prove that for any $x\in\Omega\cup\mathcal{I}$, 
\begin{equation}\label{eq:exit event x in Omega}
\mathbb{P}_x\big(X(\tau_{\partial \Omega})\in\Gamma^+,\,\nabla^2_{p,\,p}\mathrm{d}_\Omega(X(\tau_{\partial \Omega}))=\mathbb{O}_d\big)=0\;.
\end{equation}

Fix $x\in\Omega\cup\mathcal{I}$ and notice that $\tau_{\partial\Omega}>0$ $\mathbb{P}_x$ almost-surely by the definition of $\mathcal{I}$ and the continuity of the trajectories. The proof relies on small-time asymptotics of the process $(X(t))_{t\geq0}$ detailed below. Given the H\"olderian nature of the trajectories of $(p(t))_{t\geq0}$, we have that for any $\alpha\in(0,\,1/2)$, $\mathbb{P}_x$ almost-surely,
\begin{equation}\label{eq:small time asymp vitesse}
\lim_{t\rightarrow0}\frac{1}{t^\alpha}\big|p(\tau_{\partial \Omega}-t)-p(\tau_{\partial \Omega})\big|=0\;.
\end{equation}
Also, $\mathbb{P}_x$ almost-surely,
\begin{equation}\label{eq:small time asymp position}
\lim_{t\rightarrow0}\frac{1}{t}\big(q(\tau_{\partial \Omega})-q(\tau_{\partial \Omega}-t)\big)=\lim_{t\rightarrow0}\frac{1}{t}\int_{\tau_{\partial \Omega}-t}^{\tau_{\partial \Omega}}G(p(s))\mathrm{d}s=G(p(\tau_{\partial \Omega}))\;.    
\end{equation}
Define the process
$$Y:t\in[0,\,\tau_{\partial \Omega}]\mapsto X(\tau_{\partial \Omega}-t)-X(\tau_{\partial \Omega})\;.$$
By the Taylor-Lagrange inequality, $\mathbb{P}_x$ almost-surely,
\begin{align}
&\mathrm{d}_\Omega(X(\tau_{\partial \Omega}-t))\\
&=\mathrm{d}_\Omega(X(\tau_{\partial \Omega}))+\langle\nabla\mathrm{d}_{\Omega}(X(\tau_{\partial \Omega})),\,Y(t)\rangle+\frac{1}{2}\langle Y(t),\,\nabla^2\mathrm{d}_{\Omega}(X(\tau_{\partial \Omega}))Y(t)\rangle+\underset{t\rightarrow0}{O}(|Y(t)|^3)\nonumber\\
&=\langle\nabla\mathrm{d}_{\Omega}(X(\tau_{\partial \Omega})),\,Y(t)\rangle+\frac{1}{2}\langle Y(t),\,\nabla^2\mathrm{d}_{\Omega}(X(\tau_{\partial \Omega}))Y(t)\rangle+\underset{t\rightarrow0}{o}(t)\label{eq:Taylor Lagrange}\;,
\end{align}
given the small-time asymptotics in~\eqref{eq:small time asymp vitesse} and~\eqref{eq:small time asymp position}. Moreover, under the event
$$\big\{\nabla^2_{p,\,p}\mathrm{d}_\Omega(X(\tau_{\partial \Omega}))=\mathbb{O}_d,\,X(\tau_{\partial \Omega})\in\Gamma^+\big\}\;,$$
the limits~\eqref{eq:small time asymp vitesse} and~\eqref{eq:small time asymp position} guarantee that, $\mathbb{P}_x$ almost-surely,
\begin{align*}
&\langle Y(t),\,\nabla^2\mathrm{d}_{\Omega}(X(\tau_{\partial \Omega}))Y(t)\rangle\\
&=\langle q(\tau_{\partial \Omega}-t)-q(\tau_{\partial \Omega}),\,\nabla^2_{q,\,q}\mathrm{d}_{\Omega}(X(\tau_{\partial \Omega}))(q(\tau_{\partial \Omega}-t)-q(\tau_{\partial \Omega}))\rangle\\
&+2\langle q(\tau_{\partial \Omega}-t)-q(\tau_{\partial \Omega}),\,\nabla^2_{q,\,p}\mathrm{d}_{\Omega}(X(\tau_{\partial \Omega}))(p(\tau_{\partial \Omega}-t)-p(\tau_{\partial \Omega}))\rangle\\
&=\underset{t\rightarrow0}{o}(t)\;. 
\end{align*}
Furthermore, $\mathbb{P}_x$ almost-surely,
\begin{align*}
\frac{1}{t}\langle\nabla\mathrm{d}_{\Omega}(X(\tau_{\partial \Omega})),\,Y(t)\rangle&=\frac{1}{t}\langle n_q(X(\tau_{\partial \Omega})),\,q(\tau_{\partial \Omega}-t)-q(\tau_{\partial \Omega})\rangle\\
&=-\frac{1}{t}\int_{\tau_{\partial \Omega}-t}^{\tau_{\partial \Omega}}\langle n_q(X(\tau_{\partial \Omega})),\,G(p(s))\rangle\mathrm{d}s\\
&\underset{t\rightarrow0}{\longrightarrow}-\xi(X(\tau_{\partial \Omega}))<0
\end{align*}
since $X(\tau_{\partial \Omega})\in\Gamma^+$. Therefore,~\eqref{eq:Taylor Lagrange} guarantees that
$$\lim_{t\rightarrow0}\frac{1}{t}\mathrm{d}_\Omega(X(\tau_{\partial \Omega}-t))=-\xi(X(\tau_{\partial \Omega}))<0$$
which contradicts the fact that
$$\inf_{t\in[0,\,\tau_{\partial \Omega}]}\mathrm{d}_\Omega(X(t))\geq0$$
since the process $(X(t))_{t\geq0}$ stays in $\Omega$ during the time interval $(0,\,\tau_{\partial \Omega})$, hence the proof of~\eqref{eq:exit event x in Omega}.

Combining~\eqref{eq:exit event x in Omega} with Theorem~\ref{thm:polarity}, we deduce that the equality~\eqref{eq:bdy exit event} also holds. Given that the boundary set $\Lambda$ is composed only of regular points, the equality~\eqref{eq:proba hitting ext vs bdy} immediately follows which concludes the proof.
\end{proof}

\section{Boundary-value problem: proof of Theorem~\ref{thm:edp eq pot} and Theorem~\ref{thm:existence sol PDE}}\label{sec:prop function h}
In this section, we first prove the uniqueness of a classical solution to the boundary-value problem stated in Theorem~\ref{thm:edp eq pot}. Existence through the probabilistic representation~\eqref{eq:def fctn h} and the boundary continuity properties (Theorem~\ref{thm:existence sol PDE}) are given by the two propositions stated below, which will be independently proven in  Section~\ref{sec:classical sol} and Section~\ref{sec:continuity h} respectively.

\begin{prop}[Classical solution]\label{prop:classical sol}
Let $f\in C_b(\Lambda)$, $c,g\in C^\infty_b(\Omega)$. The function $h$ defined in~\eqref{eq:def fctn h} satisfies
\begin{itemize}
    \item $h\in C^\infty(\Omega)$,
    \item $\mathcal{L}h+ch=g$ on $\Omega$.
\end{itemize}
\end{prop}
\begin{prop}[Boundary continuity]\label{prop:boundary continuity h} Under the assumptions of Proposition~\ref{prop:classical sol}, for all $x\in\mathcal{R}_f$,
$$\lim_{y\in\Omega\rightarrow x}h(y)=f(x)\;.$$
For all $x\in\mathrm{int}_{\partial\Omega}\mathcal{I}$,
$$\lim_{y\in\Omega\rightarrow x}h(y)=h(x)\;.$$
Additionally, if $(x_n)_{n\geq1}$ is a sequence in $\Omega$ converging to $x\in\mathcal{I}\setminus\mathrm{int}_{\partial\Omega}\mathcal{I}$ and satisfying~\eqref{eq:cone condition} then
$$\lim_{n\rightarrow\infty}h(x_n)=h(x)\;.$$
\end{prop}

Uniqueness in Theorem~\ref{thm:edp eq pot} is obtained using standard probabilistic arguments by showing that a classical solution necessarily identifies with the probabilistic representation $h$ defined in~\eqref{eq:def fctn h}.
\begin{proof}[Proof of uniqueness in Theorem~\ref{thm:edp eq pot}]
For $k\geq1$, define the open set
$$V_k:=\big\{x\in\Omega:\,|x|<k,\,\mathrm{d}_\Omega(x)>1/k\big\}\;.$$
Consider the sequence $(\tau_{V_k^c})_{k\geq1}$ of first exit times from $V_k$ of the process~\eqref{eq:sde}. The non-explosion of the process~\eqref{eq:sde} guarantees that $(\tau_{V_k^c})_{k\geq1}$ increases towards its limit $\tau_{\Omega^c}$ when $k\rightarrow\infty$. Consider now $u\in C^2(\Omega)\cap C_b(\Omega\cup\Lambda)$ satisfying~\eqref{eq:edp h}. Applying It\^o's formula to the process
$$t\geq0\mapsto u(X(t\land\tau_{V_k^c}))\,\mathrm{e}^{\int_0^{t\land\tau_{V_k^c}}c(X(s))\mathrm{d}s}\;,$$
we obtain that for all $x\in\R^{2d}$, for all $t\geq0$,
$$\mathbb{E}_x\bigg[u(X(t\land\tau_{V_k^c}))\,\mathrm{e}^{\int_0^{t\land\tau_{V_k^c}}c(X(s))\mathrm{d}s} -\int_0^{t\land\tau_{V_k^c}}\mathrm{e}^{\int_0^sc(X(r))\mathrm{d}r}g(X(s))\mathrm{d}s\bigg]=u(x)\;.$$
Since $u$ and $g$ are bounded functions, the limit $k\rightarrow\infty$ yields
$$\mathbb{E}_x\bigg[u(X(t\land\tau_{\Omega^c}))\,\mathrm{e}^{\int_0^{t\land\tau_{\Omega^c}}c(X(s))\mathrm{d}s}-\int_0^{t\land\tau_{\Omega^c}}\mathrm{e}^{\int_0^sc(X(r))\mathrm{d}r}g(X(s))\mathrm{d}s\bigg]=u(x)\;.$$
Since $u$ and $f$ coincide on $\Lambda$, taking the limit $t\rightarrow\infty$ ensures by the dominated convergence theorem and Theorem~\ref{thm:polarity Gamma+} that
$$u(x)=h(x)\;,$$
which concludes the proof.
\end{proof}

\subsection{Proof of Proposition~\ref{prop:classical sol}}\label{sec:classical sol}

The proof of Proposition~\ref{prop:classical sol} relies on the following lemmas.
\begin{lem}[Exit time]\label{lem:exit time compact set}
For any compact set $K\subset\Omega$, there exist constants $c_1,\,c_2,\,\delta>0$ such that for all $t\in(0,\,\delta)$,
$$\sup_{x\in K}\mathbb{P}_x(\tau_{\Omega^c}\leq t)\leq c_1\mathrm{e}^{-c_2/t}\;.$$
\end{lem}
\begin{lem}\label{lem:smoothness P_tf}
For any $f\in\mathrm{L}^\infty(\R^{2d})$, the function
\begin{equation}\label{eq:def fctn u}
Pf:(t,\,x)\in\R_+^*\times\R^{2d}\mapsto P_tf(x):=\mathbb{E}_x\bigg[\mathrm{e}^{\int_0^{t}c(X(s))\mathrm{d}s}f(X(t))-\int_0^{t}\mathrm{e}^{\int_0^sc(X(r))\mathrm{d}r}g(X(s))\mathrm{d}s\bigg]
\end{equation}
is in $C^\infty(\R_+^*\times\R^{2d})$. Moreover, for all $t>0$, for all $x\in\R^{2d}$,
\begin{equation}\label{eq:pde semigroup}
\partial_tP_tf(x)=\mathcal{L}P_tf(x)+c(x)P_tf(x)-g(x)\;.
\end{equation}
\end{lem}
\begin{proof}[Proof of Lemma~\ref{lem:exit time compact set}]
Given that $K$ is a compact subset of $\Omega$, there exists a compact set $K_*$ such that
$$K\subset K_*\subset\Omega,\,\quad\delta:=\mathrm{dist}(K,\,K_*^c)>0\;.$$
As a result, for any $x\in K$,
\begin{align}
\mathbb{P}_x(\tau_{\Omega^c}\leq t)&\leq\mathbb{P}_x(\tau_{K_*^c}\leq t)\nonumber\\
&\leq\mathbb{P}_x\bigg(\sup_{s\in[0,\,t\land\tau_{K_*^c}]}|X(s)-x|>\delta\bigg)\label{eq:ineq exit time compact}\;.
\end{align}
Given that the coefficients $G,\,F$ are bounded on the compact set $K_*$, there exists a constant $C_1>0$ such that, $\mathbb{P}_x$ almost-surely, for all $s\in[0,\,t\land\tau_{K_*^c}]$,
\begin{align*}
|X(s)-x|&\leq C_1t+\sup_{s\in[0,\,t\land\tau_{K_*^c}]}\bigg|\int_0^{s}\Sigma(X(r))\mathrm{d}B(r)\bigg|\\
&\leq C_1t+\sum_{1\leq i,\,k\leq d}\,\sup_{s\in[0,\,t\land\tau_{K_*^c}]}\bigg|\int_0^{s}\Sigma_{i,\,k}(X(r))\mathrm{d}B_k(r)\bigg|\;.
\end{align*}
In particular, on the event $\{\tau_{K_*^c}\leq t\}$, for $t\geq0$ small enough, there exist $i,\,k\in\llbracket1,\,d\rrbracket$ such that
$$\sup_{s\in[0,\,t\land\tau_{K_*^c}]}\bigg|\int_0^{s}\Sigma_{i,\,k}(X(r))\mathrm{d}B_k(r)\bigg|\geq\frac{\delta}{2d^2}\;.$$
By the Dubins-Schwarz theorem, there exists a Brownian motion $(W^{i,\,k}(t))_{t\geq0}$ such that the inequality above implies that
$$\frac{\delta}{2d^2}\leq\sup_{s\in\big[0,\,\int_0^{t\land\tau_{K_*^c}}\Sigma_{i,\,k}(X(r))^2\mathrm{d}r\big]}|W^{i,\,k}(s)|\leq\sup_{s\in[0,\,C_2t]}|W^{i,\,k}(s)|$$
for some constant $C_2>0$ independent of $i,\,k\in\llbracket1,\,d\rrbracket$ given that $\Sigma$ is uniformly bounded on $K_*$. All in all, by~\eqref{eq:ineq exit time compact},
\begin{align*}
\mathbb{P}_x(\tau_{\Omega^c}\leq t)&\leq\sum_{1\leq i,\,k\leq d}\mathbb{P}\bigg(\sup_{s\in[0,\,C_2t]}|W^{i,\,k}(s)|>\delta/2\bigg)\\
&\leq2d^2\mathrm{e}^{-\delta^2/(8C_2t)} 
\end{align*}
which concludes the proof.
\end{proof}
\begin{proof}[Proof of Lemma~\ref{lem:smoothness P_tf}]
The idea of the proof consists in approaching $Pf$ by smooth solutions of a perturbation of the operator $\mathcal{L}$, which allows us to deduce that $Pf$ is a distributional solution on $\R_+^*\times\R^{2d}$ of~\eqref{eq:pde semigroup}. The hypoellipticity property of the operator $\mathcal{L}$ shown in Lemma~\ref{lem: hypoellip} then concludes the proof.

Following the proof of Lemma~\ref{lem:upper bound density}, let us define for any $M>0$ $C^\infty$ functions $G_M,\,F_M,\,\Sigma_M$ coinciding with $G,\,F,\,\Sigma$ on the ball $\mathrm{B}(0,\,M)$, which are bounded along with their successive derivatives on $\R^{2d}$ and such that $\Sigma_M$ is invertible on $\R^{2d}$. The solution to~\eqref{eq:sde} with the coefficients $G_M,\,F_M,\,\Sigma_M$ is denoted by $(X^M(t))_{t\geq0}$. For $\alpha>0$, let us also define the process $(X^{\alpha,\,M}(t)=(q^{\alpha,\,M}(t),\,p^{\alpha,\,M}(t)))_{t\geq0}$ solution to
\begin{equation}\label{eq:sde perturbed}
\left\{ \begin{aligned} & \mathrm{d}q^{\alpha,\,M}(t)=G_M(p^{\alpha,\,M}(t))\mathrm{d}t+\sqrt{2\alpha}\mathrm{d}\widetilde{B}(t)\;,\\
 & \mathrm{d}p^{\alpha,\,M}(t)=F_M(q^{\alpha,\,M}(t),\,p^{\alpha,\,M}(t))\mathrm{d}t+\Sigma_M(q^{\alpha,\,M}(t),\,p^{\alpha,\,M}(t))\mathrm{d}B(t)\;,
\end{aligned}
\right.
\end{equation}
where $(\widetilde{B}(t))_{t\geq0}$ is a Brownian motion in $\R^d$ independent of $(B(t))_{t\geq0}$.

\textbf{Step 1}: Let us first prove that for all $t>0$, for all $x\in\R^{2d}$,
\begin{equation}\label{eq:cv sup L1 X_alpha to X}
\mathbb{E}_x\bigg[\sup_{s\in[0,\,t]}\big|X^{\alpha,\,M}(s)-X^M(s)\big|\bigg]\underset{\alpha\rightarrow0}{\longrightarrow}0\;.
\end{equation}
Given that the coefficients $G_M,\,F_M$ are globally Lipschitz continuous, there exists a constant $C_1>0$ such that, $\mathbb{P}_x$ almost-surely, for all $s\in[0,\,t]$,
\begin{align*}
\big|X^{\alpha,\,M}(s)-X^M(s)\big|&\leq C_1\int_0^{s}\big|X^{\alpha,\,M}(r)-X^M(r)\big|\mathrm{d}r+\bigg|\int_0^{s}\big[\Sigma_M(X^{\alpha,\,M}(r))-\Sigma_M(X^M(r))\big]\mathrm{d}B(r)\bigg|\\
&+\sqrt{2\alpha}|\widetilde{B}(s)|\;.
\end{align*}
Therefore, by the Cauchy-Schwarz inequality, for all $s\in[0,\,t]$,
\begin{align*}
\big|X^{\alpha,\,M}(s)-X^M(s)\big|^2&\leq 3tC_1^2\int_0^{s}\big|X^{\alpha,\,M}(r)-X^M(r)\big|^2\mathrm{d}r\\
&+3\,\sup_{r\in[0,\,s]}\bigg|\int_0^{r}\big[\Sigma_M(X^{\alpha,\,M}(u))-\Sigma_M(X^M(u))\big]\mathrm{d}B(u)\bigg|^2+6\alpha\,\sup_{r\in[0,\,t]}|\widetilde{B}(r)|^2\;.
\end{align*}
Furthermore, by Doob's inequality,
\begin{align*}
&\mathbb{E}_x\bigg[\sup_{r\in[0,\,s]}\bigg|\int_0^{r}\big[\Sigma_M(X^{\alpha,\,M}(u))-\Sigma_M(X^M(u))\big]\mathrm{d}B(u)\bigg|^2\bigg]\\
&\leq4\mathbb{E}_x\bigg[\int_0^{s}\mathrm{Tr}\big(\big(\Sigma_M(X^{\alpha,\,M}(r))-\Sigma_M(X^M(r))\big)\big(\Sigma^\dagger_M(X^{\alpha,\,M}(r))-\Sigma^\dagger_M(X^M(r))\big)\big)\mathrm{d}r\bigg]\;.
\end{align*}
As a result, given that $\Sigma_M$ is globally Lipschitz-continuous, we deduce the existence of a constant $C_2>0$ such that for all $s\in[0,\,t]$,
$$\mathbb{E}_x\bigg[\sup_{r\in[0,\,s]}\bigg|\int_0^{r}\big[\Sigma_M(X^{\alpha,\,M}(u))-\Sigma_M(X^M(u))\big]\mathrm{d}B(u)\bigg|^2\bigg]\leq C_2\mathbb{E}_x\bigg[\int_0^{s}\big|X^{\alpha,\,M}(r)-X^M(r)\big|^2\mathrm{d}r\bigg]\;.$$
It follows from the inequalities above that for all $s\in[0,\,t]$,
\begin{align*}
\mathbb{E}_x\bigg[\sup_{s\in[0,\,t]}\big|X^{\alpha,\,M}(s)-X^M(s)\big|^2\bigg]\leq&(3tC_1^2+3C_2)\int_0^{t}\mathbb{E}_x\bigg[\sup_{r\in[0,\,s]}\big|X^{\alpha,\,M}(r)-X^M(r)\big|^2\bigg]\mathrm{d}s\\
&+6\alpha\,\mathbb{E}\bigg[\sup_{s\in[0,\,t]}\big|\widetilde{B}(s)\big|^2\bigg]\;.
\end{align*}
By Gr\"onwall's inequality, we deduce the existence of a constant $C_3>0$ independent of $\alpha>0$ such that
$$\mathbb{E}_x\bigg[\sup_{s\in[0,\,t]}\big|X^{\alpha,\,M}(s)-X^M(s)\big|^2\bigg]\leq C_3\alpha\underset{\alpha\rightarrow0}{\longrightarrow}0\;.$$

For $M>0$, define now functions $c_M,\,g_M$ in $C^\infty_c(\R^{2d})$ which coincide with $c,\,g\in C_b^\infty(\R^{2d})$ on the open ball $\mathrm{B}(0,\,M)$ and such that
\begin{equation}\label{eq:unif boundedness c_M g_M}
\sup_{M>0}\big(||c_M||_\infty+||g_M||_\infty\big)<\infty\;.
\end{equation}
Let also $f\in\mathrm{L}^\infty(\R^{2d})$ and let $(f_n)_{n\geq1}$ be a uniformly bounded sequence in $C_c^\infty(\R^{2d})$ converging almost-everywhere to $f$. For $\alpha>0$ and $n\geq1$, define
$$u_{M,\,\alpha,\,n}:(t,\,x)\in\R_+^*\times\R^{2d}\mapsto\mathbb{E}_x\bigg[\mathrm{e}^{\int_0^{t}c_M(X^{\alpha,\,M}(s))\mathrm{d}s}f_n(X^{\alpha,\,M}(t))-\int_0^{t}\mathrm{e}^{\int_0^sc_M(X^{\alpha,\,M}(r))\mathrm{d}r}g_M(X^{\alpha,\,M}(s))\mathrm{d}s\bigg]\;.$$
\textbf{Step 2:} Let us now prove that for all $x\in\R^{2d}$, for all $t>0$,
\begin{equation}\label{eq:cv u_M alpha to u}
\lim_{M\rightarrow\infty}\lim_{n\rightarrow\infty}\lim_{\alpha\rightarrow0}u_{M,\,\alpha,\,n}(t,\,x)=P_tf(x)\;.
\end{equation}
By~\eqref{eq:cv sup L1 X_alpha to X}, we deduce from the dominated convergence theorem that for all $(t,\,x)\in\R_+^*\times\R^{2d}$, for all $M>0$, for all $n\geq1$,
$$u_{M,\,\alpha,\,n}(t,\,x)\underset{\alpha\rightarrow0}{\longrightarrow}u_{M,\,n}(t,\,x):=\mathbb{E}_x\bigg[\mathrm{e}^{\int_0^{t}c_M(X^{M}(s))\mathrm{d}s}f_n(X^{M}(t))-\int_0^{t}\mathrm{e}^{\int_0^sc_M(X^{M}(r))\mathrm{d}r}g_M(X^{M}(s))\mathrm{d}s\bigg]\;.$$
Since the vector $X^M(t)$ admits a density with respect to the Lebesgue measure by~\cite[Theorem 2.1]{Menozzi}, we deduce from the convergence of $f_n$ that
$$u_{M,\,n}(t,\,x)\underset{n\rightarrow\infty}{\longrightarrow}u_{M}(t,\,x):=\mathbb{E}_x\bigg[\mathrm{e}^{\int_0^{t}c_M(X^{M}(s))\mathrm{d}s}f(X^{M}(t))-\int_0^{t}\mathrm{e}^{\int_0^sc_M(X^{M}(r))\mathrm{d}r}g_M(X^{M}(s))\mathrm{d}s\bigg]\;.$$
In addition, since the coefficients $F_M,\,G_M,\,\Sigma_M,\,c_M,\,g_M$ coincide on the ball $\mathrm{B}(0,\,M)$, the process $X$ and $X^M$ coincide until the exit time $\tau_{\mathrm{B}(0,\,M)^c}$. Therefore,
\begin{align*}
u_M(t,\,x)&=\mathbb{E}_x\bigg[\mathbf{1}_{\tau_{\mathrm{B}(0,\,M)^c}>t}\bigg(\mathrm{e}^{\int_0^{t}c(X(s))\mathrm{d}s}f(X(t))-\int_0^{t}\mathrm{e}^{\int_0^sc(X(r))\mathrm{d}r}g(X(s))\mathrm{d}s\bigg)\bigg]\\
&+\mathbb{E}_x\bigg[\mathbf{1}_{\tau_{\mathrm{B}(0,\,M)^c}\leq t}\bigg(\mathrm{e}^{\int_0^{t}c_M(X^{M}(s))\mathrm{d}s}f(X^{M}(t))-\int_0^{t}\mathrm{e}^{\int_0^sc_M(X^{M}(r))\mathrm{d}r}g_M(X^{M}(s))\mathrm{d}s\bigg)\bigg]\;.
\end{align*}
As a result, given the boundedness of the functions $f,\,g,\,c$ and using~\eqref{eq:unif boundedness c_M g_M}, there exists a constant $C>0$ independent of $M>0$ such that for all $x\in\R^{2d}$, for all $t>0$,
$$|u_M(t,\,x)-P_tf(x)|\leq C\mathrm{e}^{Ct}\mathbb{P}_x(\tau_{\mathrm{B}(0,\,M)^c}\leq t)\underset{M\rightarrow\infty}{\longrightarrow}0\;,$$
hence~\eqref{eq:cv u_M alpha to u}.

\textbf{Step 3:} Let us now deduce that $Pf$ is a distributional solution in $\R_+^*\times\R^{2d}$ of~\eqref{eq:pde semigroup}. The hypoellipticity of the operator $\mathcal{L}$ shown in Lemma~\ref{lem: hypoellip} would then ensure that $Pf\in C^\infty(\R_+^*\times\R^{2d})$ and satisfies~\eqref{eq:pde semigroup} in the classical sense. 

By~\cite[Theorem 5.3]{F}, $u_{M,\,n,\,\alpha}$ is a classical solution in $\R_+^*\times\R^{2d}$ of 
$$\partial_tu_{M,\,n,\,\alpha}=\mathcal{L}_{\alpha,\,M}u_{M,\,n,\,\alpha}+c_Mu_{M,\,n,\,\alpha}-g_M$$
where $\mathcal{L}_{\alpha,\,M}$ is the infinitesimal generator of the process~\eqref{eq:sde perturbed}. In particular, $u_{M,\,n,\,\alpha}$ is also a solution in the sense of distributions. As a result, taking an arbitrary $\phi\in C_c^\infty(\R_+^*\times\R^{2d})$ ensures that for $\alpha,\,M>0$ and $n\geq1$,
$$\iint_{\R_+^*\times\R^{2d}}u_{M,\,n,\,\alpha}(t,\,x)\big(\partial_t\phi(t,\,x)+\mathcal{L}_{\alpha,\,M}^\dagger\phi(t,\,x)+c_M(x)\phi(t,\,x)\big)\mathrm{d}t\mathrm{d}x=\iint_{\R_+^*\times\R^{2d}}g_M(x)\phi(t,\,x)\mathrm{d}t\mathrm{d}x$$
where $\mathcal{L}_{\alpha,\,M}^\dagger$ is the adjoint operator of $\mathcal{L}_{\alpha,\,M}$ in $\mathrm{L}^2(\mathrm{d}x)$. Take $M>0$ large enough such that the support of $\phi$ is contained in $\R_+^*\times\mathrm{B}(0,\,M)$, then
$$\iint_{\R_+^*\times\R^{2d}}u_{M,\,n,\,\alpha}(t,\,x)\big(\partial_t\phi(t,\,x)+\mathcal{L}_{\alpha}^\dagger\phi(t,\,x)+c(x)\phi(t,\,x)\big)\mathrm{d}t\mathrm{d}x=\iint_{\R_+^*\times\R^{2d}}g(x)\phi(t,\,x)\mathrm{d}t\mathrm{d}x$$
where $\mathcal{L}_\alpha$ is the infinitesimal generator of the process~\eqref{eq:sde perturbed} where the coefficients $G_M,\,F_M,\,\Sigma_M$ are replaced by $G,\,F,\,\Sigma$. Taking the successive limits $\alpha\rightarrow0$, $n\rightarrow\infty$ and $M\rightarrow\infty$ ensures by~\eqref{eq:cv u_M alpha to u} that $Pf$ is a distributional solution of~\eqref{eq:pde semigroup} on $\R_+^*\times\R^{2d}$, which concludes the proof.
\end{proof}

\begin{proof}[Proof of Proposition~\ref{prop:classical sol}]

\textbf{Step 1:} Let us prove that for any compact set $K\subset\Omega$, there exist constants $c_1,\,c_2>0$ such that for all $x\in K$, for $t\geq0$ sufficiently small,
\begin{equation}\label{eq:control variation P_th}
\bigg|\frac{P_th(x)-h(x)}{t}\bigg|\leq\frac{c_1\mathrm{e}^{-c_2/t}}{t}\;.
\end{equation}
For $x\in\Omega$, one has that for all $t\geq0$,
\begin{align*}
h(x)&=\mathbb{E}_x\bigg[\mathbf{1}_{\tau_{\Omega^c}>t}\bigg(\mathrm{e}^{\int_0^{\tau_{\Omega^c}}c(X(s))\mathrm{d}s}f(X(\tau_{\Omega^c}))-\int_0^{\tau_{\Omega^c}}\mathrm{e}^{\int_0^sc(X(r))\mathrm{d}r}g(X(s))\mathrm{d}s\bigg)\bigg]\\
&+\mathbb{E}_x\bigg[\mathbf{1}_{\tau_{\Omega^c}\leq t}\bigg(\mathrm{e}^{\int_0^{\tau_{\Omega^c}}c(X(s))\mathrm{d}s}f(X(\tau_{\Omega^c}))-\int_0^{\tau_{\Omega^c}}\mathrm{e}^{\int_0^sc(X(r))\mathrm{d}r}g(X(s))\mathrm{d}s\bigg)\bigg]
\end{align*}
Applying the Markov property at time $t\geq0$ yields the equality
\begin{align*}
&\mathbb{E}_x\bigg[\mathbf{1}_{\tau_{\Omega^c}>t}\bigg(\mathrm{e}^{\int_0^{\tau_{\Omega^c}}c(X(s))\mathrm{d}s}f(X(\tau_{\Omega^c}))-\int_0^{\tau_{\Omega^c}}\mathrm{e}^{\int_0^sc(X(r))\mathrm{d}r}g(X(s))\mathrm{d}s\bigg)\bigg]\\
&=\mathbb{E}_x\bigg[\mathbf{1}_{\tau_{\Omega^c}>t}\bigg(\mathrm{e}^{\int_0^{t}c(X(s))\mathrm{d}s}h(X(t))-\int_0^{t}\mathrm{e}^{\int_0^sc(X(r))\mathrm{d}r}g(X(s))\mathrm{d}s\bigg)\bigg]\\
&=P_th(x)-\mathbb{E}_x\bigg[\mathbf{1}_{\tau_{\Omega^c}\leq t}\bigg(\mathrm{e}^{\int_0^{t}c(X(s))\mathrm{d}s}h(X(t))-\int_0^{t}\mathrm{e}^{\int_0^sc(X(r))\mathrm{d}r}g(X(s))\mathrm{d}s\bigg)\bigg]\;.
\end{align*}
As a result, since the function $h$ is bounded by Assumption~\ref{ass:Omega} and the functions $f,\,g,\,c$ are bounded as well, there exist constants $C_1,\,C_2>0$ such that for all $x\in\Omega$, for all $t\geq0$,
\begin{align*}
|h(x)-P_th(x)|\leq C_1\mathbb{P}_x(\tau_{\Omega^c}\leq t)\mathrm{e}^{C_2t}
\end{align*}
The inequality~\eqref{eq:control variation P_th} is then an immediate consequence of Lemma~\ref{lem:exit time compact set}.

\textbf{Step 2:} When the time $t\geq0$ is small, the process $(X(s))_{s\in[0,\,t]}$ remains in a neighborhood of $X(0)$ and only depends on the values of the coefficients in~\eqref{eq:sde} in this neighborhood. Therefore, under this condition we can assume that the coefficients of the process $(X(s))_{s\in[0,\,t]}$ are $C^\infty$, bounded and globally Lipschitz continuous which ensures that the transition density of $(X(s))_{s\in[0,\,t]}$ admits the Gaussian upper-bound given in~\eqref{eq:gaussian upper bound} by~\cite[Theorem 2.1]{Menozzi}. Using a standard change of variable, we easily deduce that for any compact set $K\subset\Omega$,
\begin{equation}\label{eq:continuity P_th at zero}
\int_K|P_rh(x)-h(x)|\mathrm{d}x\underset{r\rightarrow0}{\longrightarrow}0\;.
\end{equation}
Let us now take $\phi\in C_c^\infty(\Omega)$. From~\eqref{eq:control variation P_th} we deduce that
$$\lim_{t\rightarrow0}\lim_{r\rightarrow0}\int_{\Omega}\frac{(P_th(x)-P_rh(x))}{t}\phi(x)\mathrm{d}x=0\;.$$
In addition, let us now prove that
\begin{align}
&\lim_{t\rightarrow0}\lim_{r\rightarrow0}\int_{\Omega}\frac{(P_th(x)-P_rh(x))}{t}\phi(x)\mathrm{d}x\nonumber\\
&=\int_{\Omega}h(x)\big(\mathcal{L}^\dagger\phi(x)+c(x)\phi(x)\big)\mathrm{d}x-\int_{\Omega}g(x)\phi(x)\mathrm{d}x\label{eq:h distri sol}\;.
\end{align}
Combining both limits we shall deduce that $h$ is a distributional solution of $\mathcal{L}h+ch=g$ on $\Omega$ and the hypoellipticity of the operator $\mathcal{L}$ concludes the proof. Therefore, let us now prove~\eqref{eq:h distri sol}.

For $0<r<t$, by Lemma~\ref{lem:smoothness P_tf}, we have that
$$\int_{\Omega}\frac{(P_th(x)-P_rh(x))}{t}\phi(x)\mathrm{d}x=\int_{\Omega}\frac{1}{t}\int_r^t(\mathcal{L}P_sh(x)+c(x)P_sh(x)-g(x))\mathrm{d}s\phi(x)\mathrm{d}x\;.$$
In addition, since $\phi\in C_c^\infty(\Omega)$, by the Fubini permutation theorem,
\begin{align*}
&\int_{\Omega}\frac{1}{t}\int_r^t(\mathcal{L}P_sh(x)+c(x)P_sh(x)-g(x))\mathrm{d}s\phi(x)\mathrm{d}x\\
&=\frac{1}{t}\iint_{\Omega\times(r,\,t)}(\mathcal{L}P_sh(x)+c(x)P_sh(x)-g(x))\phi(x)\mathrm{d}s\mathrm{d}x\\
&=\frac{1}{t}\iint_{\Omega\times(r,\,t)}P_sh(x)(\mathcal{L}^\dagger\phi(x)+c(x)\phi(x))\mathrm{d}s\mathrm{d}x-\frac{(t-r)}{t}\int_{\Omega}g(x)\phi(x)\mathrm{d}x
\end{align*}
by an integration-by-parts. Taking the limit $r\rightarrow0$, the quantity above is equal to
\begin{align*}
&\int_\Omega\frac{1}{t}\int_{0}^tP_sh(x)\mathrm{d}s(\mathcal{L}^\dagger\phi(x)+c(x)\phi(x))\mathrm{d}x-\int_{\Omega}g(x)\phi(x)\mathrm{d}x\\
&=\int_\Omega\int_{0}^1P_{st}h(x)\mathrm{d}s(\mathcal{L}^\dagger\phi(x)+c(x)\phi(x))\mathrm{d}x-\int_{\Omega}g(x)\phi(x)\mathrm{d}x\;.
\end{align*}
Combining with~\eqref{eq:continuity P_th at zero}, we deduce~\eqref{eq:h distri sol} by taking the limit $t\rightarrow0$ in the equality above.
\end{proof}

\subsection{Proof of Proposition~\ref{prop:boundary continuity h}}\label{sec:continuity h}

Let us introduce additional notations in this section. For $x\in\R^{2d}$, let $(X^x(t))_{t\geq0}$ be the process~\eqref{eq:sde} satisfying $X(0)=x$ and let $\tau_A^x$ be its first hitting time of a set $A\subset\R^{2d}$.  

 The main difficulty in the proof of Proposition~\ref{prop:boundary continuity h} appears when considering the continuity of $h$ at irregular boundary points. Let $(x_n)_{n\geq0}$ be a converging sequence to $x\in\mathcal{I}$. If $(x_n)_{n\geq0}$ remains too close to regular boundary points, the process $(X^{x_n}(t))_{t\geq0}$ may exit $\Omega$ for sufficiently small times and the convergence $\tau^{x_n}_{\Omega^c}$ to $\tau^x_{\Omega^c}$ fails to happen. This critical situation is avoided when $x\in\mathrm{int}_{\partial\Omega}(\mathcal{I})$ or when the cone condition~\eqref{eq:cone condition} is satisfied. Regarding the continuity of $h$ at regular boundary points, standard arguments involving trajectorial convergence of the process allow to conclude. The proof of Proposition~\ref{prop:boundary continuity h} relies therefore on intermediary results stated below.

 For $M>0$, let us define functions $G_M,\,F_M,\,\Sigma_M$ in $C_c^\infty(\R^{2d})$ which coincide with $G,\,F,\,\Sigma$ on the open ball $\mathrm{B}(0,\,M)$. Denote also by $(X^{M,\,x}(t))_{t\geq0}$ the solution to~\eqref{eq:sde} with coefficients $G_M,\,F_M,\,\Sigma_M$ such that $X^{M,\,x}(0)=x$. Its first hitting time of a set $A$ is denoted by $\tau^{M,\,x}_A$. Let us now consider a sequence $(x_n)_{n\geq0}$ of points in $\Omega$ converging to $x\in\overline{\Omega}$.
\begin{lem}[Exit time]\label{lem:rapid exit time} Assume that $x\in\mathrm{int}_{\partial\Omega}(\mathcal{I})$ or that $x\in\mathcal{I}$ satisfies~\eqref{eq:cone condition}. Then,
$$\limsup_{n\rightarrow\infty}\mathbb{P}\big(\tau^{M,\,x_n}_{\Omega^c}\leq a)\underset{a\rightarrow0}{\longrightarrow}0\;.$$
\end{lem}
\begin{lem}[Convergence of trajectories]\label{lem:subseq x_phi_n} Assume that $x\in\mathcal{R}\cup\mathrm{int}_{\partial\Omega}(\mathcal{I})$, or that~\eqref{eq:cone condition} is satisfied. For all $T\geq0$, %there exists a subsequence $(x_{\phi(n)})_{n\geq0}$ such that, almost-surely,
the following convergences in probability hold:
\begin{equation}\label{eq:cv sup X x_n and X with subseq}
\sup_{t\in[0,\, T]}|X^{x_{n}}(t)-X^x(t)|\overset{\mathbb{P}}{\underset{n\rightarrow\infty}{\longrightarrow}}0
\end{equation}
and
\begin{equation}\label{eq:cv tau x_n and tau with subseq}
\tau^{x_{n}}_{\Omega^c}\overset{\mathbb{P}}{\underset{n\rightarrow\infty}{\longrightarrow}}\tau^x_{\Omega^c}\;.
\end{equation}
\end{lem}
 Notice that since $G_M,\,F_M,\,\Sigma_M$ are $C^\infty$ and admit bounded derivatives of all orders, it was shown in~\cite[Theorem 1.1]{Bismut} that the function
 $$(t,\,x)\in\R_+\times\R^{2d}\mapsto\nabla_xX^{M,\,x}(t)$$
 is a continuous function. Therefore, for any $T\geq0$, there exists a finite random variable $Y$ such that, almost-surely, for all $n\geq0$,
\begin{equation}\label{eq:cv trajec X M x_n}
\sup_{t\in[0,\, T]}\big|X^{M,\,x_{n}}(t)-X^{M,\,x}(t)\big|\leq Y|x_n-x|\;. 
\end{equation}
The inequality above is used several times in the proof of Lemmas~\ref{lem:rapid exit time} and~\ref{lem:subseq x_phi_n}

\begin{proof}[Proof of Lemma~\ref{lem:rapid exit time}]\

\textbf{Step 1}: Fix $T>0$. Let us start with the case $x\in\mathrm{int}_{\partial\Omega}(\mathcal{I})$ and consider $a\in(0,\,T)$. Since $\mathrm{dist}(x,\,\mathcal{R})>0$ and $X^{M,\,x_n}_{\tau^{M,\,x_n}_{\Omega^c}}\in\mathcal{R}$ almost-surely by Remark~\ref{rem:first exit point}, we deduce the existence of a constant $\delta>0$ such that for $n\geq0$ large enough,
$$\mathbb{P}\big(\tau^{M,\,x_n}_{\Omega^c}\leq a\big)\leq\mathbb{P}\bigg(\sup_{t\in[0,\,a]}\big|X^{M,\,x_n}(t)-x_n\big|>\delta\bigg)\;.$$
Furthermore, by~\eqref{eq:cv trajec X M x_n},
\begin{align*}
\sup_{t\in[0,\,a]}\big|X^{M,\,x_n}(t)-x_n\big|&\leq\sup_{t\in[0,\,a]}\big|X^{M,\,x_n}(t)-X^{M,\,x}(t)\big|+\sup_{t\in[0,\,a]}\big|X^{M,\,x}(t)-x\big|+|x_n-x|\\
&\leq (1+Y)|x_n-x|+\sup_{t\in[0,\,a]}\big|X^{M,\,x}(t)-x\big|\;.
\end{align*}
Consequently,
$$\limsup_{n\rightarrow\infty}\mathbb{P}\big(\tau^{M,\,x_n}_{\Omega^c}\leq a\big)\leq\mathbb{P}\bigg(\sup_{t\in[0,\,a]}\big|X^{M,\,x}(t)-x\big|\geq\delta\bigg)$$
which converges to zero when $a\rightarrow0$ given the continuity of the trajectories of $(X^{M,\,x}(t))_{t\geq0}$.

\textbf{Step 2}: Consider now the case $x\in\mathcal{I}\setminus\mathrm{int}_{\partial\Omega}(\mathcal{I})$ such that~\eqref{eq:cone condition} is satisfied. Let us first control the quantity $\big|\mathrm{d}_\Omega(X^{M,\,x_n}(t))-\mathrm{d}_\Omega(X^{M,\,x}(t))\big|$ in order to obtain a suitable lower-bound on the distance $\mathrm{d}_\Omega(X^{M,\,x_n}(t))$ for $t\in[0,\,a]$.

Let 
$$Z:=\sup_{t\in[0,\,T]}|X^{M,\,x}(t)|+Y\;.$$
Let us take $n\geq0$ large enough such that $|x_n-x|\leq1$. The triangle inequality along with~\eqref{eq:cv trajec X M x_n} ensure that
$$\sup_{t\in[0,\, T]}\big|X^{M,\,x_{n}}(t)\big|\leq Z\;.$$
By the Taylor-Lagrange inequality, for all $a\in(0,\,T)$, for all $t\in[0,\,a]$,
\begin{align*}
&\big|\mathrm{d}_\Omega(X^{M,\,x_n}(t))-\mathrm{d}_\Omega(X^{M,\,x}(t))-\langle X^{M,\,x_n}(t)-X^{M,\,x}(t),\,\nabla\mathrm{d}_\Omega(X^{M,\,x}(t))\rangle\big|\\
&\leq\sup_{|z|\leq Z}\big\Vert\nabla^2\mathrm{d}_\Omega(z)\big\Vert\,\big|X^{M,\,x_n}(t)-X^{M,\,x}(t)\big|^2\\
&\leq\sup_{|z|\leq Z}\big\Vert\nabla^2\mathrm{d}_\Omega(z)\big\Vert\,Y^2|x_n-x|^2\;.
\end{align*}
Furthermore,
\begin{align*}
&\big|\langle X^{M,\,x_n}(t)-X^{M,\,x}(t),\,\nabla\mathrm{d}_\Omega(X^{M,\,x}(t))\rangle-\langle x_n-x,\,n(x)\rangle\big|\\
&\leq\big|\langle X^{M,\,x_n}(t)-X^{M,\,x}(t),\,\nabla\mathrm{d}_\Omega(X^{M,\,x}(t))-n(x)\rangle\big|+\big|\langle X^{M,\,x_n}(t)-X^{M,\,x}(t)-(x_n-x),\,n(x)\rangle\big|
\end{align*}
Additionally, by~\eqref{eq:cv trajec X M x_n},
\begin{align*}
&\big|\langle X^{M,\,x_n}(t)-X^{M,\,x}(t),\,\nabla\mathrm{d}_\Omega(X^{M,\,x}(t))-n(x)\rangle\big|\\
&\leq\sup_{t\in[0,\,T]}|X^{M,\,x_n}(t)-X^{M,\,x}(t)|\,\sup_{|z|\leq Z}\big\Vert\nabla^2\mathrm{d}_\Omega(z)\big\Vert\,\sup_{t\in[0,\,a]}|X^{M,\,x}(t)-x|\\
&\leq Y|x_n-x|\,\sup_{|z|\leq Z}\big\Vert\nabla^2\mathrm{d}_\Omega(z)\big\Vert\,\sup_{t\in[0,\,a]}|X^{M,\,x}(t)-x|\;.
\end{align*}
Moreover, given that $|n_p(x)|=0$, using~\eqref{eq:sde} and the Lipschitz continuity of $G_M$, we deduce the existence of a constant $C>0$ such that for all $t\in[0,\,a]$,
\begin{align*}
\big|\langle X^{M,\,x_n}(t)-X^{M,\,x}(t)-(x_n-x),\,n(x)\rangle\big|&=\bigg|\int_0^t\langle G_M(p^{M,\,x_n}(s))-G_M(p^{M,\,x}(s)),\,n_q(x)\rangle\mathrm{d}s\bigg|\\
&\leq CYa|x_n-x|\;.
\end{align*}
Define the process
$$D:a\in[0,\,T]\mapsto Y\sup_{|z|\leq Z}\big\Vert\nabla^2\mathrm{d}_\Omega(z)\big\Vert\,\sup_{t\in[0,\,a]}|X^{M,\,x}(t)-x|+CYa\;.$$
It is clear that the process $(D(a))_{0\leq a\leq T}$ is independent of $n\geq0$ and that $D(a)$ converges to $0$ when $a\rightarrow0$. Furthermore, the previous inequalities ensure that for all $t\in[0,\,a]$, for all $n\geq0$,
$$\mathrm{d}_\Omega(X^{M,\,x_n}(t))\geq\mathrm{d}_\Omega(X^{M,\,x}(t))+\langle x_n-x,\,n(x)\rangle-D(a)|x_n-x|-\sup_{|z|\leq Z}\big\Vert\nabla^2\mathrm{d}_\Omega(z)\big\Vert\,Y^2|x_n-x|^2\;.$$
In particular, on the event $\{\tau^{M,\,x}_{\Omega^c}>a\}$, for all $t\in[0,\,a]$,
\begin{equation}\label{eq:indeq dist M x_n a}
\frac{\mathrm{d}_\Omega(X^{M,\,x_n}(t))}{|x_n-x|}\geq\frac{\langle x_n-x,\,n(x)\rangle}{|x_n-x|}-D(a)-\sup_{|z|\leq Z}\big\Vert\nabla^2\mathrm{d}_\Omega(z)\big\Vert\,Y^2|x_n-x|\;.
\end{equation}
Consequently,
\begin{align*}
\mathbb{P}\big(\tau^{M,\,x_n}_{\Omega^c}\leq a\big)&\leq\mathbb{P}\big(\tau^{M,\,x_n}_{\Omega^c}\leq a,\,\tau^{M,\,x}_{\Omega^c}>a\big)+\mathbb{P}\big(\tau^{M,\,x}_{\Omega^c}\leq a\big)\;.
\end{align*}
Also, by~\eqref{eq:indeq dist M x_n a},
\begin{align*}
\limsup_{n\rightarrow\infty}\mathbb{P}\big(\tau^{M,\,x_n}_{\Omega^c}\leq a,\,\tau^{M,\,x}_{\Omega^c}>a\big)\leq\mathbb{P}\big(D(a)\geq\liminf_{n\rightarrow\infty}\langle x_n-x,\,n(x)\rangle/|x_n-x|\big)\underset{a\rightarrow0}{\longrightarrow}0
\end{align*}
since~\eqref{eq:cone condition} holds and $D(a)\rightarrow0$ when $a\rightarrow0$ almost-surely. Therefore,
$$\lim_{a\rightarrow0}\limsup_{n\rightarrow\infty}\mathbb{P}\big(\tau^{M,\,x_n}_{\Omega^c}\leq a\big)\leq\mathbb{P}\big(\tau^{M,\,x}_{\Omega^c}=0\big)=0$$
since $x\in\mathcal{I}$, hence the proof.
\end{proof}
\begin{proof}[Proof of Lemma~\ref{lem:subseq x_phi_n}]
\textbf{Step 1}: Fix $T\geq0$. Using~\eqref{eq:cv trajec X M x_n}, let us first prove the convergence to zero in probability of the uniform distance between the processes $(X^{x_n}(t))_{t\in[0,\,T]}$ and $(X^{x}(t))_{t\in[0,\,T]}$. For $\epsilon>0$ and $M>0$,
\begin{align*}
&\mathbb{P}\bigg(\sup_{t\in[0,\,T]}\big|X^{x_n}(t)-X^x(t)\big|>\epsilon\bigg)\\
&\leq\mathbb{P}\bigg(\sup_{t\in[0,\,T]}\big|X^{x_n}(t)-X^x(t)\big|>\epsilon,\,\sup_{t\in[0,\,T]}|X^x(t)|\leq M/2,\,\sup_{t\in[0,\,T]}|X^{x_n}(t)|\leq M\bigg)\\
&+\mathbb{P}\bigg(\sup_{t\in[0,\,T]}|X^x(t)|\leq M/2,\,\sup_{t\in[0,\,T]}|X^{x_n}(t)|>M\bigg)+\mathbb{P}\bigg(\sup_{t\in[0,\,T]}|X^x(t)|>M/2\bigg)\;.
\end{align*}
Given that the processes $(X^x(t))_{t\in[0,\,T]}$, $(X^{M,\,x}(t))_{t\in[0,\,T]}$ and $(X^{x_n}(t))_{t\in[0,\,T]}$, $(X^{M,\,x_n}(t))_{t\in[0,\,T]}$ coincide until their first exit time from the ball $\mathrm{B}(0,\,M)$, we deduce that the first probability in the right-hand side of the inequality above admits the following upper-bound
$$\mathbb{P}\bigg(\sup_{t\in[0,\,T]}\big|X^{M,\,x_n}(t)-X^{M,\,x}(t)\big|>\epsilon\bigg)$$
which vanishes when $n\rightarrow\infty$ by~\eqref{eq:cv trajec X M x_n}. Similarly,
\begin{align*}
&\mathbb{P}\bigg(\sup_{t\in[0,\,T]}|X^x(t)|\leq M/2,\,\sup_{t\in[0,\,T]}|X^{x_n}(t)|>M\bigg)\\
&\leq\mathbb{P}\bigg(\sup_{t\in[0,\,T]}\big|X^{M,\,x_n}(t)-X^{M,\,x}(t)\big|>M/2\bigg)\underset{n\rightarrow\infty}{\longrightarrow}0\;.
\end{align*}
All in all,
$$\limsup_{n\rightarrow\infty}\mathbb{P}\bigg(\sup_{t\in[0,\,T]}\big|X^{x_n}(t)-X^x(t)\big|>\epsilon\bigg)\leq\mathbb{P}\bigg(\sup_{t\in[0,\,T]}|X^x(t)|>M/2\bigg)$$
which converges to zero when $M\rightarrow\infty$ given the non-explosion of the process $(X^x(t))_{t\in[0,\,T]}$.

\textbf{Step 2}: Let us now prove the convergence in probability of $\tau^{x_n}_{\Omega^c}$ to $\tau^{x}_{\Omega^c}$. Let us fix $\epsilon>0$.

\textbf{Step 2.a}: Consider first the case $x\in\mathcal{R}$. Assume that $|\tau^{x_n}_{\Omega^c}-\tau^{x}_{\Omega^c}|>\epsilon$. Since $\tau^{x}_{\Omega^c}=0$, we have that
$$\inf_{t\in[0,\,\epsilon]}\mathrm{d}_\Omega(X^{x}(t))<0\;\qquad\text{and}\qquad\inf_{t\in[0,\,\epsilon]}\mathrm{d}_\Omega(X^{x_n}(t))>0\;.$$
Therefore, on the event $\{\tau^{x_n}_{\Omega^c}>\epsilon\}$, the $1$-Lipschitz continuity of the signed distance $\mathrm{d}_\Omega$ ensures that
$$\bigg|\inf_{t\in[0,\,\epsilon]}\mathrm{d}_\Omega(X^{x}(t))\bigg|\leq\sup_{t\in[0,\,\epsilon]}\big(\mathrm{d}_\Omega(X^{x_n}(t))-\mathrm{d}_\Omega(X^{x}(t))\big)\leq\sup_{t\in[0,\,\epsilon]}\big|X^{x_n}(t)-X^x(t)\big|\;.$$
All in all, we deduce from Step 1 that
$$\mathbb{P}(\tau^{x_n}_{\Omega^c}>\epsilon)\leq\mathbb{P}\bigg(\sup_{t\in[0,\,\epsilon]}|X^{x_n}(t)-X^x(t)|\geq\big|\inf_{t\in[0,\,\epsilon]}\mathrm{d}_\Omega(X^{x}(t))\big|\bigg)\underset{n\rightarrow\infty}{\longrightarrow}0$$
since $\big|\inf_{t\in[0,\,\epsilon]}\mathrm{d}_\Omega(X^{x}(t))\big|$ is a positive random variable.

\textbf{Step 2.b}: Consider now $x\in\mathrm{int}_{\partial\Omega}(\mathcal{I})$, or that $x \in \mathcal I$ and~\eqref{eq:cone condition} is satisfied. The proof here is different given that the process $(X^x(t))_{t\geq0}$ does not exit $\Omega$ immediately since $\tau^x_{\Omega^c}>0$.

Let us take $a\in(0,\,T)$. Assume that $|\tau^{x_n}_{\Omega^c}-\tau^x_{\Omega^c}|>\epsilon$ and $\tau^{x_n}_{\Omega^c}>a$. Then, either $a<\tau^{x_n}_{\Omega^c}<\tau^x_{\Omega^c}-\epsilon$, in which case
$$\inf_{t\in[a,\,\tau^{x}_{\Omega^c}-\epsilon]}\mathrm{d}_\Omega(X^{x_n}(t))<0\qquad\text{and}\qquad\inf_{t\in[a,\,\tau^x_{\Omega^c}-\epsilon]}\mathrm{d}_\Omega(X^x(t))>0\;;$$
or $\tau^{x_n}_{\Omega^c}>\tau^x_{\Omega^c}+\epsilon$, in which case
$$\inf_{t\in[0,\,\tau^x_{\Omega^c}+\epsilon]}\mathrm{d}_\Omega(X^x(t))<0\qquad\text{and}\qquad\inf_{t\in[0,\,\tau^{x}_{\Omega^c}+\epsilon]}\mathrm{d}_\Omega(X^{x_n}(t))>0\;.$$
Now let
$$Y':=\min\bigg(\inf_{t\in[a,\,\tau^x_{\Omega^c}-\epsilon]}\mathrm{d}_\Omega(X^x(t)),\,-\inf_{t\in[0,\,\tau^x_{\Omega^c}+\epsilon]}\mathrm{d}_\Omega(X^x(t))\bigg)$$
which is a positive random variable. Given the $1$-Lipschitz continuity of the signed distance $\mathrm{d}_\Omega$, we deduce that on the event $\{|\tau^{x_n}_{\Omega^c}-\tau^x_{\Omega^c}|>\epsilon,\,\tau^x_{\Omega^c}<T,\,\tau^{x_n}_{\Omega^c}>a\}$, 
$$Y'\leq\sup_{t\in[0,\,T+\epsilon]}\big|\mathrm{d}_\Omega(X^{x_n}(t))-\mathrm{d}_\Omega(X^x(t))\big|\leq\sup_{t\in[0,\,T+\epsilon]}\big|X^{x_n}(t)-X^x(t)\big|\;.$$
All in all,
$$\mathbb{P}\big(|\tau^{x_n}_{\Omega^c}-\tau^x_{\Omega^c}|>\epsilon\big)\leq\mathbb{P}\bigg(\sup_{t\in[0,\,T+\epsilon]}\big|X^{x_n}(t)-X^x(t)\big|\geq Y'\bigg)+\mathbb{P}\big(\tau^x_{\Omega^c}\geq T\big)+\mathbb{P}\big(\tau^{x_n}_{\Omega^c}\leq a\big)\;.$$
Therefore, we deduce from Lemma~\ref{lem:rapid exit time} and Step 1 that
$$\lim_{a\rightarrow0}\limsup_{n\rightarrow\infty}\mathbb{P}\big(|\tau^{x_n}_{\Omega^c}-\tau^x_{\Omega^c}|>\epsilon\big)\leq\mathbb{P}\big(\tau^x_{\Omega^c}\geq T\big)$$
which vanishes when $T\rightarrow\infty$.
\end{proof}

\begin{proof}[Proof of Proposition~\ref{prop:boundary continuity h}]
Let $f \in C_b(\Lambda)$ and recall the set $\mathcal{R}_f$ defined in~\eqref{eq:def R_f}. Assume that $x\in\mathcal{R}_f\cup\mathrm{int}_{\partial\Omega}(\mathcal{I})$, or that $x \in \mathcal I$ and~\eqref{eq:cone condition} holds. The convergences in probability in Lemma~\ref{lem:subseq x_phi_n} ensure the almost-sure convergences through an appropriate subsubsequence $(x_{\phi(n)})_{n\geq0}$ of any subsequence of $(x_{n})_{n\geq0}$. Let us now prove that
\begin{equation}\label{eq:lim x_phi_n h}
\lim_{n\rightarrow\infty}h(x_{\phi(n)})=h(x)\;.
\end{equation}
To simplify the notations, we write below
\begin{align*}
\tau^{(n)}:=\tau_{\Omega^c}^{x_{\phi(n)}}\;,\qquad\tau:=\tau_{\Omega^c}^{x}\\
X^{(n)}:=X^{x_{\phi(n)}}\;,\qquad X:=X^x
\end{align*}
For $n\geq0$,
$$h(x_{\phi(n)})=\mathbb{E}\bigg[\mathrm{e}^{\int_0^{\tau^{(n)}}c(X^{(n)}(s))\mathrm{d}s}f(X^{(n)}(\tau^{(n)}))-\int_0^{\tau^{(n)}}\mathrm{e}^{\int_0^sc(X^{(n)}(r))\mathrm{d}r}g(X^{(n)}(s))\mathrm{d}s\bigg]\;.$$
In particular, the following equality holds for all $T\geq0$,
\begin{align*}
h(x_{\phi(n)})&=\mathbb{E}\bigg[\mathbf{1}_{\tau^{(n)}\geq T}\bigg(\mathrm{e}^{\int_0^{\tau^{(n)}}c(X^{(n)}(s))\mathrm{d}s}f(X^{(n)}(\tau^{(n)}))-\int_0^{\tau^{(n)}}\mathrm{e}^{\int_0^sc(X^{(n)}(r))\mathrm{d}r}g(X^{(n)}(s))\mathrm{d}s\bigg)\\
&+\mathbf{1}_{\tau^{(n)}<T}\bigg(\mathrm{e}^{\int_0^{\tau^{(n)}}c(X^{(n)}(s))\mathrm{d}s}f(X^{(n)}(\tau^{(n)}))-\int_0^{\tau^{(n)}}\mathrm{e}^{\int_0^sc(X^{(n)}(r))\mathrm{d}r}g(X^{(n)}(s))\mathrm{d}s\bigg)\bigg]\;.
\end{align*}
On the event $\{\tau^{(n)}<T\}$,
\begin{align*}
\big|X^{(n)}(\tau^{(n)})-X(\tau)\big|&\leq\big|X^{(n)}(\tau^{(n)})-X(\tau^{(n)})\big|+\big|X(\tau^{(n)})-X(\tau)\big|\\
&\leq\sup_{t\in[0,\,T]}\big|X^{(n)}(t)-X(t)\big|+\big|X(\tau^{(n)})-X(\tau)\big|
\end{align*}
which converges to zero almost-surely when $n\rightarrow\infty$. Hence,
$$X^{(n)}(\tau^{(n)})\underset{n\rightarrow\infty}{\longrightarrow}X(\tau)\;.$$
Let us now show that, almost-surely,
\begin{equation}\label{eq:cv f X_tau n}
f(X^{(n)}(\tau^{(n)}))\underset{n\rightarrow\infty}{\longrightarrow}f(X(\tau))\;.
\end{equation}
If $x\in\mathcal{I}$, Theorem~\ref{thm:polarity Gamma+} ensures that $X(\tau)\in\Lambda$ and $X^{(n)}(\tau^{(n)})\in\Lambda$ for all $n\geq0$. Given that $f\in C_b(\Lambda)$, the convergence~\eqref{eq:cv f X_tau n} follows. In the case $x\in\mathcal{R}_f$, $\tau=0$ and therefore $X(\tau)=x$. As a result, since $x\in\mathcal{R}_f$, we obtain~\eqref{eq:cv f X_tau n} as well. All in all, we conclude that, almost-surely,
\begin{align*}
&\mathbf{1}_{\tau^{(n)}<T}\bigg(\mathrm{e}^{\int_0^{\tau^{(n)}}c(X^{(n)}(s))\mathrm{d}s}f(X^{(n)}(\tau^{(n)}))-\int_0^{\tau^{(n)}}\mathrm{e}^{\int_0^sc(X^{(n)}(r))\mathrm{d}r}g(X^{(n)}(s))\mathrm{d}s\bigg)\\
&\underset{n\rightarrow\infty}{\longrightarrow}\mathbf{1}_{\tau<T}\bigg(\mathrm{e}^{\int_0^{\tau}c(X(s))\mathrm{d}s}f(X(\tau))-\int_0^{\tau}\mathrm{e}^{\int_0^sc(X(r))\mathrm{d}r}g(X(s))\mathrm{d}s\bigg)\;.
\end{align*}
Moreover, given the boundedness of the functions $f,\,g,\,c$ and recalling the constant $c_\infty$ in~\eqref{eq:def c_infty}, we also have that
\begin{align*}
&\mathbb{E}\bigg[\mathbf{1}_{\tau^{(n)}\geq T}\bigg(\mathrm{e}^{\int_0^{\tau^{(n)}}c(X^{(n)}(s))\mathrm{d}s}f(X^{(n)}(\tau^{(n)}))-\int_0^{\tau^{(n)}}\mathrm{e}^{\int_0^sc(X^{(n)}(r))\mathrm{d}r}g(X^{(n)}(s))\mathrm{d}s\bigg)\bigg]\\
&\leq||f||_\infty\mathbb{E}\bigg[\mathbf{1}_{\tau^{(n)}\geq T}\bigg(\mathbf{1}_{c_\infty\leq0}+\mathbf{1}_{c_\infty>0}\,\mathrm{e}^{\,c_\infty\tau^{(n)}}\bigg)\bigg]+||g||_\infty\mathbb{E}\bigg[\mathbf{1}_{\tau^{(n)}\geq T}\bigg(\mathbf{1}_{c_\infty\leq0}\,\tau^{(n)}+\mathbf{1}_{c_\infty>0}\,\frac{\mathrm{e}^{\,c_\infty\tau^{(n)}}}{c_\infty}\bigg)\bigg]\;.
\end{align*}
Therefore, taking the limit $T\rightarrow\infty$ and using Assumption~\ref{ass:Omega}, we obtain by the dominated convergence theorem that
$$\limsup_{T\rightarrow\infty}\limsup_{n\rightarrow\infty}\mathbb{E}\bigg[\mathbf{1}_{\tau^{(n)}\geq T}\bigg(\mathrm{e}^{\int_0^{\tau^{(n)}}c(X^{(n)}(s))\mathrm{d}s}f(X^{(n)}(\tau^{(n)}))-\int_0^{\tau^{(n)}}\mathrm{e}^{\int_0^sc(X^{(n)}(r))\mathrm{d}r}g(X^{(n)}(s))\mathrm{d}s\bigg)\bigg]=0\;.$$
Consequently, from any subsequence of $(x_n)_{n\geq0}$, we can extract a subsubsequence $(x_{\phi(n)})_{n\geq0}$ such that $h(x_{\phi(n)})$ converges to $h(x)$ when $n\rightarrow\infty$. This argument ensures that $h(x_n)$ converges to $h(x)$ when $n\rightarrow\infty$, hence the proof.
\end{proof}

\section{Proof of Proposition~\ref{prop:discont set}}\label{sec:counter example}

Let us now prove Proposition~\ref{prop:discont set}. We place ourselves in the two-dimensional setting. The goal is to provide an open set $\Omega\subset\R^{2}$ and a classical solution to~\eqref{eq:edp h} which is not in $C_b(\overline{\Omega})$ and whose discontinuity set has positive Lebesgue surface measure on $\partial\Omega$. The example relies on the construction of a smooth function whose zeros correspond to the Smith-Volterra-Cantor set.

\begin{lem}[Smith-Volterra-Cantor set]\label{lem: fat cantor set} For any $0<a<b$, there exists a non-negative function $U\in C^\infty([a,\,b])$ such that $U^{-1}(0)$ is a closed subset of $[a,\,b]$ with empty interior which admits a positive Lebesgue measure.
\end{lem}
\begin{proof}
Let $0<a<b$. Let $\mathcal{S}$ be the Smith-Volterra-Cantor set on the interval $[a,\,b]$. It is well-known that $\mathcal{S}$ is a closed set with empty interior which admits a positive Lebesgue measure on $[a,\,b]$. By~\cite[Theorem 2.29]{lee2003smooth}, there exists a non-negative function $U\in C^\infty(\R)$ such that $U^{-1}(0)=\mathcal{S}$.
\end{proof}
\begin{proof}[Proof of Proposition~\ref{prop:discont set}] \textbf{Step 1}: We first construct a $C^\infty$ open set $\Omega$ in $\R^2$. Fix $0<a<b$ and let $U\in C^\infty([a,\,b])$ be the function obtained from Lemma~\ref{lem: fat cantor set}. Up to extending $U$ on $\R$ and multiplying it by a function with an appropriate sign, we can assume that $U\in C^\infty(\R)$, $U\geq0$ on $[b,\,+\infty)$ and $U\leq0$ on $(-\infty,\,a]$.

Take constants $m,\,L>0$ satisfying
\begin{equation}\label{eq:upper bound int_a b U}
\int_a^bU(s)\mathrm{d}s<m^2<L^2\;.
\end{equation}
Let us define the open sets
\begin{align*}
\mathcal{A}&:=\big\{(q,\,p)\in\R^2:\,q^2+\int_a^pU(s)\mathrm{d}s<m^2\big\}\;.\\
\mathcal{B}&:=\big\{q\in\R:\,|q|>L\big\}\times\R\;.
\end{align*}
Note that $\mathcal{A}$, $\mathcal{B}$ are disjoint $C^\infty$ open sets. Indeed, the smoothness of $\mathcal{B}$ is immediate. Regarding the open set $\mathcal{A}$, for any $(q,\,p)\in\partial\mathcal{A}$,
\begin{equation}\label{eq:boundary equality energy}
q^2+\int_a^pU(s)\mathrm{d}s=m^2\;.
\end{equation}
Therefore, the outward unitary normal vector from $\mathcal{A}$ at $(q,\,p)$ is given by
\begin{equation}\label{eq:def gradient defining fctn}
n(q,\,p)=\frac{1}{\sqrt{4q^2+U(p)^2}}\begin{pmatrix}2q \\
U(p)
\end{pmatrix}\in\R^2
\end{equation}
which is well defined given that $q=0$ and $p\in U^{-1}(0)$ cannot occur simultaneously by~\eqref{eq:boundary equality energy},~\eqref{eq:upper bound int_a b U} and the fact that $U^{-1}(0)\subset[a,\,b]$. Therefore, we deduce that $\mathcal{A}$ is $C^\infty$ since $U\in C^\infty(\R)$. In particular, the following open set 
$$\Omega:=\R^{2}\setminus\overline{\mathcal{A}\cup\mathcal{B}}$$ is also $C^\infty$. 

\textbf{Step 2}: Let $(X(t)=(q(t),\,p(t)))_{t\geq0}$ satisfying
$$\left\{ \begin{aligned} & \mathrm{d}q(t)=p(t)\mathrm{d}t\;,\\
 & \mathrm{d}p(t)=\mathrm{d}B(t)\;,
\end{aligned}
\right.$$
Take $c,\,g\equiv0$ and notice that Assumption~\ref{ass:coeff} holds. Let us now prove that Assumption~\ref{ass:Omega} also holds. Applying~\cite[Theorem 2.9 and Proposition 2.2]{champagnat2024quasi} to the process $(X(t))_{t\geq0}$ on the bounded-in-position cylinder set $\mathcal{B}^c$, we deduce that $(X(t))_{t\geq0}$ admits a unique quasi-stationary distribution on $\mathcal{B}^c$. Additionally, there exists a bounded function $\varphi$ on $\mathcal{B}^c$ and constants $C,\,\lambda,\,\beta>0$ such that for all $x\notin\mathcal{B}$, for all $t\geq0$,
$$\big|\mathbb{P}_x(\tau_{\mathcal{B}}>t)\mathrm{e}^{\lambda t}-\varphi(x)\big|\leq C\mathrm{e}^{-\beta t}\;.$$
Therefore,
$$\mathbb{P}_x(\tau_{\mathcal{B}}>t)\leq(||\varphi||_\infty+C)\mathrm{e}^{-\lambda t}\;.$$
Hence~\eqref{eq:uniform integ tau} since $\tau_{\Omega^c}\leq\tau_{\mathcal{B}}$ implies that $\mathbb{P}_x(\tau_{\Omega^c}>t)\leq\mathbb{P}_x(\tau_{\mathcal{B}}>t)$.

\textbf{Step 3}: Define the function
\begin{equation}\label{eq:expr hAB}
h_{\mathcal{A},\,\mathcal{B}}:x\in\Omega\mapsto\mathbb{P}_x(\tau_{\mathcal{A}}<\tau_{\mathcal{B}})\;.
\end{equation}
By Theorems~\ref{thm:edp eq pot} and~\ref{thm:existence sol PDE}, $h_{\mathcal{A},\,\mathcal{B}}\in C^\infty(\Omega)$ and satisfies
\begin{equation}\label{eq:edp h_AB}
  \left\{\begin{aligned}
    \mathcal{L}\,h_{\mathcal{A},\,\mathcal{B}}(x) & =0, &&\qquad x\in\Omega\;,\\
    h_{\mathcal{A},\,\mathcal{B}}(x) &=1, && \qquad x\in\Lambda\cap\partial\mathcal{A}\\
    h_{\mathcal{A},\,\mathcal{B}}(x) &=0, && \qquad x\in\Lambda\cap\partial\mathcal{B}\;.
  \end{aligned}\right. 
\end{equation}
Consider the following subset of $\partial\mathcal{A}$
$$\mathcal{D}:=\big\{(q,\,p)\in\partial\mathcal{A}:\,q>0,\;p\in U^{-1}(0)\cap U'^{-1}(0)\big\}\;.$$
Let us prove that $\mathcal{D}$ is a set of discontinuity points for $h_{\mathcal{A},\,\mathcal{B}}$ by showing that any point $x\in\mathcal{D}$ satisfies $h_{\mathcal{A},\,\mathcal{B}}(x)\in[0,\,1)$ and can be approached by a sequence of points $(x_n)_{n\geq0}$ on $\partial\mathcal{A}$ satisfying $h(x_n)=1$ for all $n\geq0$. 

Fix $x=(q,\,p)\in\mathcal{D}$. By~\eqref{eq:def gradient defining fctn},
$$\xi(x)=\frac{2pq}{\sqrt{4q^2+U(p)^2}}>0\;,$$
since $q,\,p>0$. Therefore, $x\in\Gamma^+$ and by Theorem~\ref{thm:reg points},
$$\mathbb{P}_x(\tau_\mathcal{A}>0)=1\;.$$
Hence, for any $\epsilon\in(0,\,1)$, there exists a constant $\delta>0$ such that
\begin{equation}\label{eq:ineq 1 tau_A}
\mathbb{P}_x(\tau_\mathcal{A}\leq\delta)\leq\epsilon/2\;.
\end{equation}
Furthermore, since $U(p)=U'(p)=0$, applying Theorem~\ref{thm:small time asymp distance} to the signed distance function $\mathrm{d}_\Omega$ and using the expression,
$$\lim_{t\rightarrow0}\frac{\mathrm{d}_\Omega(X(t))}{t}=\xi(x)>0.$$
As a result, up to decreasing the constant $\delta$, there exists a compact set $K_\delta\subset\Omega$ such that 
\begin{equation}\label{eq:ineq 2 X_delta K_delta}
\mathbb{P}_x(X(\delta)\notin K_\delta)\leq\epsilon/2\;.
\end{equation}
Let us show that $h_{\mathcal{A},\,\mathcal{B}}(x)\in[0,\,1)$. For any $T\geq0$,
\begin{align*}
1-h_{\mathcal{A},\,\mathcal{B}}(x)&=\mathbb{P}_x\big(\tau_{\mathcal{B}}<\tau_{\mathcal{A}}\big)\\
&\geq\mathbb{P}_x\big(\tau_\mathcal{A}>T+\delta,\,X(\delta)\in K_\delta,\,X(T+\delta)\in\mathcal{B}\big)\\
&=\mathbb{E}_x\big[\mathbf{1}_{\tau_\mathcal{A}>\delta,\,X(\delta)\in K_\delta}\mathbb{P}_{X(\delta)}\big(X(T)\in\mathcal{B},\,\tau_\mathcal{A}>T\big)\big]\\
&\geq\mathbb{P}_x(\tau_\mathcal{A}>\delta,\,X(\delta)\in K_\delta)\;\inf_{y\in K_\delta}\mathbb{P}_y(X(T)\in\mathcal{B},\,\tau_\mathcal{A}>T)\;.
\end{align*}
Consequently, by~\eqref{eq:ineq 1 tau_A} and~\eqref{eq:ineq 2 X_delta K_delta},
$$1-h_{\mathcal{A},\,\mathcal{B}}(x)\geq(1-\epsilon)\inf_{y\in K_\delta}\mathbb{P}_y(X(T)\in\mathcal{B},\,\tau_\mathcal{A}>T)\;.$$
The positivity of the infimum above follows from standard controllability arguments which can be found for instance in~\cite[Lemma 4.10]{cutoff_Langevin}. As a result, $h_{\mathcal{A},\,\mathcal{B}}(x)\in[0,\,1)$. 

Let us now construct a sequence of points $(x_n)_{n\geq0}$ on $\partial\mathcal{A}$ which converge to $x=(q,\,p)$ and such that $h_{\mathcal{A},\,\mathcal{B}}(x_n)=1$. Since $U^{-1}(0)$ is a set with empty interior, there exists a sequence $(p_n)_{n\geq0}$ in $[a,\,b]$ converging to $p$ such that for all $n\geq0$, $U(p_n)>0$. Define the sequence
$$\forall n\geq0\;,\qquad x_n:=\bigg(\sqrt{m^2-\int_a^{p_n}U(s)\mathrm{d}s},\,p_n\bigg)\in\partial\mathcal{A}\;.$$
We deduce from~\eqref{eq:def gradient defining fctn} that $x_n\in\mathrm{T}\cap\partial\mathcal{A}$ since $U(p_n)>0$. Therefore, by Theorem~\ref{thm:reg points}, $\mathbb{P}_{x_n}(\tau_\mathcal{A}=0)=1$ which yields $h_{\mathcal{A},\,\mathcal{B}}(x_n)=1$ by the definition of $h_{\mathcal{A},\,\mathcal{B}}$ , hence the discontinuity of $h_{\mathcal{A},\,\mathcal{B}}$ at $x$.

\textbf{Step 3}: To conclude the proof it remains to show that the surface measure $\sigma_1(\mathcal{D})$ is positive. Parametrizing the boundary set $\mathcal{D}$ in the $p$-coordinates using~\eqref{eq:boundary equality energy}, we obtain
\begin{align*}
\sigma_1(\mathcal{D})&=\int_{U^{-1}(0)\cap U'^{-1}(0)}\sqrt{1+\frac{U(p)^2/4}{m^2-\int_a^p U(s)\mathrm{d}s}}\mathrm{d}p\\
&=\int_{U^{-1}(0)\cap U'^{-1}(0)}\mathrm{d}p\;.
\end{align*}
Furthermore,
\begin{align*}
\sigma_1\big(U^{-1}(0)\cap U'^{-1}(0)\big)&=\sigma_1\big(U^{-1}(0)\big)-\sigma_1\big(U^{-1}(0)\cap (U'^{-1}(0))^c\big)\\
&=\sigma_1\big(U^{-1}(0)\big)
\end{align*}
since the set $U^{-1}(0)\cap (U'^{-1}(0))^c$ is composed of isolated points in $[a,\,b]$ and therefore has zero Lebesgue measure. The positivity of $\sigma_1(U^{-1}(0))$ ensured by Lemma~\ref{lem: fat cantor set} concludes the proof.
\end{proof}

\textbf{Acknowledgements}: The work of M. Ramil was partially supported by ANR-25-CE40-6875-01 (DySLoS).

\section{Appendix}
 
\begin{lem}[Hypoellipticity]\label{lem: hypoellip} The operator $\mathcal{L}$ is hypoelliptic.    
\end{lem}

\begin{proof}[Proof of Lemma~\ref{lem: hypoellip}]
We aim to use H\"{o}rmander's theorem in~\cite{HM} which provides a sufficient condition ensuring the hypoellipticity. To this extent, we notice that the operator $\mathcal{L}$ can be written in the following form
$$\mathcal{L}=Y_0+\sum_{i=1}^dY_i^2\;,$$
where for every $i\in\llbracket1,\,d\rrbracket$,
$$Y_i=\sum_{k=1}^d\Sigma_{k,\,i}\,\partial_{p_k}$$
and
$$Y_0=\sum_{j=1}^dG_j\,\partial_{q_j}+\sum_{j=1}^dF_j\,\partial_{p_j}-\sum_{j,\,k=1}^d\Sigma_{k,\,i}\,\partial_{p_k}\Sigma_{j,\,i}\,\partial_{p_j}\;.$$
H\"{o}rmander's theorem ensures that if the Lie algebra generated by $\{Y_0,\,Y_1,\ldots,\,Y_d\}$ spans the derivatives in all directions then the operator $\mathcal{L}$ is hypoelliptic. Let us first notice that since the matrix $\Sigma$ is invertible,
\begin{equation}\label{eq:span Y_i velocities}
\mathrm{Span}\{Y_1,\ldots,\,Y_d\}=\mathrm{Span}\{\partial_{p_1},\ldots,\,\partial_{p_d}\}\;.
\end{equation}
It remains to generate the partial derivatives in the $q$-coordinates. In order to do that let us consider for $i\in\llbracket1,\,d\rrbracket$ the following Lie bracket
$$[Y_i,\,Y_0]=Y_iY_0-Y_0Y_i\;.$$
Easy computations show the existence of smooth functions $(H_j)_{1\leq j\leq d}$ on $\Omega$ such that
\begin{align*}
[Y_i,\,Y_0]&=\sum_{k,\,j=1}^d\Sigma_{i,\,k}\,\partial_{p_k}G_j\,\partial_{q_j}+\sum_{j=1}^dH_j\,\partial_{p_j}\\
&=\sum_{j=1}^d(\Sigma\,DG^\dagger)_{i,\,j}\partial_{q_j}+\sum_{j=1}^dH_j\,\partial_{p_j}\;.
\end{align*}
Since the matrix $\Sigma\,DG^\dagger$ is invertible we deduce that the first summation spans all the partial derivatives in the $q$-coordinates while the second sum involves only partial derivatives in the $p$-coordinates. Therefore, it follows from~\eqref{eq:span Y_i velocities} that
$$\mathrm{Span}\{Y_1,\ldots,\,Y_d,\,[Y_1,\,Y_0],\ldots,\,[Y_d,\,Y_0]\}=\mathrm{Span}\{\partial_{q_1},\ldots,\,\partial_{q_d},\,\partial_{p_1},\ldots,\,\partial_{p_d}\}\;.$$
which concludes the proof.
\end{proof}
\bibliographystyle{plain}
\bibliography{biblio}

@article{derridj1971probleme,
  title={Un probl{\`e}me aux limites pour une classe d'op{\'e}rateurs du second ordre hypoelliptiques},
  author={Derridj, M.},
  booktitle={Annales de l'institut Fourier},
  volume={21},
  number={4},
  pages={99--148},
  year={1971}
}

@article{bernard2024existence,
  title={On the existence and asymptotic behavior of weak solution of the kinetic {F}okker-{P}lanck equation in a bounded domain with absorbing boundary},
  author={Bernard, E.},
  journal={Bulletin des Sciences Math{\'e}matiques},
  volume={197},
  pages={103523},
  year={2024},
  publisher={Elsevier}
}

@article{avelin2026weak,
  title={Weak and {P}erron Solutions for {S}tationary {K}ramers-{F}okker-{P}lanck Equations in {B}ounded {D}omains},
  author={Avelin, B. and Hou, M.},
  journal={Potential Analysis},
  volume={64},
  number={1},
  pages={11},
  year={2026},
  publisher={Springer}
}

@article{garain2022regularity,
  title={On regularity and existence of weak solutions to nonlinear {K}olmogorov-{F}okker-{P}lanck type equations with rough coefficients},
  author={Garain, P. and Nystr{\"o}m, K.},
  journal={arXiv preprint arXiv:2204.12277},
  year={2022}
}

@article{sanchez2023krein,
  title={On the {K}rein-{R}utman theorem and beyond},
  author={Sanchez, C. F. and Gabriel, P. and Mischler, S.},
  journal={arXiv:2305.06652},
  year={2023}
}

@article{kim2026sharp,
  title={Sharp regularity near the grazing set for kinetic {F}okker-{P}lanck equations},
  author={Kim, K. and Weidner, M.},
  journal={arXiv preprint arXiv:2603.04121},
  year={2026}
}

@article{ramil2022explicit,
  title={Explicit and sharp two-sided estimates for the killed {L}angevin process},
  author={Ramil, M.},
  journal={arXiv preprint arXiv:2209.02378},
  year={2022}
}

@article{zhu2024regularity,
  title={Regularity of kinetic {F}okker--{P}lanck equations in bounded domains},
  author={Zhu, Y.},
  journal={Annales Henri Lebesgue},
  volume={7},
  pages={1323--1366},
  year={2024}
}

@inproceedings{bony1969principe,
  title={Principe du maximum, in{\'e}galit{\'e} de {H}arnack et unicit{\'e} du probleme de {C}auchy pour les op{\'e}rateurs elliptiques d{\'e}g{\'e}n{\'e}r{\'e}s},
  author={Bony, J.-M.},
  booktitle={Annales de l'institut Fourier},
  volume={19},
  number={1},
  pages={277--304},
  year={1969}
}

@incollection{lee2003smooth,
  title={Smooth manifolds},
  author={Lee, J. M.},
  booktitle={Introduction to smooth manifolds},
  pages={1--29},
  year={2003},
  publisher={Springer}
}

@article{Bismut,
  title={A generalized formula of {I}t{\^o} and some other properties of stochastic flows},
  author={Bismut, J.-M.},
  journal={Zeitschrift f{\"u}r Wahrscheinlichkeitstheorie und verwandte Gebiete},
  volume={55},
  number={3},
  pages={331--350},
  year={1981},
  publisher={Springer}
}

@inproceedings{cattiaux1992stochastic,
  title={Stochastic calculus and degenerate boundary value problems},
  author={Cattiaux, P.},
  booktitle={Annales de l'institut Fourier},
  volume={42},
  number={3},
  pages={541--624},
  year={1992}
}

@inproceedings{dvoretzky1951some,
  title={Some problems on random walk in space},
  author={Dvoretzky, A. and Erd{\"o}s, P.},
  booktitle={Proc. Second Berkeley Symp. Math. Statist. Probab},
  pages={353--367},
  year={1951}
}

@article{lachal1997local,
  title={Local asymptotic classes for the successive primitives of {B}rownian motion},
  author={Lachal, A.},
  journal={The Annals of Probability},
  pages={1712--1734},
  year={1997},
  publisher={JSTOR}
}

@article{gantert1993inversion,
  title={An inversion of {S}trassen's law of the iterated logarithm for small time},
  author={Gantert, N.},
  journal={The Annals of Probability},
  volume={21},
  number={2},
  pages={1045--1049},
  year={1993},
  publisher={Institute of Mathematical Statistics}
}

@article{EK,
  title={Eyring-{K}ramers law for the underdamped {L}angevin process},
  author={Lee, S. and Ramil, M. and Seo, I.},
  journal={arXiv preprint arXiv:2503.12610},
  year={2025}
}

@book{maggi2012sets,
  title={Sets of finite perimeter and geometric variational problems: an introduction to {G}eometric Measure Theory},
  author={Maggi, F.},
  volume={135},
  year={2012},
  publisher={Cambridge University Press}
}

@book {F,
    AUTHOR = {Friedman, A.},
     TITLE = {Stochastic differential equations and applications. {V}ol. 1},
      NOTE = {Probability and Mathematical Statistics, Vol. 28},
 PUBLISHER = {Academic Press, New
              York-London},
      YEAR = {1975},   
}

@article{spitzer1958some,
  title={Some theorems concerning 2-dimensional {B}rownian motion},
  author={Spitzer, F.},
  journal={Trans. Amer. Math. Soc},
  volume={87},
  number={1},
  pages={187--197},
  year={1958},
  publisher={Springer}
}

@article{palais1960,
  title={Extending diffeomorphisms},
  author={Palais, R. S.},
  journal={Proceedings of the American Mathematical Society},
  volume={11},
  number={2},
  pages={274--277},
  year={1960},
  publisher={JSTOR}
}

@article {Menozzi,
    AUTHOR = {Konakov, V. and Menozzi, S. and Molchanov,
              S.},
     TITLE = {Explicit parametrix and local limit theorems for some
              degenerate diffusion processes},
   JOURNAL = {Ann. Inst. Henri Poincar\'{e} Probab. Stat.},
  FJOURNAL = {Annales de l'Institut Henri Poincar\'{e} Probabilit\'{e}s et
              Statistiques},
    VOLUME = {46},
      YEAR = {2010},
    NUMBER = {4},
     PAGES = {908---923}, 
}

@article {WeylTube,
    AUTHOR = {Weyl, H.},
     TITLE = {On the {V}olume of {T}ubes},
   JOURNAL = {Amer. J. Math.},
  FJOURNAL = {American Journal of Mathematics},
    VOLUME = {61},
      YEAR = {1939},
    NUMBER = {2},
     PAGES = {461--472}, 
}

@article{Mischler,
  title={The {K}inetic {F}okker--{P}lanck equation in a domain: {U}ltracontractivity, hypocoercivity, and long-time asymptotic behavior},
  author={Carrapatoso, K. and Mischler, S.},
  journal={Rendiconti Lincei},
  volume={35},
  number={4},
  pages={643--680},
  year={2025}
}

@article {Vel,
    AUTHOR = {Hwang, H. J. and Jang, J. and Vel\'{a}zquez, J. J. L.},
     TITLE = {The {F}okker-{P}lanck equation with absorbing boundary
              conditions},
   JOURNAL = {Arch. Ration. Mech. Anal.},
  FJOURNAL = {Archive for Rational Mechanics and Analysis},
    VOLUME = {214},
      YEAR = {2014},
    NUMBER = {1},
     PAGES = {183--233}, 
}

@article {HM,
    AUTHOR = {H\"{o}rmander, L.},
     TITLE = {Hypoelliptic second order differential equations},
   JOURNAL = {Acta Math.},
  FJOURNAL = {Acta Mathematica},
    VOLUME = {119},
      YEAR = {1967},
     PAGES = {147--171}, 
}

@book{fichera1959unified,
  title={On a unified theory of boundary value problems for elliptic-parabolic equations of second order},
  author={Fichera, G.},
  year={1959},
  publisher={Mathematics Research Center, United States Army, University of Wisconsin New}
}

@article{cutoff_Langevin,
  title={Asymptotic stability and cut-off phenomenon for the underdamped {L}angevin dynamics},
  author={Lee, S. and Ramil, M. and Seo, I.},
  journal={Accepted in Annals of Applied Probability},
  year={2023}
}

@inproceedings{lelievre2024estimation,
  title={Estimation of statistics of transitions and {H}ill relation for {L}angevin dynamics},
  author={Leli{\`e}vre, T. and Ramil, M. and Reygner, J.},
  booktitle={Annales de l'Institut Henri Poincar{\'e} (B) Probabilit{\'e}s et Statistiques},
  volume={60},
  number={3},
  pages={1645--1683},
  year={2024},
  organization={Institut Henri Poincar{\'e}}
}

@article{champagnat2024quasi,
  title={Quasi-stationary distribution for kinetic {SDE}s with low regularity coefficients},
  author={Champagnat, N. and Leli{\`e}vre, T. and Ramil, M. and Reygner, J. and Villemonais, D.},
  journal={arXiv preprint arXiv:2410.01042},
  year={2024}
}

@book {GT,
    AUTHOR = {Gilbarg, D. and Trudinger, N. S.},
     TITLE = {Elliptic partial differential equations of second order},
    SERIES = {Classics in Mathematics},
      NOTE = {Reprint of the 1998 edition},
 PUBLISHER = {Springer-Verlag, Berlin},
      YEAR = {2001}, 
}

@article{LelRamRey,
  title={A probabilistic study of the kinetic {F}okker-{P}lanck equation in cylindrical domains},
  author={Leli{\`e}vre, T. and Ramil, M. and Reygner, J.},
  journal={Journal of Evolution Equations},
volume = 22,
number= 38,
  year={2022}
}
 
\end{document}